\documentclass[11pt]{article}%
\usepackage{amsmath}
\usepackage{xcolor}
\usepackage{amsfonts}
\usepackage{amssymb}
\usepackage{graphicx}
\usepackage[top=1cm, bottom=3cm, left=2cm, right=2cm, headheight=0.5cm,
footskip=60pt]{geometry}%
\usepackage{hyperref}
\providecommand{\U}[1]{\protect\rule{.1in}{.1in}}

\newtheorem{theorem}{Theorem}

\newtheorem{lemma}{Lemma}

\newtheorem{proposition}{Proposition}
\newtheorem{remark}{Remark}

\numberwithin{equation}{section}
\numberwithin{theorem}{section}
\numberwithin{remark}{section}
\numberwithin{proposition}{section}
\numberwithin{lemma}{section}
\newenvironment{proof}[1][Proof]{\noindent\textbf{#1.} }{\ \rule{0.5em}{0.5em}}
\allowdisplaybreaks
\def\div{ \hbox{\rm div}\,  }
\def\ddj{\dot{\Delta}_{j}}
\begin{document}

\title{Relaxation limit for a damped one-velocity Baer-Nunziato model to a Kappila model}
\author{Burtea Cosmin, Crin-Barat Timoth\'{e}e, Tan Jin}
\date{}
\maketitle

\begin{abstract}
 In this paper we study a singular limit problem in the context of partially dissipative first order quasilinear systems. This problem arises in multiphase fluid mechanics. More precisely, taking into account dissipative effects for the velocity, we show that the so-called Kapilla system is obtained as a relaxation limit from the Baer-Nunziato (BN) system and derive the convergence rate of this process. 
 The main problem we encounter is that the (BN)-system does not verify the celebrated (SK) condition due to Shizuta and Kawashima. It turns out that we can rewrite the (BN)-system in terms of new variables such as to highlight a subsystem for which the linearized does verify the (SK) condition
 which is coupled through lower-order terms with a transport equation.    
 We construct an appropriate weighted energy-functional which allows us to tackle the lack of symmetry of the system, provides decay information and allows us to close the estimate uniformly with respect to the relaxation parameter.

\end{abstract}

\section{Introduction}
\subsection{Motivations}
Multiphase flows are ubiquitous in real world applications ranging from
engineering to biological systems. The term multiphase includes flows
that are topologically very different. As it is explained in \cite{Ishii}, we
distinguish mixtures with separated phases flows (film flows, jet
flows), mixed or transitional phase flows (gas pockets in liquids) and dispersed
phase flows (bubbly flows, sprays). Understanding the mathematical qualitative
properties of the different governing models is important, for instance, in
order to construct more pertinent numerical schemes which would increase the
predictive power of these models. ~\ 

In order to describe \textit{dispersed two-phase flows}, besides the classical
variables like densities, velocities and state laws, one need to introduce two extra
ones called volume fractions, measuring \textit{how much space} does one phase occupy
at a given position in space. It goes without saying that in order to have a
closed system, besides the equations expressing conservation of mass and
momentum, two extra equations are needed. 

In \cite{Baer_Nunziato_1986}, Baer and Nunziato proposed a model for which the
volume fractions verify%
\begin{equation}
\partial_{t}\alpha_{\pm}+v_{I}\cdot\nabla\alpha_{\pm}=\dfrac{P_{\pm}-P_{\mp}%
}{\varepsilon}, \label{BaerNunziato_closure}%
\end{equation}
where $\varepsilon$ can be seen as a time relaxation parameter and, in practice, is chosen to
be very small. The unknown $v_{I}$ is interpreted as an interface velocity
which depends on the densities, volumes fractions and phase velocities. In
\cite{Baer_Nunziato_1986}, the interface velocity coincides with one of the
phase velocities but different choices have also been used in the literature, see
for instance \cite{Herard_2005} and \cite{gav2003JFM}. 

From the Baer-Nunziato model, Kapilla et al. \cite{kapila}
proposed the following equations
\begin{equation}
\left\{
\begin{array}
[c]{l}%
\alpha_{+}+\alpha_{-}=1,\\
P_{+}=P_{-}.
\end{array}
\right.  \label{Kapilla}%
\end{equation}

Of course, one might wonder what is the link between $\left(
\text{\ref{BaerNunziato_closure}}\right)  $ and $\left(  \text{\ref{Kapilla}%
}\right)  $. First of all, from $\left(  \text{\ref{BaerNunziato_closure}%
}\right)  $ one immediately infers that%
\[
\partial_{t}\left(  \alpha_{+}+\alpha_{-}\right)  +v_{I}\cdot\nabla(\alpha
_{+}+\alpha_{-})=0.
\]
Hence
\[
\alpha_{+}+\alpha_{-}=1
\]
is recovered provided it is true initially. Moreover, one expects to obtain the second
equation of $\left(  \text{\ref{Kapilla}}\right)  $ in the limit
$\varepsilon\rightarrow0$.

The same closure equations \eqref{Kapilla} are obtained using the so-called averaging methods
\cite{Ishii} or variational methods \cite{Gavrilyuk_2011,BurGavPer}. In this framework, one can interpret the Baer-Nunziato model as a relaxation model for the Kapilla model. 

Let us mention that in the mathematical community \cite{BrHi1, BrDeGhGrHi}, equations  $\left(
\text{\ref{BaerNunziato_closure}}\right)  $ are refereed to as PDEs closure
laws while $\left(  \text{\ref{Kapilla}}\right)  $ are called algebraic
closure laws.

Multiphase models received a lot of attention from the mathematical community
recently. In the context of weak solutions we mention the results of
Novotny \cite{Nov2020}, Novotny and Pokorny \cite{NovPok2020}, Bresch,
Mucha and Zatorska \cite{BreMucZat2019}, Vasseur, Wen and Yu
\cite{VasWenYu2019}. About applications in biology see Gwiazda et al
\cite{GwiPerSwi2019} and Debiec et al \cite{DkePerSchVau2021}. We shall bring up that in all the above
papers, viscosity plays a crucial role.

In \cite{BrBurLa}, constructing upon previous works \cite{BrHu, BrHi1, BrHi2}, the authors showed that it is possible to obtain a viscous Baer-Nunziato system following a
homogenization procedure. The basic assumption is that if we zoom in the
mixture, we arrive at a mesoscale where the two phases are separated. Assuming
that each phase verifies the Navier-Stokes equations in their own domain, the
authors were able to write a closed system for the mixture. When going back to the macroscopic scale, loosely speaking, the density of a fluid mixture is assumed to wildly oscillate
between two reference densities. The propagation of oscillations is quantified
through Young measures and it is shown that if these measures are convex
combinations of Dirac masses at initial time, then this structure is preserved
for later times. \textit{At this point, it is important to note that in these
papers, the authors obtained equations for the volume fractions }$\alpha_{\pm}%
$\textit{ of the form }$\left(  \text{\ref{BaerNunziato_closure}}\right)
$\textit{ while the relaxation time }$\varepsilon$\textit{ is proportional to the mean
viscosity of the two phases.}

In this paper, we justify rigorously that the solutions of a one-velocity Baer-Nunziato model tend, when the relaxation parameter $\varepsilon$ goes to $0$, to solutions of a multiphase fluid system with algebraic closure laws $\left(  \text{\ref{Kapilla}}\right) $. Moreover, we derive an explicit convergence rate of the relaxation process.

\subsection{Presentation of the models and main results}

We consider a mixture of two compressible fluids filling the ambient space
$\mathbb{R}^{d}$ with $d\geq 2$. The characteristic state function of the two phases will be
denoted separately by $+$, $-$. We suppose that the flow of the mixture is
animated by a single velocity vector field:
\[
u:\mathbb{R}_{+}\times\mathbb{R}^{d}\rightarrow\mathbb{R}^{d}.
\]
We denote the two mass densities of the phases $\pm$ by
\[
\rho_{\pm}:\mathbb{R}_{+}\times\mathbb{R}^{d}\rightarrow\mathbb{R}_{+}%
\]
and we introduce the volume fractions of the fluid $\pm$
\[
\alpha_{\pm}:\mathbb{R}_{+}\times\mathbb{R}^{d}\rightarrow\lbrack0,1].
\]

The multidimensional version of the system obtained in \cite{BrBurLa} reads:
\begin{equation}
\left\{
\begin{array}
[c]{l}%
\partial_{t}\alpha_{\pm}+u\cdot\nabla\alpha_{\pm}=\pm\dfrac{\alpha_{+}%
\alpha_{-}}{2\mu+\lambda}\left(  P_{+}\left(  \rho_{+}\right)  -P_{-}\left(
\rho_{-}\right)  \right)  ,\\
\partial_{t}\left(  \alpha_{\pm}\rho_{\pm}\right)  +\operatorname{div}\left(
\alpha_{\pm}\rho_{\pm}u\right)  =0,\\
\partial_{t}(\rho u)+\operatorname{div}(\rho u\otimes u)-\mathcal{A}%
_{\mu,\lambda}u+\nabla P+\eta\rho u=0,\\
\rho=\alpha_{+}\rho_{+}+\alpha_{-}\rho_{-},\\
P=\alpha_+ P_{+}\left(  \rho_{+}\right)  +\alpha_{-}P_{-}\left(  \rho
_{-}\right)  
\end{array}
\right.  \tag{$BN$}\label{DBN}%
\end{equation}
where we added a damping term in the equation of the velocity with a parameter $\eta\geq1$. This terms models elastic-type drag forces slowing down the fluid and it is a crucial   in our mathematical analysis since it allows to use techniques coming from the theory of partially dissipative hyperbolic systems. The relevant  model   with common pressure 
is obtained in \cite{BrHi2}. 
Above, $\mathcal{A}_{\mu,\lambda}$ stands for the Lam\'{e} operator:%
\[
\mathcal{A}_{\mu,\lambda}u=\mu\Delta u+\left(  \mu+\lambda\right)
\nabla\operatorname{div}u
\]
with $\mu,$ $\lambda$ given constants verifying
\[
\mu\geq0, ~\lambda+\mu\geq0 \quad\text{and}\quad \nu=\lambda+2\mu\leq1.
\]
The functions $P_{+}$ and $P_{-}$ model the internal barotropic pressures for each
fluid. We will assume that they take the following explicit form
\begin{equation}
P_{\pm}\left(  s\right)  =A_{\pm}s^{\gamma_{\pm}}\quad \text{ for all }s\geq0,
\label{pressures_def}%
\end{equation}
where $\gamma_{\pm}\geq1$, $A_{\pm}>0$ are given constants. Moreover, without
loss of generality, we will suppose that%
\begin{equation}
\gamma_{+}>\gamma_{-}. \label{gamma_+>gamma_-}%
\end{equation}
The density and pressure of the mixture are denoted by
\[
\rho=\alpha_{+}\rho_{+}+\alpha_{-}\rho_{-}\quad \text{ and }\quad P=\alpha_{+}P_{+}\left(
\rho_{+}\right)  +\alpha_{-}P_{-}\left(  \rho_{-}\right)  .
\]
We are concerned with solutions that satisfy
\begin{equation}
\alpha_{\pm}\left(  t,x\right)  \rightarrow\bar{\alpha}_{\pm},~~\rho_{\pm
}\left(  t,x\right)  \rightarrow\bar{\rho}_{\pm},~~u\left(  t,x\right)
\rightarrow0_{\mathbb{R}^d}~~\text{ as }\left\vert x\right\vert
\rightarrow\infty, \label{constants_at_infinity_1}%
\end{equation}
where $0<\bar{\alpha}_{\pm}\leq1, 0<\bar{\rho}_{\pm}$ are given constants and%
\begin{equation}
\bar{\alpha}_{+}+\bar{\alpha}_{-}=1. \label{sum=1}%
\end{equation}
We denote by $\left(  \alpha_{\pm0},\rho_{\pm0},u_{0}\right)  $ the initial
condition:
\begin{equation}
\alpha_{\pm}(t,x)|_{t=0}=\alpha_{\pm0},~~\rho_{\pm}(t,x)|_{t=0}=\rho_{\pm
0},~~u(t,x)|_{t=0}=u_{0},\quad x\in\mathbb{R}^{d}. \label{initial-data-S1}%
\end{equation}

Notice that when $\nu:=2\mu+\lambda$ tends to $0$, System $\eqref{DBN}$
formally converges to the following system:
\begin{equation}
\left\{
\begin{array}
[c]{l}%
\alpha_{+}+\alpha_{-}=1,\\
\partial_{t}\left(  \alpha_{\pm}\rho_{\pm}\right)  +\operatorname{div}\left(
\alpha_{\pm}\rho_{\pm}u\right)  =0,\\
\partial_{t}(\rho u)+\operatorname{div}(\rho u\otimes u)+\nabla P+\eta\rho
u=0,\\
\rho=\alpha_{+}\rho_{+}+\alpha_{-}\rho_{-},\\
P=P_{+}\left(  \rho_{+}\right)  =P_{-}\left(  \rho_{-}\right)  .
\end{array}
\right.  \tag{$K$}\label{K}%
\end{equation}
\ Our main goal is to justify the relaxation/inviscid limit on
a solid mathematical background, at least when the initial data $(\alpha
_{\pm0},\rho_{\pm0},u_{0})$ are close to the constant equilibrium $(\bar{\alpha
}_{\pm},\bar{\rho}_{\pm},0)$. Inasmuch as we expect the limiting pressures to
agree in the vanishing viscosity limit, and in order to avoid initial time
layers, we will suppose that the pressures are at equilibrium at infinity:%
\begin{equation}
P_{+}\left(  \bar{\rho}_{+}\right)  =P_{-}\left(  \bar{\rho}_{-}\right)
\overset{not.}{=}\bar{P}. \label{constants_at_infinity_2}%
\end{equation}
 Throughout the paper, $C$ stands for a “harmless” constant.
We are now in the position of stating our main results. First we state our uniform (with respect to  $\mu$ and  $\lambda$) global existence result for the System \eqref{DBN}.
\begin{theorem}\label{Th-2}
Let $d\geq 2$ and assume that the parameters satisfy $\mu\geq0, \lambda+\mu\geq 0, \nu\leq 1$ and  $\eta\geq1$.
 Let 
constants $ \bar{\alpha}_{\pm}\in\left(  0,1\right)  ,\bar{\rho
}_{\pm}>0$ satisfy \eqref{sum=1} and  \eqref{constants_at_infinity_2}.
There exists a constant $c_1>0$ independent of the viscosity coefficients $\mu,\lambda$ such
that for any initial data $(\alpha_{+0}, \alpha_{-0}, \rho_{+0}, \rho_{-0}, u_0 ) $
verifying%
\[
\left\Vert (\alpha_{\pm0}-\bar{\alpha}_{\pm}, \rho_{\pm0}-\bar{\rho}_{\pm}, u_0)\right\Vert _{B^{\frac{d}{2}-1}\cap{B^{\frac
{d}{2}+1}}}\leq c_1,
\]
then System \eqref{DBN} admits a unique global-in-time solution
$(\alpha_{+}, \alpha_{-},  \rho_{+}, \rho_{-}, u)$ such
that
\[
\left\{
\begin{array}
[c]{l}%
(\alpha_{\pm} -\bar{\alpha}_{\pm}, \rho_{\pm} -\bar{\rho}_{\pm}, u)\in\mathcal{C}_{b}(\mathbb{R}_{+};  
{B}^{\frac{d}{2}-1}\cap B^{\frac{d}{2}+1}),\;\\
\dfrac{P_{+}\left(  \rho_{+}\right)
-P_{-}\left(  \rho_{-}\right) }{2\mu+\lambda}\in L^1(\mathbb{R}_+; B^{\frac{d}{2}-1}\cap B^{\frac{d}{2}}) \quad \text{and} \quad
u\in L^1(\mathbb{R}_+; B^{\frac{d}{2}}\cap B^{\frac{d}{2}+1}).
\end{array}
\right.  \;
\]
Moreover, the following estimate holds true uniformly with respect to the viscosity coefficients $\mu$ and $\lambda$:
\begin{multline*}
\left\Vert \left( \alpha_{\pm}-\bar{\alpha
}_{\pm},\rho_{\pm} -\bar{\rho}_{\pm}, u\right)  \right\Vert
_{ {L}^{\infty}(\mathbb{R}_+;      B^{\frac{d}{2}-1}\cap B
^{\frac{d}{2}+1})}+\left\Vert u  \right\Vert _{L^{1}(\mathbb{R}_+; B^{\frac{d}{2}}%
\cap B^{\frac{d}{2}+1})}\\
\dfrac{1}{2\mu+\lambda}\|{P_{+}\left(  \rho_{+}\right)
-P_{-}\left(  \rho_{-}\right) }\|_{L^1(\mathbb{R}_+; B^{\frac{d}{2}-1}\cap  B^{\frac{d}{2}})}\leq Cc_1,
\end{multline*}
with a universal constant $C>0.$
\end{theorem}
\begin{remark}
The same result is valid for the quasilinear first order system associated to \eqref{DBN}, in other words, the viscosity plays no role in the mathematical analysis, the same result is valid  if  $\mathcal{A}_{\mu,\lambda}\equiv0$. 
\end{remark}

As a consequence of relaxation limit arguments, we obtain the following theorem regarding System \eqref{K}.

\begin{theorem}\label{Th-1}
 Let $d\geq 2$ and $\eta\geq1$. Let constants $ \bar{\alpha}_{\pm}\in\left(  0,1\right)  ,\bar{\rho
}_{\pm}>0$ satisfy \eqref{sum=1} and  \eqref{constants_at_infinity_2}.
There exists a constant $c_2>0$ depending on $\eta,\gamma_{\pm},A_{\pm}$ and $d$ such
that for any initial data $\left(\alpha_{+0}, \alpha_{-0}, \rho_{+0}, \rho_{-0}, u_0\right)$
verifying%
\[
\left\Vert (\alpha_{\pm0}-\bar{\alpha}_{\pm}, \rho_{\pm0}-\bar{\rho}_{\pm}, u_0)\right\Vert _{B^{\frac{d}{2}-1}\cap{B^{\frac
{d}{2}+1}}}\leq c_2,
\]
then System \eqref{K} admits a unique global-in-time solution
$(\alpha_{+}, \alpha_{-},  \rho_{+}, \rho_{-}, u)$ such
that
\[
\left\{
\begin{array}
[c]{l}%
(\alpha_{\pm} -\bar{\alpha}_{\pm}, \rho_{\pm} -\bar{\rho}_{\pm}, u)\in\mathcal{C}_{b}(\mathbb{R}_{+}, 
{B}^{\frac{d}{2}-1}\cap B^{\frac{d}{2}+1}),\;\\
\rho_{\pm}-\bar\rho_{\pm}\in L^1(\mathbb{R}_+;  B^{\frac{d}{2}+1})\quad \text{ and }\quad
u\in  L^1(\mathbb{R}_+;  B^{\frac{d}{2}}\cap B^{\frac{d}{2}+1}).
\end{array}
\right.  \;
\]
\end{theorem}


It turns out that we can further obtain  a convergence rate of solutions of System \eqref{DBN} towards solutions of System \eqref{K}.
 \begin{theorem}\label{Th-3}
Let $d\geq3$ and assume the same hypothesis on the parameters as in Theorem \ref{Th-2}. Let $(\alpha_+^\nu, \alpha_-^\nu, \rho_+^\nu, \rho_+^\nu, u^\nu)$ (resp. $(\alpha_+, \alpha_-, \rho_+, \rho_-, u)$) be the solution to the Cauchy problem \eqref{DBN}, associated with the initial data $(\alpha^\nu_{+0}, \alpha^\nu_{-0}, \rho^\nu_{+0}, \rho^\nu_{-0}, u^\nu_0 )$, from Theorem \ref{Th-2} (resp. \eqref{K}-\eqref{initial-data-S1} from Theorem \ref{Th-1}) such that
\begin{multline*}
\|(\dfrac{\alpha_{+0}^\nu\rho_{+0}^\nu}{\alpha_{+0}^\nu\rho_{+0}^\nu+\alpha_{0-}^\nu\rho_{-0}^\nu}-\dfrac{\alpha_{+0}\rho_{+0}}{\alpha_{+0}\rho_{+0}+\alpha_{-0}\rho_{-0}},P^\nu_{\pm0}-P_{\pm0}, u^\nu_0-u_0)\|_{ B^{\frac{d}{2}-\frac{3}{2}}\cap B^{\frac{d}{2}-\frac{1}{2}}}\\+\|P_{+0}^\nu-\dfrac{\gamma_+\alpha_-^\nu P_+^\nu}{\gamma_+\alpha_-^{\nu} P_+^{\nu} +\gamma_-\alpha_+^{\nu}P_- ^{\nu}} (P_{+0}^\nu-P_{-0}^\nu)-P_{+0}\|_{ B^{\frac{d}{2}-\frac{3}{2}}\cap B^{\frac{d}{2}-\frac{1}{2}}}\leq C\sqrt{\nu}
\end{multline*}

Then there exists a constant $C>0$ independent of $\nu$ such that $(\alpha_+^\nu, \alpha_-^\nu, \rho_+^\nu, \rho_+^\nu, u^\nu)$  converges toward $(\alpha_+, \alpha_-, \rho_+, \rho_-, u)$ in the following sense 
\begin{align*}
\|(\alpha_{\pm}^\nu-\alpha_\pm, \rho_{\pm}^\nu-\rho_\pm, \rho_{-}^\nu-&\rho_-, u^\nu-u)\|_{L^\infty(B^{\frac{d}{2}-\frac{1}{2}})}+    \|\rho_{\pm}^\nu-\rho_{\pm}\|_{L^2  (B^{\frac{d}{2}-\frac{1}{2}})}+\|u^\nu-u\|_{L^1  (B^{\frac{d}{2}-\frac{1}{2}})}\leq C\sqrt{\nu}.
\end{align*}
  \end{theorem}
  \begin{remark}
  The rather strange condition on the difference of initial data is due to a technical limitation on the composition arguments, see Proposition \ref{Composition} for more details. In Section 6, we   state a more comprehensive  theorem concerning stability   between the two systems,  which immediately implies Theorem \ref{Th-3}.
  \end{remark}
  
  
  %

  \subsection{A short review of recent results concerning partially dissipative systems}

Since we are interested in the inviscid limit, our approach will use
techniques from \textit{partially dissipative first order quasilinear
systems}, that is, systems of the following general form:
\begin{equation}
\partial_{t}w+\sum_{j=1}^{d}\partial_{j}F_{j}(w)=\frac{Q(w)}{\varepsilon}%
\quad\text{with}\quad Q(w)=\left(
\begin{matrix}
0\\
q(w)
\end{matrix}
\right)  ,\text{ }\varepsilon>0 \label{PartDissip}%
\end{equation}
where $w\in\mathbb{R}^{n}$ and $\varepsilon>0$ is a given parameter. These types
of models govern the dynamic of physical systems out of thermodynamic
equilibrium which is typically the case in gas dynamics. The constant
$\varepsilon$ can be seen as a relaxation time and in practice is very small.
Observe that, a priori, the dissipative effect does not concern all the
components of the unknown $w\in\mathbb{R}^{n}$. In order to obtain global
existence results, for initial data close to equilibria, one aims at recovering
such an effect for all the components of $w$. And as observed in \cite{chen1994}, using the Chapman-Enskog
expansion, one can deduce a more accurate correction for
the limiting system which has the general form
\begin{equation}
\partial_{t}w+\sum_{j=1}^{d}\partial_{j}F_{j}(w)-\varepsilon\sum
_{i,j=1}^{d}\partial_{j}(B_{ij}(w)\partial_{j}w)=\frac{Q(w)}{\varepsilon}.
\label{LinearBZ2}%
\end{equation}

A first approach in order to obtain global-in-time existence results for
initial data around a constant equilibria is restricted to systems verifying
the so called (SK) condition which was established by Shizuta and Kawashima in
\cite{SK}. This is an explicit linear stability criterion which ensures that
all the components of $w$ decay as $t\rightarrow+\infty$. In order to extend
these results to more general quasilinear systems, a second condition was put
forward by Yong in \cite{Yong}: the existence of an entropy that provides a
symmetrisation compatible with the dissipation $Q$ appearing in system
$\left(  \text{\ref{PartDissip}}\right)  $.
\bigbreak
In a situation close to the one we consider here, Qu and Wang in \cite{QuWangNoSk} established a global existence result for quasilinear hyperbolic systems such that one and only one of the eigen-family violates the (SK) condition. The (BN)-system satisfies this condition but it doesn't satisfy all the conditions necessary to directly apply their result and more importantly, we need to obtain uniform estimates to tackle the relaxation problem.
\bigbreak The study of relaxation problems associated to systems of conservation laws
can be tracked back to the work of Chen and al \cite{chen1994}. More recently,
Giovangigli and Yong in \cite{Gio1,Gio2} studied a relaxation-limit
problem arising in the dynamics of perfect gases out of
thermodynamical-equilibrium. At a mathematical level, they dealt with
dissipative and diffusive systems of conservation laws of the form $\left(
\text{\ref{LinearBZ2}}\right)  $. Roughly speaking, assuming the existence of
an entropy compatible with the diffusion and dissipation operators,
they proved the local existence of solutions for the Cauchy problem as well as
error estimates between solutions at fixed $\varepsilon$ and the solution of
the limit system. We mention that as their results hold only locally in time,
the (SK) condition is not relevant in their work (the dissipative term being
responsible for the large-time behaviour of the solution). We also mention
some previous work by Giovangigli and Matuszewski \cite{Gio3,Gio4} where
they study chemically-reactive multicomponent flows.

Recently, Danchin and the second author in \cite{CBD1,CBD2} studied partially
dissipative hyperbolic systems satisfying both the (SK) and a symmetrisation condition (which is weaker than the one imposed by Yong in \cite{Yong})\footnote{They consider non-conservative partially dissipative hyperbolic systems that are Friedrichs-symmetrizable} in the framework of critical homogeneous Besov
spaces. In particular, this includes the compressible Euler
system with damping in the velocity equation which has been studied in several papers \cite{XuFang,XK1E,LC}. In some sense, this is the
functional analysis framework which uses the optimal regularity that one
has to impose on the initial data in order to obtain global well-posedness
results. They combined two ideas in order to obtain these results. On the one
hand, inspired by the work of Beauchard and Zuazua \cite{BZ} they
constructed a Lyapunov functional implying that, close to equilibrium, the low
frequencies of the solutions of the non-linear system $\left(
\text{\ref{PartDissip}}\right)  $ behave, qualitatively, like a
heat equation. And on the other hand, they highlighted a \textit{damped mode} which enjoys better
decay properties in low frequencies and deduced from it crucial regularity enhancement in order to
close the a priori estimates.

It turns out that this method is flexible enough to be adapted for the systems
\eqref{DBN} and \eqref{K} if we want to obtain a well posedness result. However, the estimates that one would obtain
by adapting the results from \cite{CBD1,CBD2} would not uniform with respect
to the parameters $\mu$ and $\lambda$, thus they cannot be used directly in order
to justify this vanishing viscosity limit.
\subsection{Strategy of proof}

\indent Our problem is not covered by the papers mentioned above and it is not completely clear if and how the general theories from
\cite{BZ,CBD2,Gio1,Gio2,Gio3,Gio4} could be adapted to study $\left(
\text{\ref{DBN}}\right)  $ and the associated relaxation limit. The first
obvious reason is that $\left(  \text{\ref{DBN}}\right)  $ is not a system of
conservation laws because the equations of the volume fractions cannot be put
in conservative form. The second reason is that the entropy that is naturally
associated with this system is not positive
definite since it is linear with respect to the volume fractions. Concerning global
existence results, we remark that the associated quasilinear system does not
satisfy the (SK) condition as it admits the eigenvalue $0$. It turns out that
the situation is not too degenerate in the sense that the eigenspace associated
to the eigenvalue $0$ is of dimension $1$ and that, roughly speaking, the
non-degenerate part (i.e. the part associated to non-zero eigenvalues) fulfils the (SK)
condition. 
Thus we will be able to isolate the undamped mode, and rewrite the remaining
system as a partially dissipative quasilinear system satisfying the (SK) condition while the undamped mode
will be seen as a parameter and always appears in nonlinear terms as a
prefactor of a function of the damped variable. 

More precisely, after observing that there exist four main unknowns in the
system, namely, $\alpha_{+},\rho_{+},\rho_{-}$ and $u$ we consider a change of
variables that leaves the velocity $u$ invariant:%
\[
\left(  y,w,r\right)  =\Phi\left(  \alpha_{+},\rho_{+},\rho_{-}\right)
\]
where the variable $w$ is proportional to $P_{+}-P_{-}$ and $r$ is like an effective pressure.
The system verified by the new variables is of the form
\begin{equation}
\left\{
\begin{array}
[c]{l}%
\partial_{t}y+u\cdot\nabla y=0,\\
\partial_{t}w+u\cdot\nabla w+\bigl(\bar{H}_{1}+H_{1}%
(w,r,y)\bigr)\operatorname{div}u+\bigl(\bar{H}_{2}+H_{2}(w,r,y)\bigr)\dfrac
{w}{\nu}=0,\\
\partial_{t}r+u\cdot\nabla r+\bigl(\bar{H}_{3}+H_{3}%
(w,r,y)\bigr)\operatorname{div}u=\bigl(\bar{H}_{4}+H_{4}(w,r,y)\bigr)\dfrac
{w^{2}}{\nu},\\
\partial_{t}u+u\cdot\nabla u-\dfrac{1}{\rho}\mathcal{A}_{\mu,\lambda}u+\eta
u+\dfrac{1}{\rho}\nabla r+\left(  \gamma_{+}-\gamma_{-}\right)  \dfrac{1}%
{\rho}\nabla w=0,
\end{array}
\right.
\label{generalform}
\end{equation}
where $\nu=2\mu+\lambda,$ $\bar H_{i}$  and $H_{i}$ $(i=1, 2, 3, 4)$ stands for constant part and perturbation part, respectively. The very specific form of the nonlinear part appearing in the above system is crucial to close our estimates.
Obviously there is no hope to recover time decay
properties for $y$ and in fact we only need $L^\infty_T$ integrability on $y$, therefore we will treat this transport equation separately.
Considering the system satisfied by the three unknowns $(w,r,u)$ and by adapting similar ideas developed in   \cite{BZ,CBD1,CBD2} to this sub-system, we will obtain the necessary
integrability on all the components of the solution, which will allow us to
obtain uniform a priori estimates with respect to $\lambda$ and $\mu$.

The main difference with the papers  by Danchin and the second author is that we cannot perform a rescaling to keep track of the coefficients as what they did, because we are not able to treat the low frequencies in $B^{d/2}$ as $\rho$ lacks of time integrability since the complex form of the pressure.

It is important to point out that system \eqref{generalform} does not verify the (SK) condition and it is not symmetric. However, the subsystem formed by the last three equations of \eqref{generalform} satisfies the (SK) condition and the coupling with the first equation is achieved via lower-order terms. In order to deal with the lack of symmetry, we construct a nonlinear energy-functional in order to derive a priori estimates.

 Another technical difficulty that we encounter is that we cannot recover   uniform dissipation with respect  to $1/\nu$ for $w$ at the higher energy level (i.e. $d/2+1$). We can recover such a strong decay effect only for the $d/2-$energy level. This renders delicate the estimation of nonlinear terms which are proportional with $1/\nu$. Thus, the quadratic form of the nonlinearity appearing in the equation of $r$ turns out to be crucial.


 Moving on to the justification of the relaxation limit, the fact that the solutions to the Kapilla system \eqref{K} are obtained as limits of the \eqref{DBN} system is a consequence of the uniform estimates and classical weak-compactness arguments.
 In order to obtain a convergence rate, we estimates the difference of the solutions of the two systems. Since we are not able to a obtain decay rate for $\partial_t(P_+^\nu-P_-^\nu)$ in any space, we define a new unknown to avoid treating this term as a source term. Under a smallness assumption depending on $\nu$ and on the difference of the initial data we are able to obtain decay rate of $\sqrt{\nu}$ in ${L}^\infty(B^{\frac{d}{2}-\frac{3}{2}}\cap B^{\frac{d}{2}-\frac{1}{2}})$ which allows us to recover decay rate for the source terms involving the unknown $\alpha_{\pm}.$ But still,  it seems hard to control $\alpha_{\pm}$ in the spaces $L^1(B^s)$ for any $s$, to our knowledge,  this is the reason why we end up with a convergence rate equal to $\sqrt{\nu}$.

\medbreak
\medbreak
\medbreak
\noindent The rest of the paper unfolds as follows. 

\noindent The next section is devoted to Littlewood-Paley theory and Besov spaces with some of its useful properties.
 
 \noindent In Section 3, we rewrite the original (BN)-system into good unknowns and give a theorem for the reformulated system, that is  Theorem \ref{Thm-III}.
 
 \noindent In Section 4 we derive a priori estimates uniform with respect to the viscosity parameters for a mixed linear partially dissipative hyperbolic system and prove Theorem \ref{Thm-III}.
 
 \noindent The fifth section is devoted to the stability estimates between (BN)-system and (K)-system, which implies  Theorem \ref{Th-3}.
 
 \noindent Finally in the Appendix, we show  some basic estimates for several  classical linear problem. 
 
\section{A primer on Besov spaces and Littlewood-Paley theory}
At this stage, we need to introduce a few notations and the main functional
spaces that we will use in our paper. First, throughout the paper, we fix a
homogeneous Littlewood-Paley decomposition $(\dot{\Delta}_{j})_{j\in
\mathbb{Z}}$ that is defined by
\[
\dot{\Delta}_{j}\triangleq\varphi(2^{-j}D)
\]
where $\varphi$ stands for a smooth function supported in $\mathcal{C}=\{\xi\in\mathbb{R}^d,\:5/6\leq|\xi|\leq 12/5\}$ such that $$\sum_{q\in\mathbb{Z}}\varphi(2^{-j}\xi)=1 \text{  for  } \xi\ne0.$$ 
\smallbreak Following \cite{HJR}, we introduce the homogeneous Besov
semi-norms:
\[
\Vert u\Vert_{B^{s}}\triangleq \bigl\|2^{js}\Vert\dot{\Delta}%
_{j}u\Vert_{L^{2}(\mathbb{R}^{d})}\bigr\|_{\ell^{1}(\mathbb{Z})}.
\]
Notice that, since we will only work with homogeneous Besov spaces with second index equal to 2 and third index equal to 1, we omitted those index in the definition of the norm.
Then define the homogeneous Besov spaces $B^{s}$ (for any
$s\in\mathbb{R}$ and $(p,r)\in\lbrack1,\infty]^{2}$) to be the subset of $u$
in $\mathcal{S}_{h}^{\prime}$ such that $\Vert u\Vert_{B^{s}}$ is
finite. \smallbreak To any element $u$ of $\mathcal{S}_{h}^{\prime},$ we
associate the low and high frequency of its Besov norms through \footnote{For
technical reasons, we need a small overlap between low and high frequencies.}
if $r<\infty$
\[
\left\Vert u\right\Vert _{B^{s}}^{\ell}\triangleq
\sum_{j\leq 0}2^{js}\left\Vert \dot{\Delta}_{j}u\right\Vert _{L^{2}}%
\quad\text{and}\quad\left\Vert u\right\Vert
_{B^{s}}^{h}\triangleq\sum_{j\geq -1}%
2^{js}\left\Vert \dot{\Delta}_{j}u\right\Vert _{L^{2}}
\cdotp
\]
We define 

$$u^{\ell}=\sum_{j\leq-1}\ddj u,\text{ and }~u^h=u-u^{\ell}.$$ 

We will frequently use that 

$$\|u^{\ell}\|_{B^s}\leq C\|u\|^{\ell}_{B^s} \text{ and } \|u^{h}\|_{B^s}\leq C\|u\|^{h}_{B^s}.$$

For any Banach space $X,$ index $\rho$ in $[1,\infty]$ and time $T\in\lbrack
0,\infty],$ we use the notation 
$$\Vert u \Vert
_{L_{T}^{\rho}(X)}\triangleq\bigl\|\Vert u\Vert_{X}\bigr\|_{L^{\rho}(0,T)} \text{ and } \Vert u \Vert
^h_{L_{T}^{\rho}(X)}\triangleq\bigl\|\Vert u\Vert_{X}^h\bigr\|_{L^{\rho}(0,T)}$$ and similarly, with $\ell$ instead of $h$, for the low frequencies.
If $T=+\infty$, then we just write $\Vert\left\Vert u\right\Vert
\Vert_{L^{\rho}(X)}.$
Finally, in the case where $u$ has $n$ components $z_{j}$ in $X,$ we slightly
abusively keep the notation $\left\Vert u\right\Vert $ to mean $\sum
_{j\in\{1,\cdots,n\}}\left\Vert u_{j}\right\Vert _{X}$. Moreover, for our computations, we need to introduce the following Chemin-Lerner norms: 
\[
\Vert u\Vert_{\widetilde{L}^\rho_T(B_{p,r}^{s})}\triangleq\bigl\|2^{js}\Vert\dot{\Delta}%
_{j}u\Vert_{L^\rho_T(L^{p})}\bigr\|_{\ell^{r}(\mathbb{Z})}.
\]
Those spaces are link to the classical one through the Minkowski inequality, indeed we have:
\begin{eqnarray*}
\text{if }r\geq\rho, \quad \Vert u\Vert_{\widetilde{L}^\rho_T(B_{p,r}^{s})}\leq \Vert u\Vert_{L^\rho_T(B_{p,r}^{s})} \text{  and if }r\leq \rho, \quad \Vert u\Vert_{\widetilde{L}^\rho_T(B_{p,r}^{s})}\geq \Vert u\Vert_{L^\rho_T(B_{p,r}^{s})}.
\end{eqnarray*}

We now state some classical results that can be found in \cite{HJR}.
First, we introduce product laws.
\begin{proposition} \label{Productlaw} Let $s\in ]0,\infty[$. Then, 
 $B^s\cap L^\infty$ is an algebra and we have
\begin{equation}\label{eq:prod1}
\|uv\|_{B^{s}}\leq C\bigl(\|u\|_{L^\infty}\|v\|_{B^{s}}+\|u\|_{B^{s}}\|v\|_{L^\infty}\bigr)\cdotp
\end{equation}
If, furthermore, $-d/2<s\leq d/2,$ then the following inequality holds:
\begin{equation}\label{eq:prod2}
\|uv\|_{B^{s}}\leq C\|u\|_{B^{\frac{d}{2}}}\|v\|_{ B^{s}}.
\end{equation}
and if $(s_1,s_2) \in ]-d/2,d/2[$ such that $s_1+s_2>0$, then \begin{equation}\label{eq:prod3}
\|uv\|_{ B^{s_1+s_2-\frac{d}{2}}}\leq C\|u\|_{ B^{s_1}}\|v\|_{ B^{s_2}}.
\end{equation}
\end{proposition}
Then, we state a result concerning commutator estimates.

\begin{proposition}\label{Commutator} Consider $s\in \left]-\frac d2 -1,\frac{d}{2}\right]$. The following inequalities hold true: 

 \begin{equation}\label{eq:com1}
2^{j(s+1)}\|[u,\dot{\Delta}_j]v\|_{L^2}\leq Cc_j\|\nabla u\|_{B^{\frac{d}{2}}}\|v\|_{B^{s}}.
\end{equation}
 \end{proposition}

Finally  we present a proposition related to composition operator that can be found in \cite{RunstSickel} p.387-388.
\begin{proposition}\label{Composition}
Let $m\in\mathbb{N}$ and $s>0$. Let $G$ be a function in $\mathcal{C}^\infty(\mathbb{R}^m)$ such that $G(0,..,0)=0$.
\\Then for every real-valued functions $f_1,..,f_m$ in $B^{s}\cap L^\infty$, the function $G(f_1,..,f_m)$ belongs to $B^{s}\cap L^\infty$ and we have
$$\|G(f_1,..,f_m)\|_{B^{s}}\leq \|(f_1,..,f_m)\|_{B^{s}}\left(1+C(\|(f_1\|_{L^\infty}+...+\|f_m)\|_{L^\infty})\right).$$
\end{proposition}
 
We give classical estimates concerning the damped transport equation and the Lamé system in the Appendix.

\section{A reformulation of the System  \texorpdfstring{\eqref{DBN}}{TEXT} and sketch of the proof}
  The first part of this section will concern the reformulation of the System \eqref{DBN} so it is in the range of  application of recent developments about partially dissipative hyperbolic system. Then we state a global existence result which contains the statement of Theorem \ref{Th-2}.

\subsection{Change of unknowns}
Here,  we propose new unknowns that are more appropriate  to obtain uniform a priori estimates. 
First, observe that by adding the equations of $\alpha_{+}$ and
$\alpha_{-}$ together  we get%
\[
D_{t}(\alpha_{+}+\alpha_{-})=0,
\]
where    $D_{t}$ is the material derivative defined by
\begin{equation}
D_{t}:=\partial_{t}+u\cdot\nabla\label{material_derivative}.%
\end{equation}
 Thus, at least formally, if the initial data considered satisfies

\[
 \alpha_{+0}+\alpha_{-0} =1,
\]
then this property remains true for latter times. Thus, the number of
independent unknowns for system \eqref{DBN} is
reduced to four: $\alpha_{+},\rho_{+}, \rho_-$ and $u$.

We now consider the change of unknowns from $\left(  \alpha_{+},\rho
_{+},\rho_{-}\right)  $ to $\left(  w, R, Y\right)  $ given by
\begin{equation}
\left(  w,R,Y\right)  :=\Phi\left(  \alpha_{+},\rho_{+},\rho_{-}\right)
=\left(  \Phi_{1}\left(  \alpha_{+},\rho_{+},\rho_{-}\right)  ,\Phi_{2}\left(
\alpha_{+},\rho_{+},\rho_{-}\right)  ,\Phi_{3}\left(  \alpha_{+},\rho_{+}%
,\rho_{-}\right)  \right)  , \label{def_Phi}%
\end{equation}
with
\begin{equation}
\left\{
\begin{array}
[c]{l}%
\Phi_{1}\left(  \alpha_{+},\rho_{+},\rho_{-}\right)  =\dfrac{P_{+}\left(
\rho_{+}\right)  -P_{-}\left(  \rho_{-}\right)  }{\dfrac{\gamma_{+}}%
{\alpha_{+}}+\dfrac{\gamma_{-}}{\alpha_{-}}},\\
\Phi_{2}\left(  \alpha_{+},\rho_{+},\rho_{-}\right)  =\alpha_{+}P_{+}\left(
\rho_{+}\right)  +\alpha_{-}P_{-}\left(  \rho_{-}\right)  -\left(  \gamma
_{+}-\gamma_{-}\right)  \dfrac{P_{+}\left(  \rho_{+}\right)  -P_{-}\left(
\rho_{-}\right)  }{\dfrac{\gamma_{+}}{\alpha_{+}}+\dfrac{\gamma_{-}}%
{\alpha_{-}}},\\
\Phi_{3}\left(  \alpha_{+},\rho_{+},\rho_{-}\right)  =\dfrac{\alpha_{+}%
\rho_{+}}{\alpha_{+}\rho_{+}+\alpha_{-}\rho_{-}}.
\end{array}
\right.  \label{Phi}%
\end{equation}
We see  its equilibrium state  $(\bar w, \bar R, \bar Y)$ will be   $(0, \bar P, \frac{\bar\alpha_+\bar\rho_+}{\bar\alpha_+\bar\rho_++\bar\alpha_-\bar\rho_-}).$

The differential of the transformation
computed at  $\left(  \bar{\alpha}_{+},\bar{\rho}_{+},\bar{\rho}_{-}\right)  $
is%
\[
\begin{pmatrix}
0 & \dfrac{P_{+}^{\prime}\left(  \bar{\rho}_{+}\right)  }{\dfrac{\gamma_{+}%
}{\bar{\alpha}_{+}}+\dfrac{\gamma_{-}}{\bar{\alpha}_{-}}} & \dfrac
{-P_{-}^{\prime}\left(  \bar{\rho}_{-}\right)  }{\dfrac{\gamma_{+}}%
{\bar{\alpha}_{+}}+\dfrac{\gamma_{-}}{\bar{\alpha}_{-}}}\\
0 & \bar{\alpha}_{+}P_{+}^{\prime}\left(  \bar{\rho}_{+}\right)  -\left(
\gamma_{+}-\gamma_{-}\right)  \dfrac{P_{+}^{\prime}\left(  \bar{\rho}%
_{+}\right)  }{\dfrac{\gamma_{+}}{\bar{\alpha}_{+}}+\dfrac{\gamma_{-}}%
{\bar{\alpha}_{-}}} & \bar{\alpha}_{-}P_{-}^{\prime}\left(  \bar{\rho}%
_{-}\right)  +\left(  \gamma_{+}-\gamma_{-}\right)  \dfrac{P_{-}^{\prime
}\left(  \bar{\rho}_{-}\right)  }{\dfrac{\gamma_{+}}{\bar{\alpha}_{+}}%
+\dfrac{\gamma_{-}}{\bar{\alpha}_{-}}}\\
\dfrac{\bar{\rho}_{+}\bar{\rho}_{-}}{\bar{\rho}^{2}} & \dfrac{\bar{\alpha}%
_{+}\bar{\alpha}_{-}\bar{\rho}_{-}}{\bar{\rho}^{2}} & -\dfrac{\bar{\alpha}%
_{+}\bar{\alpha}_{-}\bar{\rho}_{+}}{\bar{\rho}^{2}}%
\end{pmatrix}
,
\]
such that the Jacobian of the transformation computed at $\left(  \bar{\alpha
}_{+},\bar{\rho}_{+},\bar{\rho}_{-}\right)  $ \ is%
\[
J_{|\left(  \bar{\alpha}_{+},\bar{\rho}_{+},\bar{\rho}_{-}\right)  }%
=\dfrac{\bar{\rho}_{+}\bar{\rho}_{-}}{\bar{\rho}^{2}}\cdot\frac{P_{+}^{\prime
}\left(  \bar{\rho}_{+}\right)  P_{-}^{\prime}\left(  \bar{\rho}_{-}\right)
}{\dfrac{\gamma_{+}}{\bar{\alpha}_{+}}+\dfrac{\gamma_{-}}{\bar{\alpha}_{-}}%
}>0.
\]
Thus, owing to the Inverse Function Theorem there exists constants $\delta
_{1},\delta_{2}$ 
 and a function 
 \begin{align*}
     \Psi:& ~B_{\delta_2}(\bar w, \bar R, \bar Y)    \to B_{\delta_1}(\bar \alpha_+, \bar \rho_+, \bar \rho_-)  \\
   &~\quad(w, R, Y)\mapsto (\alpha_+, \rho_+, \rho_-)
 \end{align*}
such that $\Psi$ is one-to-one and $\Psi$ is the inverse of the restriction of
$\Phi$ to  the ball  $B_{\delta_1}(\bar \alpha_+, \bar \rho_+, \bar \rho_-).$
Thus, $\Psi$ is smooth on the ball $B_{\delta_2}(\bar w, \bar R, \bar Y)\subset \mathbb{R}^3.$  An extension theorem in the book of Evans \cite{Evans} (p. 254) enables us to assume that $\Psi$ is smooth in $\mathbb{R}^3$ without loss of generality, this extension of $\Psi$ will be useful later to use a composition lemma.  

Consider $\left( w , R , Y , u\right)(0, x)$ such that%
\begin{equation}\label{initial-data-(w, R, Y, u)}
\| (  w , R , Y , u )(0, \cdot) %
-(0, \bar P, \frac{\bar\alpha_+\bar\rho_+}{\bar\alpha_+\bar\rho_++\bar\alpha_-\bar\rho_-}, 0)\|_{B^{\frac{d}{2}-1}\cap
B^{\frac{d}{2}+1}}\leq c
\end{equation}
where $c$ is chosen such that $c\leq \frac{\delta_2}{4}.$ Furthermore,
define
\begin{equation}
T_{\max}=\sup\left\{  T>0: \| (  w , R , Y , u)(t, \cdot) %
-(0, \bar P, \dfrac{\bar\alpha_+\bar\rho_+}{\bar\alpha_+\bar\rho_++\bar\alpha_-\bar\rho_-}, 0)\| _{B^{\frac{d}{2}-1}\cap
B^{\frac{d}{2}+1}}\leq2c\right\}. \label{maximal_time}%
\end{equation}
Obviously, owing to the embedding ${B^{\frac{d}{2}-1}\cap
B^{\frac{d}{2}+1}}\hookrightarrow L_{t,x}^{\infty},$ it is clear that $(  w , R , Y , u)(t, x)$ lies in a ball centered in $(0, \bar P, \frac{\bar\alpha_+\bar\rho_+}{\bar\alpha_+\bar\rho_++\bar\alpha_-\bar\rho_-}, 0)$ with radius depending on $\delta_2$, on the time interval $[0, T_{\max}).$
Theorem \ref{Thm-III} in the next subsection shows that for $c$ chosen sufficiently small then $T_{\max}=+\infty$, and thus $(\alpha_+, \rho_+, \rho_-)=\Psi(w, R, Y)$ for all time.

Let us now derive the equations of $(w, R, Y, u).$
We observe from the first and the second equations of System \eqref{DBN} that
\[ 
\alpha_{\pm}(\partial_{t}\rho_{\pm}+\operatorname{div}(\rho_{\pm}%
u))=\mp\frac{\alpha_{+}\alpha_{-}\rho_{\pm}}{\nu}(P_{+}(\rho_+)-P_{-}(\rho_-)).
\]
Since we work with the power laws $P_{\pm}(\rho_{\pm})=A_{\pm}\rho_{\pm}^{\gamma_{\pm}}$ (which we simply represent by $P_{\pm}$),   by multiplying   the  two equations above by $P^{\prime}_{\pm}$ respectively,  we obtain that%
\[
D_{t}P_{\pm} +\gamma_{\pm}P_{\pm} \operatorname{div}u=\mp\frac
{\gamma_{\pm}\alpha_{\mp}P_{\pm} }{\nu}(P_{+} -P_{-} ).
\]

Taking into consideration the equations of $\alpha_{\pm}$, we get that%

\begin{equation}
D_{t}P+(\gamma_{+}\alpha_{+}P_{+} +\gamma_{-}\alpha_{-}P_{-} )\operatorname{div}%
u+\frac{\alpha_{+}\alpha_{-}}{\nu}(P_{+} -P_{-} )\bigl(  \left(
\gamma_{+}-1\right)  P_{+} -\left(  \gamma_{-}-1\right)  P_{-}\bigr)  =0,
\label{equation_of_P_1}%
\end{equation}
which we further put under the form
\begin{equation}
D_{t}P+(\gamma_{+}\alpha_{+}P_{+}+\gamma_{-}\alpha_{-}P_{-})\operatorname{div}%
u+\left(  \gamma_{+}-1\right)  \frac{\alpha_{+}\alpha_{-}}{\nu}%
(P_{+}-P_{-})^{2}+\left(  \gamma_{+}-\gamma_{-}\right)  \frac{\alpha_{+}%
\alpha_{-}}{\nu}(P_{+}-P_{-})P_{-}=0. \label{equation_of_P_2}%
\end{equation}
Next, observe that
\begin{equation}
D_{t}\left(  P_{+}-P_{-}\right)  +(\gamma_{+}P_{+}-\gamma_{-}P_{-}%
)\operatorname{div}u+(\frac{\gamma_{+}P_{+}}{\alpha_{+}}+\frac{\gamma_{-}%
P_{-}}{\alpha_{-}})\frac{\alpha_{+}\alpha_{-}}{\nu}\left(
P_{+}-P_{-}\right)  =0, \label{equation_deltaP_1}%
\end{equation}
which we further put under the form%
\begin{equation}
D_{t}\left(  P_{+}-P_{-}\right)  +(\gamma_{+}P_{+}-\gamma_{-}P_{-}%
)\operatorname{div}u+\frac{\gamma_{+}\alpha_{-}}{\nu}\left(
P_{+}-P_{-}\right)  ^{2}+\left(  \frac{\gamma_{+}}{\alpha_{+}}+\frac
{\gamma_{-}}{\alpha_{-}}\right)  \frac{\alpha_{+}\alpha_{-}}{\nu}\left(  P_{+}-P_{-}\right)  P_{-}=0. \label{equation_deltaP_2}%
\end{equation}

Notice that we chose to define the effective pressure $R$ by:
$$R=P-\left(  \gamma_{+}-\gamma_{-}\right) \frac{P_{+}-P_{-}}{ \frac{\gamma_{+}}{\alpha_{+}}+\frac
{\gamma_{-}}{\alpha_{-}}}= \frac{\gamma_-\alpha_+}{\gamma_+\alpha_-+\gamma_-\alpha_+}P_++\frac{\gamma_+\alpha_-}{\gamma_+\alpha_-+\gamma_-\alpha_+}P_-$$
to cancel the coupling between $P_+-P_-$ and $P_-$ in \eqref{equation_of_P_2}, and  the choice of $$w=\frac{P_{+}-P_{-}}{ \frac{\gamma_{+}}{\alpha_{+}}+\frac
{\gamma_{-}}{\alpha_{-}}}=\frac{\alpha_+\alpha_-}{\gamma_+\alpha_-+\gamma_-\alpha_+}(P_+-P_-)$$ enables us to rewrite the  pressure in a simple form: $P= R+(\gamma_+-\gamma_-)w.$ The unknown $w$ can be comprehended as the \textit{damped part} of the pressure.
A straightforward computation yields%
\[
D_{t}  \Bigl(\frac{\gamma_{+}}{\alpha_{+}}+\frac
{\gamma_{-}}{\alpha_{-}}\Bigr)^{-1}= \dfrac{\gamma_+\alpha_-^2-\gamma_-\alpha_+^2}{\gamma_+\alpha_-+\gamma_-\alpha_+}\frac{w}{\nu}
\]
and
\[
\left\{
\begin{array}
[c]{l}%
P_+= R+\dfrac{\gamma_+}{\alpha_+}w,\\
P_-=R-\dfrac{\gamma_-}{\alpha_-}w.
\end{array}
\right.
\]
Then the equation of $w$ reads
\begin{align}
D_{t}w+F_1\div u+F_2\frac{w}{\nu}=0
\label{equation_of_w_1}%
\end{align}
with
\[
\left\{
\begin{array}
[c]{l}%
F_1:=  \dfrac{(\gamma_{+}-\gamma_-)\alpha_+\alpha_-}{\gamma_+\alpha_-+\gamma_-\alpha_+}R+\dfrac{\gamma_+^2\alpha_-+\gamma_-^2\alpha_+}{\gamma_+\alpha_-+\gamma_-\alpha_+}w,\\
F_2:= (\gamma_+\alpha_-+\gamma_-\alpha_+)R-\dfrac{(\gamma_+-\gamma_+^2)\alpha_-^2-(\gamma_--\gamma_-^2)\alpha_+^2}{\alpha_+\alpha_-}w,
\end{array}
\right.
\]
and the equation of $R$ reads
\begin{equation} 
D_{t}R+F_{3}\operatorname{div}u=F_{4}\frac{w^{2}}{\nu} \label{equation_of_R_1}%
\end{equation}
with
\[
\left\{
\begin{array} 
[c]{l}%
F_3:=\dfrac{\gamma_+\gamma_-}{\gamma_+\alpha_-+\gamma_-\alpha_
+}\bigl(R +(\gamma_+-\gamma_-)w\bigr),\\
F_4:=\dfrac{\gamma_+\gamma_-}{\alpha_+\alpha_-}\bigl(1-(\gamma_+\alpha_-+\gamma_-\alpha_+)\bigr).
\end{array}
\right.
\]
Moreover, we have
\begin{equation}\label{rewrite_alpha}
\alpha_+= \Psi_1(w, R, Y), \quad\alpha_-=1-\Psi_1(w, R, Y).
\end{equation}

Next, as
\begin{equation*}
Y=\frac{\alpha_{+}\rho_{+}}{\rho}, 
\end{equation*}
 we observe that%
\begin{equation}
D_{t}Y=0. \label{equation_of_Y}%
\end{equation}

We denote the perturbations of $Y$ and $R$ by
\begin{equation}
y:=Y-\frac{\bar{\alpha}_{+}\bar{\rho}_{+}}{\bar{\alpha}_{+}\bar{\rho}_{+}+\bar{\alpha}_{-}\bar{\rho}_{-}},\text{ }r=R-\bar{P}.
\label{definition_little_y_r}%
\end{equation}
We see that there exists a function $G_0$ of the unknowns $(w, r, y)$ such that
\begin{align*}
\frac{1}{\rho}=\frac{1}{\alpha_+\rho_++\alpha_-\rho_-}= \bar F_0+G_0(w, r, y),\quad{\rm{with}}~~\bar F_0:= \frac{1}{\bar\alpha_+\bar\rho_++\bar\alpha_-\bar\rho_-}.
\end{align*}
Gathering the equations \eqref{equation_of_w_1}, \eqref{equation_of_R_1}, \eqref{equation_of_Y} and the equation of $u$ together, we obtain the following system in terms of the unknowns $(y,w,r,u)$:%
\begin{equation}
\left\{
\begin{array}
[c]{l}%
D_{t}y=0,\\
D_{t}w+\bigl(\bar F_{1}+G_1\bigr)\div u+\bigl(\bar F_2+G_2\bigr)\dfrac{w}{\nu}=0,\\
D_{t}r+\bigl(\bar F_{3}+G_3\bigr)\div u=F_4\dfrac{w^{2}}{\nu},\\
D_{t}u- \bigl(\bar F_0+G_0\bigr)\mathcal{A}_{\mu,\lambda}u+\eta u+ \bigl(\bar F_0+G_0\bigr)%
\nabla  r+\left(  \gamma_{+}-\gamma_{-}\right)  \bigl(\bar F_0+G_0\bigr)\nabla w   =0
\end{array}
\right.\label{III}
\end{equation}
where%
\begin{equation}
\left\{
\begin{array}
[c]{l}%
\bar F_1:=  \dfrac{(\gamma_{+}-\gamma_-)\bar\alpha_+\bar\alpha_-}{\gamma_+\bar\alpha_-+\gamma_-\bar\alpha_+}\bar P>0,\\
\bar F_2:= (\gamma_+\bar\alpha_-+\gamma_-\bar\alpha_+)\bar P>0,\\
\bar F_3:= \dfrac{\gamma_+\gamma_-}{\gamma_+\alpha_-+\gamma_-\alpha_+}\bar P>0,\\
G_{i}(w, r, y):= F_{i}-\bar F_{i} \quad {\rm{for ~~ each}}~~ i=1, \cdots, 3.
\end{array}
\right.  \label{definition_barF}%
\end{equation}
Note that by virtue of \eqref{rewrite_alpha} and \eqref{definition_little_y_r}, the $G_i$ (for $i=0, 1, 2, 3$)
can be written as smooth functions of the unknowns $(w, r, y)$ which vanish at origin.

\subsection{Elements of proof for Theorem \ref{Th-2}}

\indent Let us observe that the first equation is a pure transport equation, and there is  a linear coupling between the second to fourth equations.
 As we are considering viscosity vanishing limit, we need to get appropriate estimates independent of $\mu, \nu,$ in other words we should hardly use the smoothing effect of operator $\mathcal{A}_{\mu, \nu}.$ This fact motives us to study the following mixed linear  system:
\begin{equation}
\left\{
\begin{array}
[c]{l}%
\partial_{t}w+v\cdot\nabla w+\bigl(h_{1}+H_1 \bigr)\div u+ (h_2+H_2) \dfrac{w}{\nu}=S_2,\\
\partial_{t}r+v\cdot\nabla r+\bigl(h_3+H_3 \bigr)\div u=S_3,\\
\partial_{t}u+v\cdot\nabla u-\bigl(h_4+H_4 \bigr)\mathcal{A}_{\mu,\lambda}u+\eta u+\bigl(h_5+H_5 \bigr)%
\nabla  r+\bigl(h_6+H_6 \bigr) \nabla w   =S_4
\end{array}
\right.\label{L}
\end{equation}
where $S_2, S_3, S_4$ and  $H_1, H_2, \cdots, H_6$ are given functions of $(t, x)$, and $h_1, \cdots, h_6$ are given positive constants.

  Inspired by the work of Beauchard and Zuazua in \cite{BZ}, the work of Danchin in \cite{Handbook} and following the work of Danchin and the second author in \cite{CBD2,CBD1},  we expect that System \eqref{L} behaves like the heat equation in the low frequencies regime while we expect a damping effect for the high frequencies under if $H_1, \cdots, H_6$ are small enough. 
  Moreover we expect to recover better integrability  properties for $w$ and $u$ in the low frequency regime because they undergo \textit{direct} damping.
  
Precisely,  we obtain  the following \textit{a priori} estimates uniformly with respect to $\mu, \lambda, \nu$.
\begin{proposition}\label{Prop-L}
Let $d\geq 2$ and let parameters satisfy $  \mu, \lambda+\mu\geq0,~   \nu=2\mu+\lambda\leq 1$ and $\eta\geq1$.  Assume that
\begin{align}\label{A1}
H_1, H_2, \cdots, H_6, v \in \mathcal{C}^1(\mathbb{R}_+;\mathcal{S}(\mathbb{R}^d)),\qquad \|H_i\|_{L^\infty_{t, x}}\leq \frac{1}{2}h_i, \qquad i=1, 2, \cdots 6.
\end{align}
Let $-\frac{d}{2}< s_1 \leq \frac{d}{2}-1$ and  $s_1\leq s_2-1\leq s_1+1.$  
Let $(w, r, u)$ be a solution of system \eqref{L} on the time interval $[0, T)$.
There exists a    positive constant $c_0$ independent of $\mu, \lambda$ such that if 
\begin{align}\label{A2}
\sum_{i=1}^6 \| H_i\|_{L^\infty_t(B^{\frac{d}{2}-1}\cap B^{\frac{d}{2}+1})}\leq c_0,
\end{align}
then the following estimate holds on $[0, T):$
\begin{multline}\label{L-H-total}
 \|(w, r, u)\|_{\widetilde{L}_t^\infty (B^{s_1})}^{\ell}+\|(w, r, u)\|_{\widetilde{L}_t^\infty (B^{s_2})}^{h}+  {\kappa } \,\Bigl(\|(w, r, u) \|_{ {L}^1_t(B^{s_1+2})}^{\ell}+\|(w, r, u) \|_{ {L}^1_t(B^{s_2})}^{h} \Bigr) \\
 +\int_0^t(\|(\partial_t w, \frac{w}{\nu})\|_{B^{s_1}}^\ell+ \|(\partial_t w, \frac{w}{\nu})\|_{B^{s_2-1}}^h)+
\int_0^t(\|(\partial_t u,  \eta u, \partial_t r)\|_{B^{s_1+1}}^\ell+\|(\partial_t u, \eta u, \partial_t r)\|_{B^{s_2-1}}^h)  \\
\qquad+\int_0^t (\|\mu\Delta u, (\mu+\lambda)\nabla\div u\|_{B^{s_1+1}}^\ell+ \|\mu\Delta u, (\mu+\lambda)\nabla\div u\|_{B^{s_2-1}}^h) \\
\leq \exp\Bigl(C\bigl(H(t)+V(t)\bigr)\Bigr)\Bigl(\|(w_0, r_0, u_0)\|_{B^{s_1}\cap B^{s_2}}+\int_0^t\,  \|(S_2, S_3, S_4)(\tau)\|_{B^{s_1}\cap B^{s_2}} \,\Bigr).
\end{multline}
where $\displaystyle V(t):=\int_0^t \|v(\tau)\|_{B^{\frac{d}{2}}\cap{B^{\frac{d}{2}+1}}}\,$,  $\displaystyle H(t):=\sum_{i=1}^6\|\partial_t H_i(t)\|_{B^{\frac{d}{2}}}$ and $\kappa$ is a constant defined by \eqref{def-kappa}.
\end{proposition}

\begin{remark}\label{Re-Prop-L}
It turns out that we have to  choose $c_0<<\operatorname{min}\{\eta,\dfrac{1}{\eta}\}$, therefore we are not able to take into account the case when $\eta\to\infty$ in the same time as the relaxation parameter $\nu\to 0$. This is due to the overdamping phenomena that  described in e.g. \cite{SlideZuazua}.
\end{remark}

We will now state a global-in-time existence result for \eqref{III}. 
First, let us introduce the functional spaces which appear in the global
existence theorem and the rest of the paper :
\begin{align}\label{SolutionspaceE}E^{s_1,s_2}_T\triangleq\{(y,w,r,u)\in C_b([0,T],B^{s_1}\cap B^{s_2}),\;\;\;(w,r,u)^h\in L^{1}([0,T]
,B^{s_2}),\,\;\;\;r^\ell\in L^{1}([0,T]
,B^{s_1+2}),\notag\\\quad(\dfrac{w}{\nu},\partial_tw)^\ell\in L^{1}([0,T]
,B^{s_1}),\,\,(\dfrac{w}{\nu},\partial_tw)^h\in L^{1}([0,T]
,B^{s_2-1}),\,\,(\eta u,\partial_tr,\partial_tu)^\ell\in L^{1}([0,T]
,B^{s_1+1})\notag\\\,\,(\eta u,\partial_tr,\partial_tu)^h\in L^{1}([0,T]
,B^{s_2-1}),\,\,(\mu\Delta u,(\mu+\lambda)\nabla\div u)^\ell\in L^1([0,T], B^{s_1+1}\,\notag\\\text{and}\, (\mu\Delta u,(\mu+\lambda)\nabla\div u)^h\in L^1([0,T], B^{s_2-1}\}.
\end{align}
We define $\|\cdot \|_{E^{s_1,s_2}_T}$ as the norm associated to $E^{s_1,s_2}_T$ and if $T = +\infty$, we use the notation $E^{s_1,s_2}$ and replace the interval $[0,T]$ by $\mathbb{R}_+$.

We have the following theorem.
   \begin{theorem}\label{Thm-III}
   Let $d\geq 2$ and assume that the parameters satisfy $\mu\geq0, \lambda+\mu\geq 0, 0<\nu\leq 1$ and  $\eta\geq1$. Let the constants $ \bar{\alpha}_{\pm}\in\left(  0,1\right)  ,\bar{\rho
}_{\pm}>0$ satisfying \eqref{sum=1} and  \eqref{constants_at_infinity_2}.
There exists a constant $c>0$ independent of the  viscosity coefficients $\mu,\lambda$ such
that for any initial data such that
\[
\left\Vert (y_{0},w_{0},r_{0},u_{0})\right\Vert _{B^{\frac{d}%
{2}-1}\cap{B^{\frac{d}{2}+1}}}\leq c
\]
then System \eqref{III} admit a unique global-in-time solution
$(y,w,r,u)$ in the space $E^{\frac{d}{2}-1,\frac{d}{2}+1}$.
Moreover, the following estimate holds true uniformly w.r.t. the viscosity
coefficients $\mu$ and $\lambda$:

\begin{align} \label{X2Th4}
\left\Vert \left(y,w,r,u\right)  \right\Vert  &  _{L%
^{\infty}(B^{\frac{d}{2}-1}\cap B^{\frac{d}{2}+1}%
)}+\left\Vert (w,r,u)\right\Vert_{L^{1}(B^{\frac
{d}{2}+1})}+\|\frac{w}{\nu}\|_{L^{1}(B^{\frac{d}{2}-1}\cap B^{\frac{d}{2}})}+\|\eta u\|_{L^{1}(B^{\frac{d}{2}})} \leq Cc.
\end{align}
\end{theorem}

For $d\geq3$, using Proposition \ref{Composition} related to composition operator, Theorem \ref{Thm-III} directly implies Theorem \ref{Th-2}.
In the two-dimensional setting, however, the  Proposition \ref{Composition} fails to work as the regularity index is equal to 0 and therefore one must be careful when trying to recover the regularity properties for the original unknowns. Let us explain how to proceed to  this issue.
For example, concerning $\alpha_+-\bar\alpha$, we have $\alpha_+-\bar\alpha=\Psi_1(w,R,Y)-\Psi_1(0, \bar P, \dfrac{\bar\alpha_+\bar\rho_+}{\bar\alpha_+\bar\rho_++\bar\alpha_-\bar\rho_-})$.
Using the decomposition 
\begin{multline*}
\qquad\quad\Psi_1(w,R,Y)-\Psi_1(0, \bar P, \dfrac{\bar\alpha_+\bar\rho_+}{\bar\alpha_+\bar\rho_++\bar\alpha_-\bar\rho_-})\\
=\bigl(\partial_{w}\Psi_1(0, \bar P, \dfrac{\bar\alpha_+\bar\rho_+}{\bar\alpha_+\bar\rho_++\bar\alpha_-\bar\rho_-}) + G^1(w,r,y)\bigr) w + \bigl(\partial_{R}\Psi_1(0, \bar P, \dfrac{\bar\alpha_+\bar\rho_+}{\bar\alpha_+\bar\rho_++\bar\alpha_-\bar\rho_-}) + G^2(w,r,y)\bigr) r\\  \quad+\bigl(\partial_{Y}\Psi_1(0, \bar P, \dfrac{\bar\alpha_+\bar\rho_+}{\bar\alpha_+\bar\rho_++\bar\alpha_-\bar\rho_-}) + G^3(w,r,y)\bigr) y,
\end{multline*}
where $G^1,G^2,G^3$ are smooth functions vanishing at $(0,0,0).$ We get, thanks to product law \eqref{eq:prod2}  that
$$\|\alpha_+-\bar\alpha_+\|_{B^0}\leq C(1+\|(G^1,G^2,G^3)(w,r,y)\|_{B^{1}}) \|w,r,y\|_{B^{0}}$$
And by  Proposition \ref{Composition} we see  $\|(G^1,G^2,G^3)(w,r,y)\|_{B^{1}}$ is under control,   thus we can recover  estimate for $\alpha_+-\bar\alpha_+$ in the space  $\mathcal{C}_{b}(\mathbb{R}_+; B^0).$   Doing similar arguments for $\alpha_-, \rho_+$ and $\rho_-$, one can finally deduce Theorem \ref{Th-2} and also Theorem \ref{Th-1} in the case $d=2$.
\bigbreak

\section{Analysis of the linear  system (3.17)}
This section is devoted to the proof of Proposition \ref{Prop-L} for the linear system \eqref{L}.
We  first localize \eqref{L}  in frequencies thanks to the Littlewood-Paley decomposition,  then use a renormalized energy method to estimate each dyadic block. In the following computations,   assume that we are given a smooth solution $(y, w, r, u)$ of \eqref{L} on $[0,T)\times \mathbb{R}^d$.  
      And $\left(  q_{j}\right)  _{j\in\mathbb{Z}}$ is a generic
sequence such that%
\[
\left\Vert \left(  q_{j}\right)  _{j\in\mathbb{Z}}\right\Vert _{\ell
^{1}\left(  \mathbb{Z}\right)  }\leq1.
\]

We will consider the following energy functional,
\begin{align*} 
\mathcal{L}_{j}\left(  t\right)  &  :=\sqrt{\int
_{\mathbb{R}^{d}}\Bigl(\frac{h_6 }{h_1 }w_j^2+\frac{h_5 }{h_3 } r_j^2 +|u_j|^2+2\varepsilon_\ell u_j\cdot\nabla r_j\Bigr)} \quad{\rm{for }}\quad j\leq 0
\end{align*}
and
\begin{align*} 
\mathcal{L}_{j}\left(  t\right)  &  :=\sqrt{\int
_{\mathbb{R}^{d}}\Bigl(\frac{h_6+H_6}{h_1+H_1}w_j^2+\frac{h_5+H_5}{h_3+H_3} r_j^2 +|u_j|^2+2\varepsilon_h 2^{-2j} u_j\cdot\nabla r_j\Bigr)}\quad {\rm{for }}\quad j> 0
\end{align*}
where  $\varepsilon_h, \varepsilon_l>0$  are constants that will be fixed later on such that
\begin{align}
|\mathcal{L}_j(t)|^2\sim \|(w_j, r_j, u_j)(t)\|_{L^2}^2.\label{equiv-norm}
\end{align}

The next two steps  will explain how we constructed those two functionals to derive a priori estimates in low and high frequencies, respectively.   

\subsection{Low frequencies analysis}
Throughout this part, we shall suppose that  $j\leq0.$    
Using spectral localization properties, that is 
\begin{align*}
 \|\nabla r_j\|_{L^2}\leq \frac{12}{5}2^j \|r_j\|_{L^2} \leq \frac{12}{5}  \|r_j\|_{L^2},
\end{align*}
we easily obtain that as soon as $\varepsilon_\ell\leq  \min\{\dfrac{5h_5}{24h_3}, \dfrac{5}{24}\}$ then \eqref{equiv-norm} is satisfied in the following way
\begin{equation}\label{equiv-norm-l}
 \frac{C_1^2}{4}\|(w_j, r_j, u_j)(t)\|_{L^2}^2\leq\mathcal{L}_j^2(t)\leq 4C_2^2\|(w_j, r_j, u_j)(t)\|_{L^2}^2,
\end{equation}
where $C_1=\min\{\dfrac{h_6 }{h_1 }, \dfrac{h_5 }{h_3 }, 1\},~C_2=\max\{\dfrac{h_6 }{h_1 }, \dfrac{h_5 }{h_3 }, 1\}$ 
and note that 
\begin{align}
C_1\leq\dfrac{h_6}{h_1},~~\dfrac{h_5}{h_3},~1\leq C_2.\label{def-C1C2}
\end{align}
In the sequel we will use $C_1$ and $C_2$ frequently.
 Indeed, one has
\begin{align*}
 \mathcal{L}_j^2&= {\int
_{\mathbb{R}^{d}}\Bigl(\frac{h_6 }{h_1 }w_j^2+\frac{h_5 }{h_3 } r_j^2 +|u_j|^2+2\varepsilon_\ell u_j\cdot\nabla r_j\Bigr)} \notag\\
&\geq \frac{h_6 }{h_1 }\|w_j\|_{L^2}^2+\frac{h_5 }{h_3 } \|r_j\|_{L^2}^2 +\|u_j\|^2-\frac{12}{5}\varepsilon_\ell (\|u_j\|_{L^2}^2+\| r_j\|_{L^2}^2)\notag\\
&\geq \min\{\dfrac{h_6 }{h_1 }, \dfrac{h_5 }{2h_3 }, \frac{1}{2}\}\|(w_j, r_j, u_j)\|_{L^2}^2\geq \frac{C_1^2}{4}\|(w_j, r_j, u_j)\|_{L^2}^2
\end{align*}
and
\begin{align*}
 \mathcal{L}_j^2&\leq \frac{h_6 }{h_1 }\|w_j\|_{L^2}^2+\frac{h_5 }{h_3 } \|r_j\|_{L^2}^2 +\|u_j\|^2+\frac{12}{5}\varepsilon_\ell (\|u_j\|_{L^2}^2+\| r_j\|_{L^2}^2)\notag\\
&\leq  \max\{\dfrac{h_6 }{h_1 }, \dfrac{3h_5 }{2h_3 }, \dfrac{3}{2}\}\|(w_j, r_j, u_j)\|_{L^2}^2\leq 4C_2^2\|(w_j, r_j, u_j)\|_{L^2}^2.
\end{align*}

Applying the operator $\dot{\Delta}_{j}$ to the  three equations of
\eqref{L} yields%
\begin{equation}
\left\{
\begin{array}
[c]{l}%
\partial_{t}w_{j}+v\cdot\nabla w_j+h_1 \operatorname{div}%
u_{j}+h_2\dfrac{w_j}{\nu}=\dot{\Delta}_{j} S_2- K_{j}^{1},\\
\partial_{t}r_{j}+v\cdot\nabla r_j+h_3\operatorname{div}%
u_{j}=\dot{\Delta}_{j} S_3-K_{j}^{2},\\
\partial_{t}u_{j}+v\cdot\nabla u_j-h_4\mathcal{A}_{\mu,\lambda}%
u_{j}+  \eta u_{j}+h_5\nabla r_j+h_6\nabla w_{j}=\dot{\Delta}_{j} S_4- K_{j}^{3},
\end{array}
\right.  \label{ESyst:1}%
\end{equation}
where $K_{j}^{1}, K_{j}^{2}, K_{j}^{3}$ will act as source terms
\[
\left\{
\begin{array}
[c]{l}%
K_{j}^{1}=[\ddj, v]\nabla w+\ddj(H_1\,\div u)+\dfrac{1}{\nu}\ddj(H_2 \,w)\\
K_{j}^{2}=[\ddj, v]\nabla r+\ddj(H_3\,\div u),\\
K_{j}^{3}=[\ddj, v]\nabla u-\ddj(H_4\,\mathcal{A}_{\mu,\lambda}%
\,u)+\ddj(H_5\, \nabla r)+\ddj(H_6 \,\nabla w).
\end{array}
\right.
\]

Multiplying the first equation of \eqref{ESyst:1} with $\dfrac{h_6}{h_1} w_j$, the second equation with $\dfrac{h_5}{h_3}r_j$, the last one with $u_j,$ respectively, we get
\begin{align*}
\frac{1}{2}\frac{h_6 }{h_1 } \frac{d}{dt}\|w_j\|_{L^2}^2+h_6\int_{\mathbb{R}^d} w_j \,\div u_j+\frac{h_2h_6}{h_1}\frac{\|w_j\|_{L^2}^2}{\nu}&\leq \frac{h_6}{h_1}\|w_j\|_{L^2}(\|\dot{\Delta}_{j} S_2\|_{L^2}+ \|K_{j}^{1}\|_{L^2})\\&\quad+\frac{h_6}{h_1}\|w_j\|_{L^2}^2\|\div v\|_{L^\infty},
\end{align*}
\begin{align*}
\frac{1}{2}\frac{h_5 }{h_3 } \frac{d}{dt}\|r_j\|_{L^2}^2+ h_5\int_{\mathbb{R}^d} r_j \,\div u_j\leq\frac{h_5}{h_3}\|r_j\|_{L^2}(\|\dot{\Delta}_{j} S_3\|_{L^2}+ \|K_{j}^{2}\|_{L^2}) +\frac{h_5}{h_3}\|r_j\|_{L^2}^2\|\div v\|_{L^\infty}
\end{align*}
and
\begin{align*}
\frac{1}{2}\frac{d}{dt} &\|u_j\|_{L^2}^2+h_4\bigl(\mu \|\nabla u_j\|_{L^2}^2+ (\mu+\lambda)\|\div u_j\|_{L^2}^2\bigr)+ \eta\|u_j\|_{L^2}^2-h_5 \int_{\mathbb{R}^{d}}r_j\,\div u_j-h_6 \int_{\mathbb{R}^{d}}w_j\,\div u_j\\
\leq& \|u_j\|_{L^2}(\|\dot{\Delta}_{j} S_4\|_{L^2}+ \|K_{j}^{3}\|_{L^2})+\|u_j\|_{L^2}^2\|\div v\|_{L^\infty}.
\end{align*}

Summing up the resulting inequalities together, we obtain
\begin{align}
\frac{1}{2}\frac{d}{dt} \Bigl(\frac{h_6 }{h_1 }& \|w_j\|_{L^2}^2+ \frac{h_5 }{h_3 } \|r_j\|_{L^2}^2+ \|u_j\|^2\Bigr)+ \frac{h_2h_6}{h_1}\frac{\|w_j\|_{L^2}^2}{\nu}+h_4\bigl(\mu \|\nabla u_j\|_{L^2}^2+ (\mu+\lambda)\|\div u_j\|_{L^2}^2\bigr)+ \eta\|u_j\|_{L^2}^2\notag\\
\leq &  C_2\|w_j\|_{L^2}(\|\dot{\Delta}_{j} S_2\|_{L^2}+ \|K_{j}^{1}\|_{L^2})+ C_2\|r_j\|_{L^2}(\|\dot{\Delta}_{j} S_3\|_{L^2}+ \|K_{j}^{2}\|_{L^2})\notag\\
&\quad\quad\quad+\|u_j\|_{L^2}(\|\dot{\Delta}_{j} S_4\|_{L^2}+ \|K_{j}^{3}\|_{L^2})+  C_2\|(w_j, r_j, u_j)\|_{L^2}^2\|\div v\|_{L^\infty}.\label{low-basic-es1}
\end{align}

\noindent We observe that the left-hand side of \eqref{low-basic-es1} does not provide any decay information for $r$. Recovering such information is our objective in the following lines. Taking the gradient of the second equation in \eqref{ESyst:1},  one  find that $\nabla r_j$ verifies
\begin{equation*}
\partial_{t}\nabla r_{j}  +h_3\nabla \div u_j =-\nabla(v\cdot\nabla r_j)+\nabla (\ddj S_3-  K_{j}^{2}).
\end{equation*}
Multiplying this equation with $u_{j}$, testing the equation of $u_{j}$
from $\left(  \text{\ref{ESyst:1}}\right)  $ with $\nabla r_{j}$ and summing
up the results we end up with%
\begin{align}
 \frac{d}{dt}\int_{\mathbb{R}^{d}}& (u_{j}\cdot\nabla r_{j})%
+ h_5 \|\nabla r_j\|_{L^2}^2\notag\\
 \leq & \, h_3\|\div u_j\|_{L^2}^2+\|\div u_j\|_{L^2}\|v\cdot\nabla r_j\|_{L^2}+h_4\|\mathcal{A}_{\mu, \lambda} u_j\|_{L^2}\|\nabla r_j\|_{L^2}\notag\\&+ \eta\|u_j\|_{L^2}\|\nabla r_j\|_{L^2}+h_6\|\nabla w_j\|_{L^2}\|\nabla r_j\|_{L^2}\notag+ \|v\cdot\nabla u_j\|_{L^2}\|\nabla r_j\|_{L^2}\\&+\|\div u_j\|_{L^2}(\|\dot{\Delta}_{j} S_3\|_{L^2}+ \|K_{j}^{2}\|_{L^2})+\|\nabla r_j\|_{L^2}(\|\dot{\Delta}_{j} S_4\|_{L^2}+ \|K_{j}^{3}\|_{L^2}).\label{low-cross-es}
\end{align}
Recall definition of $\mathcal{L}_j,$ by combining \eqref{low-basic-es1}  and \eqref{low-cross-es} one has
\begin{align}
\frac{1}{2}\frac{d}{dt}& \mathcal{L}_j^2+ \frac{h_2h_6}{h_1}\frac{\|w_j\|_{L^2}^2}{\nu}+h_5\varepsilon_\ell\|\nabla r_j\|_{L^2}^2+ \eta\|u_j\|_{L^2}^2+h_4\bigl(\mu \|\nabla u_j\|_{L^2}^2+ (\mu+\lambda)\|\div u_j\|_{L^2}^2\bigr)\notag\\
\leq& C_2\|w_j\|_{L^2}(\|\dot{\Delta}_{j} S_2\|_{L^2}+ \|K_{j}^{1}\|_{L^2})+ C_2(\|r_j\|_{L^2}+\varepsilon_{\ell}\|\div u_j\|_{L^2})(\|\dot{\Delta}_{j} S_3\|_{L^2}+ \|K_{j}^{2}\|_{L^2})\notag\\
&+  (\|u_j\|_{L^2}+\varepsilon_\ell\|\nabla r_j\|_{L^2}) (\|\dot{\Delta}_{j} S_4\|_{L^2}+ \|K_{j}^{3}\|_{L^2})+ C_2\|(w_j, r_j, u_j)\|_{L^2}^2\|\div v\|_{L^\infty}\notag\\
&+\varepsilon_\ell\Bigl( h_3\|\div u_j\|_{L^2}^2+h_4\|\mathcal{A}_{\mu, \lambda} \,u_j\|_{L^2}\|\nabla r_j\|_{L^2}+\eta \|u_j\|_{L^2}\|\nabla r_j\|_{L^2}+h_6\|\nabla w_j\|_{L^2}\|\nabla r_j\|_{L^2}\Bigr)\notag\\
&+\varepsilon_\ell(\|\div u_j\|_{L^2}\|v\cdot\nabla r_j\|_{L^2}+\|\nabla r_j\|_{L^2}\|v\cdot\nabla u_j\|_{L^2})\label{low-total-energy}.
\end{align}
 
At this stage, we are ready to estimate the right-hand side of \eqref{low-total-energy}.
Using \eqref{equiv-norm-l} and $ C_1\leq 1\leq C_2$, we have
\begin{align*}
C_2&\|w_j\|_{L^2}(\|\dot{\Delta}_{j} S_2\|_{L^2}+ \|K_{j}^{1}\|_{L^2})+ C_2\|r_j\|_{L^2}(\|\dot{\Delta}_{j} S_3\|_{L^2}+ \|K_{j}^{2}\|_{L^2})+\|u_j\|_{L^2} (\|\dot{\Delta}_{j} S_4\|_{L^2}+ \|K_{j}^{3}\|_{L^2})\\
\leq & C_2 \|(w_j, r_j, u_j)\|_{L^2}\,\|(\dot{\Delta}_{j} S_2,~    \dot{\Delta}_{j} S_3,~\dot{\Delta}_{j} S_4, ~  K_{j}^{1},~K_{j}^{2},~K_{j}^{3})\|_{L^2}\\
\leq &  \frac{2C_2}{ {C_1}}\, \mathcal{L}_j\, \|(\dot{\Delta}_{j} S_2,~    \dot{\Delta}_{j} S_3,~\dot{\Delta}_{j} S_4, ~  K_{j}^{1},~K_{j}^{2},~K_{j}^{3})\|_{L^2}.
\end{align*}
Since $\varepsilon_{\ell}\leq \frac{5}{24}$, $j\leq 0$, we have
\begin{align*}
C_2& \varepsilon_{\ell} \|\div u_j\|_{L^2}(\| \dot{\Delta}_{j} S_3\|_{L^2}+ \|K_{j}^{2}\|_{L^2})+  \varepsilon_{\ell} \|\nabla r_j\|_{L^2}(\| \dot{\Delta}_{j} S_4\|_{L^2}+ \|K_{j}^{3}\|_{L^2})\\
 \leq& C_2 \varepsilon_{\ell}\frac{12}{5}\|(u_j, r_j)\|_{L^2}(\dot{\Delta}_{j} S_3, ~\dot{\Delta}_{j} S_4,~K_{j}^{2},~K_{j}^{3})\|_{L^2}
\leq  \frac{C_2} { {C_1}}\,  \mathcal{L}_j \,  \|(\dot{\Delta}_{j} S_3, ~\dot{\Delta}_{j} S_4,~K_{j}^{2},~K_{j}^{3})\|_{L^2}. 
\end{align*}
Obviously, using again \eqref{equiv-norm-l} we obtain
\begin{align*}
 C_2\|(w_j, r_j, u_j)\|_{L^2}^2\|\div v\|_{L^\infty}\leq \frac{4C_2}{C_1^2}\mathcal{L}_j^2\,\|\div v\|_{L^\infty}.
\end{align*}
Owing to the fact that $j\leq 0,$ we use the following inequalities  
\begin{align}
\|\nabla u_j\|_{L^2}\leq \frac{12}{5}\|u_j\|_{L^2},\quad \|\nabla r_j\|_{L^2}\leq \frac{12}{5}\|r_j\|_{L^2},\quad \|\nabla w_j\|_{L^2}\leq \frac{12}{5}\|w_j\|_{L^2},\label{low-localized}
\end{align}
and Young's inequality to write
\begin{align*}
\varepsilon_\ell h_3\|\div u_j\|_{L^2}^2&\leq \varepsilon_\ell h_3(\frac{12}{5})^2\|u_j\|_{L^2}^2\leq  6\,h_3\varepsilon_\ell \|u_j\|_{L^2}^2,\\
\varepsilon_\ell\eta \|u_j\|_{L^2}\|\nabla r_j\|_{L^2}&\leq \dfrac{8\eta^2}{h_5} \varepsilon_\ell\|u_j\|_{L^2}^2+ \frac{h_5}{8}\varepsilon_\ell\|\nabla r_j\|_{L^2}^2,\\
\varepsilon_\ell h_6\|\nabla w_j\|_{L^2}\|\nabla r_j\|_{L^2}
&\leq \dfrac{8h_6^2}{h_5} \varepsilon_\ell \|\nabla w_j\|_{L^2}^2+  \frac{h_5}{8}\varepsilon_\ell\|\nabla r_j\|_{L^2}^2\\
&\leq \dfrac{48h_6^2}{h_5} \varepsilon_\ell \| w_j\|_{L^2}^2+  \frac{h_5}{8}\varepsilon_\ell\|\nabla r_j\|_{L^2}^2.
\end{align*}
Using once more \eqref{equiv-norm-l} and $\varepsilon_{\ell}\leq \frac{5}{24}$, we arrive at
\begin{align*}
\varepsilon_\ell\|\div u_j\|_{L^2}\|v\cdot\nabla r_j\|_{L^2}&\leq \varepsilon_\ell(\frac{12}{5})^2 \|u_j\|_{L^2} \|r_j\|_{L^2} \|v\|_{L^\infty}\leq \frac{5}{C_1^2}\, \mathcal{L}_j^2 \,\|v\|_{L^\infty},\\
\varepsilon_\ell\|\nabla r_j\|_{L^2}\|v\cdot\nabla u_j\|_{L^2}&\leq \varepsilon_\ell(\frac{12}{5})^2 \|u_j\|_{L^2} \|r_j\|_{L^2}\|v\|_{L^\infty}
\leq  \frac {5}{C_1^2} \,  \mathcal{L}_j^2\,\|v\|_{L^\infty}.
\end{align*}
Similar arguments lead to the following estimate:  
$$\|\mathcal{A}_{\mu, \lambda} \,u_j\|_{L^2}\leq \frac{12}{5}\mu\|\nabla u_j\|_{L^2}+ \frac{12}{5}(\mu+\lambda)\|\div u_j\|_{L^2},$$
 then, using the fact that $0\leq\max\{\mu, \mu+\lambda\}\leq \nu$,  we readily have
 $$\|\mathcal{A}_{\mu, \lambda}\, u_j\|_{L^2}^2\leq 12\,\nu\,(\mu\|\nabla u_j\|_{L^2}^2+  (\mu+\lambda)\|\div u_j\|_{L^2}^2),$$
 and thus
\begin{align*}
\varepsilon_\ell h_4 \|\mathcal{A}_{\mu, \lambda}\, u_j\|_{L^2}\|\nabla r_j\|_{L^2}&\leq \varepsilon_\ell (\frac{8}{h_5}h_4^2 \|\mathcal{A}_{\mu, \lambda}\, u_j\|_{L^2}^2+\frac{h_5}{8}\|  \nabla r_j\|_{L^2}^2)\\
&\leq \dfrac{96h_4^2\nu}{h_5}\varepsilon_\ell \bigl(\mu\|\nabla u_j\|_{L^2}^2+(\mu+\lambda)\|\div u_j\|_{L^2}^2\bigr)+\frac{h_5}{8}\varepsilon_\ell\|r_j\|_{L^2}^2.
\end{align*}
Inserting above estimates into \eqref{low-total-energy}  we are led to
\begin{align}
\frac{1}{2}\frac{d}{dt}&\mathcal{L}_j^2+ (\frac{h_2h_6}{c_1}-\frac{48h_6^2}{h_5}\nu\varepsilon_\ell)\frac{\|w_j\|_{L^2}^2}{\nu}+\frac{h_5}{2}\varepsilon_l \|\nabla r_j\|_{L^2}^2+(\eta-6 h_3\varepsilon_\ell-\frac{8\eta^2}{h_5} \varepsilon_\ell)\|u_j\|_{L^2}^2\notag\\
+&(h_4-\frac{96\,h_4^2\nu}{h_5}\varepsilon_\ell)\,(\mu \|\nabla u_j\|_{L^2}^2+ (\mu+\lambda)\|\div u_j\|_{L^2}^2)\notag\\
\leq & \frac{3C_2}{ {C_1}}\, \mathcal{L}_j\, \|(\dot{\Delta}_{j} S_2,~    \dot{\Delta}_{j} S_3,~\dot{\Delta}_{j} S_4, ~  K_{j}^{1},~K_{j}^{2},~K_{j}^{3})\|_{L^2}+ \frac{14C_2}{C_1^2}\mathcal{L}_j^2\,\bigl(\|\div v\|_{L^\infty} +  \|v\|_{L^\infty}\bigr)
.\label{low-total-energy-es1}
\end{align}

Let us choose $\varepsilon_\ell>0$ such that
\begin{align*}
\varepsilon_\ell&\leq  \min\{\frac{5h_5}{24h_3}, \frac{5}{24}\},\\
\frac{h_2h_6}{2h_1}&\leq\frac{h_2h_6}{h_1}-\frac{48h_6^2}{h_5}\nu\varepsilon_\ell,\\
\frac{\eta}{2}&\leq \eta-6h_3\varepsilon_\ell-\frac{8\eta^2}{h_5}\varepsilon_\ell,\\
\frac{h_4}{2}&\leq h_4-\frac{96h_4^2}{h_5}\nu\varepsilon_\ell
\end{align*}
for example one may take  
\begin{align*}
\varepsilon_\ell=\frac{1}{192}\min\{\frac{h_5}{h_3}, ~1,~ \frac{h_2h_5}{h_1h_6\,\nu}, ~\frac{h_5 \eta }{h_3h_5+\eta^2 }, ~\frac{h_5}{h_4\,\nu}\}.
\end{align*}
Under the consideration of such $\varepsilon_\ell$, we further define  
\begin{align}
\kappa:= \frac{1}{4}\min\{\dfrac{h_2h_6}{h_1\,\nu}, ~ {h_5\,\varepsilon_\ell},~ {\eta}\}~~
(\geq \min\{\dfrac{h_2h_6}{2h_1\,\nu}, ~\dfrac{h_5\varepsilon_\ell}{2}(\frac{5}{6})^2,~\frac{\eta}{2}\}).\label{def-low-damping-constant}
\end{align}
Combining the fact that $\dfrac{5}{6}\,2^{j} \|  r_j\|_{L^2}\leq\|\nabla r_j\|_{L^2}$, the fact that for all $\,j\leq 0,~2^{2j}\|(w_j, u_j)\|_{L^2}^2\leq\|(w_j, u_j)\|_{L^2}^2,$ and \eqref{equiv-norm-l}, one is able to rewrite \eqref{low-total-energy-es1} as
\begin{align}
\frac{1}{2}\frac{d}{dt}& \mathcal{L}_j^2+\frac{\kappa}{4C_2^2}\,2^{2j}\,\mathcal{L}_j^2\notag\\
\leq&\frac{3C_2}{ {C_1}} \,\mathcal{L}_j\, \|(\dot{\Delta}_{j} S_2,~    \dot{\Delta}_{j} S_3,~\dot{\Delta}_{j} S_4, ~  K_{j}^{1},~K_{j}^{2},~K_{j}^{3})\|_{L^2}+\frac{14C_2}{C_1^2}\mathcal{L}_j^2\, \bigl(\|\div v\|_{L^\infty} +   \|v\|_{L^\infty}\bigr). \label{low-total-energy-es2}
\end{align}
 Owing to Proposition \ref{Productlaw} we know  that the  product  is continuous from  (notice that $s_1\in(-\frac{d}{2}, \frac{d}{2}-1]$)
 \begin{align} \label{p-law-40}
 {B^{\frac{d}{2}}}\times {B^{s_1}}\quad{\rm{to}} \quad{B^{s_1}} \quad \text{and} \quad
 {B^{\frac{d}{2}-1}}\times {B^{s_1+1}}\quad{\rm{to}} \quad{B^{s_1}}
 \end{align}
Then, using the commutator estimate \eqref{eq:com1} and \eqref{p-law-40}, we obtain that (notice that $s_1\leq s_2-1$)
\begin{align*}
\|K^1_j\|_{L^2}&\leq C2^{-s_1j}q_j \Bigl(\|\nabla v\|_{B^{\frac{d}{2}}}\|\nabla w\|_{B^{s_1-1}}+ \|H_1\,\div u\|_{B^{s_1}}+\frac{1}{\nu}\|H_2\, w\|_{B^{s_1}}\Bigr)\\
&\leq C2^{-s_1j}q_j \Bigl(\|v\|_{B^{\frac{d}{2}+1}}\|w\|_{B^{s_1}}+\|H_1\|_{B^{\frac{d}{2}}}\|  u\|_{B^{s_1+1}}+ \|H_2\|_{B^{\frac{d}{2}}}\|\frac{w}{\nu}\|_{B^{s_1}}\Bigr)\\
\|K^2_j\|_{L^2}&\leq C 2^{-s_1j}q_j \Bigl(\|\nabla v\|_{B^{\frac{d}{2}}}\|\nabla r\|_{B^{s_1-1}}+ \|H_3\,\div u\|_{B^{s_1}}\Bigr)\\
&\leq C2^{-s_1j}q_j \Bigl(\|v\|_{B^{\frac{d}{2}+1}}\|r\|_{B^{s_1}}+\|H_3\|_{B^{\frac{d}{2}}}\|  u\|_{B^{s_1+1}}\Bigr).
\end{align*}

The term $\ddj(H_5\,\nabla r)$ appearing in $K^3_j$ deserves some particular attention. Using \eqref{p-law-40} we observe that the product maps  
$$B^{\frac{d}{2}+1+s_1-s_2}\times B^{s_2-1}\hookrightarrow B^{s_1}~(\,{\rm{note~that}}~\frac{d}{2}+1+s_1-s_2\leq \frac{d}{2}),$$ 
 and we infer that
 \begin{align}
 \|\ddj(H_5\,\nabla r)\|_{L^2}&\leq  C2^{-s_1j}q_j\Bigl( \| H_5\,\nabla r^{\ell}\|_{B^{s_1}}+ \| H_5\,\nabla r^{h}\|_{B^{s_1}}\Bigr)\notag\\
 &\leq C2^{-s_1j}q_j \Bigl(\|H_5\|_{B^{\frac{d}{2}-1}}\|\nabla r^{\ell}\|_{B^{s_1+1}}+  \|H_5\|_{B^{\frac{d}{2}+1+s_1-s_2}}\|\nabla r^{h}\|_{B^{s_2-1}}\Bigr)\notag\\
 &\leq C2^{-s_1j}q_j \Bigl(\|H_5\|_{B^{\frac{d}{2}-1}}\|r\|_{B^{s_1+2}}^{\ell}+  \|H_5\|_{B^{\frac{d}{2}+1+s_1-s_2}} \| r\|_{B^{s_2}}^h\Bigr).
 \end{align}
In conclusion, using once more the commutator estimate \eqref{eq:com1} and \eqref{p-law-40}, we obtain
\begin{align*}
\|K^3_j\|_{L^2}&\leq C2^{-s_1j}q_j \Bigl(\|v\|_{B^{\frac{d}{2}+1}}\|u\|_{B^{s_1}}+\|H_4\|_{B^{\frac{d}{2}}}\|\mathcal{A}_{\mu, \lambda} u\|_{B^{s_1}}+\|H_6\|_{B^{\frac{d}{2}}}\|\nabla w\|_{B^{s_1}}\\
&\hspace*{5cm}+(\|H_5\|_{B^{\frac{d}{2}-1}}\|r\|_{B^{s_1+2}}^{\ell}+  \|H_5\|_{B^{\frac{d}{2}+1+s_1-s_2}} \| r\|_{B^{s_2}}^h)\Bigr).
\end{align*}

Define   $\mathcal{L}^{\ell}_j=2^{s_1j}\,\mathcal{L}_j.$  Recall assumption \eqref{A2} and  by interpolation, we    write   (notice that $\nu<1$)  
\begin{align*}
&2^{s_1j}\|(K^1_j, K^1_j, K^3_j)\|_{L^2}\\
\leq& C q_j \Bigl(\|\nabla v\|_{B^{\frac{d}{2}}}\|(w, r, u)\|_{B^{s_1}} +c_0\bigl(\|\frac{w}{\nu}\|_{B^{s_1}}+\|w\|_{B^{s_1+1}} + \|u\|_{B^{s_1+1}}+ \|\mathcal{A}_{\mu, \lambda}\,u\|_{B^{s_1}} + \|r\|_{B^{s_1+2}}^{\ell}+  \| r\|_{B^{s_2}}^h\bigr)\Bigr)\\
\leq&  C K \|\nabla v\|_{B^{\frac{d}{2}}}\mathcal{L}^{\ell}_j  + C \|\nabla v\|_{B^{\frac{d}{2}}}\Bigl(q_j\|(w, r, u)\|_{B^{s_1}} -K\mathcal{L}^{\ell}_j\Bigr)\\&+  C c_0\,q_j \Bigl(\|\frac{w}{\nu}\|_{B^{s_1}} +\|(w, u)\|_{B^{s_1+1}}+ \|\mathcal{A}_{\mu, \lambda}\,u\|_{B^{s_1}} + \|r\|_{B^{s_1+2}}^{\ell}+  \| r\|_{B^{s_2}}^h\Bigr)
\end{align*}
where $K$ is a positive large constant to be fixed later.

Applying Lemma \ref{SimpliCarre} from the Appendix, we get
\begin{align*}
\mathcal{L}^{\ell}_j(t)&   + {\kappa} \,2^{2j}\,\int_0^t \mathcal{L}^{\ell}_j(\tau)\,\notag\\
\leq& \exp\Bigl(\frac{CKC_2^2}{C_1^2}V(t)\Bigr)\,  \Big{\{}\mathcal{L}^{\ell}_j(0)+\int_0^t\,q_j(\tau) \, \Bigl(\|(S_2, S_3, S_4)(\tau)\|_{B^{s_1}}\\&\hspace*{3.2cm}+  \| \nabla v(\tau)\|_{B^\frac{d}{2}}\Bigl(q_j\,\|(w, r, u)(\tau)\|_{B^{s_1}}-K\mathcal{L}^{\ell}_j(\tau)\Bigr)
\Bigr)\,\\
&\hspace*{3.2cm}+\int_0^t
 c_0  \bigl(\|\frac{w}{\nu}\|_{B^{s_1}}^\ell+\|\frac{w}{\nu}\|_{B^{s_2-1}}^h + \|u\|_{B^{s_1+1}}+ \|\mathcal{A}_{\mu, \lambda}\,u\|_{B^{s_1}}  +\|r\|_{B^{s_1+2}}^{\ell}+  \| r\|_{B^{s_2}}^h  \bigr)
\Big{\}}
\end{align*}

Note that $\mathcal{L}^{\ell}_j(t)$ may be replaced by $\sup_{[0, t]} \mathcal{L}^{\ell}_j(\tau)$ on the left-hand side of above inequality.
Thanks to \eqref{equiv-norm-l} and after summation on $j\leq 0$, we conclude that 
\begin{align}
\|(w, &r, u)\|_{\widetilde{L}_t^\infty (B^{s_1})}^{\ell}+  {\kappa } \,\|(w, r, u) \|_{{L}^1_t(B^{s_1+2})}^{\ell} \notag\\
\leq& \exp\Bigl(\frac{CKC_2^2}{C_1^2}V(t)\Bigr)\,  \Big{\{}\|(w_0, r_0, u_0)\|_{ B^{s_1}}^{\ell} +\int_0^t \| \nabla v(\tau)\|_{B^\frac{d}{2}}\Bigl( \|(w, r, u)(\tau)\|_{B^{s_1}}-K\sum_{j\leq 0}\mathcal{L}^{\ell}_j(\tau)\Bigr)\,\notag\\
& +\int_0^t    \Bigl(\|(S_2, S_3, S_4)(\tau)\|_{B^{s_1}}+  c_0  \bigl(\|\frac{w}{\nu}\|_{B^{s_1}} +\|(w, u)\|_{B^{s_1+1}}+ \|\mathcal{A}_{\mu, \lambda}\,u\|_{B^{s_1}}  +\|r\|_{B^{s_1+2}}^{\ell}+  \| r\|_{B^{s_2}}^h \bigr)\Bigr)\,\Big{\}}\label{low-energy}.
\end{align}

\subsection{High frequencies analysis}
Throughout this part, we shall suppose that  $ j> 0.$   
First of all, recall the definitions of $C_1,~C_2$ from \eqref{def-C1C2} and notice that the assumptions \eqref{A1}, \eqref{A2} ensure that
\begin{align}
\frac{1}{2}h_i\leq h_i+H_i\leq \frac{3}{2}h_i\quad~~i=1, \cdots, 6,\label{in-cH1}
\end{align}
\begin{align}
 \frac{C_1}{3}\leq\frac{h_6}{3h_1}\leq\frac{h_6+H_6}{h_1+H_1}\leq  \frac{3h_6}{h_1}\leq 3C_2,\qquad\frac{C_1}{3}\leq\frac{h_5}{3h_3}\leq\frac{h_5+H_5}{h_3+H_3}\leq  \frac{3h_5}{h_3}\leq 3C_2,\label{in-cH2}
\end{align}
and
\begin{align}
\|\nabla\Bigl(  \frac{h_6+H_6}{h_1+H_1}\Bigr)\|_{L^\infty_{t, x}}+\|\nabla\Bigl( \frac{h_5+H_5}{h_3+H_3}\Bigr)\|_{L^\infty_{t, x}}
\leq  6c_0\frac{h_1+h_6}{h_1^2}+  6c_0\frac{h_5+h_3}{h_3^2}\leq \,C_3 c_0,\label{in-cH3}
\end{align}
\begin{align}
\|\partial_t\Bigl(  \frac{h_6+H_6}{h_1+H_1}\Bigr)(t)\|_{L^\infty_x}+\|\partial_t\Bigl( \frac{h_5+H_5}{h_5+H_3}\Bigr)(t)\|_{L^\infty_x}&\leq  \Bigl(6\frac{h_1+h_6}{h_1^2}+6\frac{h_5+h_3}{h_3^2}\Bigr)\|(\partial_t H_1, \partial_t H_3, \partial_t H_5, \partial_t H_6)(t)\|_{L^\infty_x}\notag\\
&\leq C_3 H(t)\label{in-cH4}
\end{align}
with $C_3:=  6\dfrac{h_1+h_6}{h_1^2}+6\dfrac{h_5+h_3}{h_3^2}.$

Using spectral localization properties, we have
\begin{align*}
2^{-2j} \|\nabla r_j\|_{L^2}\leq \frac{12}{5}2^{-j} \|r_j\|_{L^2} \leq \frac{24}{5}  \|r_j\|_{L^2},
\end{align*}
thus we easily obtain that as soon as $\varepsilon_h\leq  \min\{\dfrac{5h_5}{144h_3}, \dfrac{5}{48}\}$, \eqref{equiv-norm} is satisfied in the following way
\begin{align}
\frac{C_1^2}{9}\|(w_j, r_j, u_j)\|_{L^2}^2\leq\mathcal{L}_j^2(t)\leq 4C_2^2 \|(w_j, r_j, u_j)(t)\|_{L^2}^2.\label{equiv-norm-h}
\end{align}
Indeed, one has
\begin{align*}
 \mathcal{L}_j^2&= {\int
_{\mathbb{R}^{d}}\Bigl(\frac{h_6+H_6 }{h_1+H_1 }w_j^2+\frac{h_5+H_5 }{h_3+H_3 } r_j^2 +|u_j|^2+2\varepsilon_h2^{-2j}  u_j\cdot\nabla r_j\Bigr)} \notag\\
&\geq \frac{h_6 }{3h_1 }\|w_j\|_{L^2}^2+\frac{h_5 }{3h_3 } \|r_j\|_{L^2}^2 +\|u_j\|^2-\frac{24}{5}\varepsilon_h (\|u_j\|_{L^2}^2+\| r_j\|_{L^2}^2)\\
 &\geq \min\{\frac{h_6 }{6h_1 }, \frac{h_5 }{6h_3 }, \frac{1}{6}\} \|(w_j, r_j, u_j)\|_{L^2}\geq\frac{C_1^2}{9} \|(w_j, r_j, u_j)\|_{L^2}^2
\end{align*}
and
\begin{align*}
\mathcal{L}_j^2&\leq \frac{3h_6 }{h_1 }\|w_j\|_{L^2}^2+\frac{3h_5 }{h_3 } \|r_j\|_{L^2}^2 +\|u_j\|^2+\frac{24}{5}\varepsilon_l (\|u_j\|_{L^2}^2+\| r_j\|_{L^2}^2)\leq 4C_2^2\|(w_j, r_j, u_j)\|_{L^2}^2.
\end{align*}

Slightly different from \eqref{ESyst:1}, after applying the operator $\dot{\Delta}_{j}$ to the  three equations of \eqref{L}, page $12$,  we write that
\begin{equation}\label{ESyst:2}
\left\{
\begin{array}
[c]{l}%
\partial_{t}w_{j}+v\cdot\nabla w_j+(h_1+H_1 )\operatorname{div}%
u_{j}+(h_2+H_2)\dfrac{w_j}{\nu}=\dot{\Delta}_{j} S_2+T_{j}^{1},\\
\partial_{t}r_{j}+v\cdot\nabla r_j+(h_3+H_3)\operatorname{div}%
u_{j}=\dot{\Delta}_{j} S_3+T_{j}^{2},\\
\partial_{t}u_{j}+v\cdot\nabla u_j-(h_4+H_4)\mathcal{A}_{\mu,\lambda}%
u_{j}+  \eta\, u_{j}+(h_5+H_5)\nabla r_j+(h_6+H_6)\nabla w_{j}=\dot{\Delta}_{j} S_4+ T_{j}^{3},
\end{array}
\right. %
\end{equation}
where $T_{j}^{1}, T_{j}^{2}, T_{j}^{3}$ are commutators defined by 
\[
\left\{
\begin{array}
[c]{l}%
T_{j}^{1}=[v,\dot{\Delta}_{j}]\nabla w+[H_{1},\dot{\Delta}_{j}%
]\operatorname{div}u+\dfrac{1}{\nu}[H_{2},\dot{\Delta}_{j}] {w},\\
T_{j}^{2}=[v,\dot{\Delta}_{j}]\nabla r+[H_3, \dot{\Delta}_{j}%
]\operatorname{div}u,\\
T_{j}^{3}=[v,\dot{\Delta}_{j}]\nabla u-[H_4, \dot{\Delta}_{j}]\mathcal{A}_{\mu,\lambda}u+
{ {[H_5, \dot{\Delta}_{j}]\nabla r}}+[H_6, \dot{\Delta}%
_{j}]\nabla w.
\end{array}
\right.
\]

Now, multiplying the first equation of $\left(\text{\ref{ESyst:2}}\right)  $ with $\dfrac{h_6+H_6}{h_{1}
+H_{1}} w_{j}$ and the second equation with $\dfrac{h_5+H_5}{h_3+H_3} r_{j}$, we get  from \eqref{in-cH1}-\eqref{in-cH4} that
\begin{align*}
\frac{1}{2}\frac{d}{dt}&\int_{\mathbb{R}^d} \frac{h_6+H_6}{h_1+H_1} w_j^2+\frac{h_2h_6}{6h_1}\frac{\|w_j\|_{L^2}^2}{\nu}+\int_{\mathbb{R}^d} (h_6+H_6)w_j \,\div u_j\\
\leq&\frac{1}{2}\int_{\mathbb{R}^d} \Big\lbrace w_j^2\,\partial_t\Bigl(\frac{h_6+H_6}{h_1+H_1}\Bigr)+ \frac{h_6+H_6}{h_1+H_1} w_j^2\, \div v+ w_j^2 \,v\cdot\nabla\Bigl(\frac{h_6+H_6}{h_1+H_1}\Bigr)+2(\dot{\Delta}_{j} S_2+T^1_j)\frac{h_6+H_6}{c_1+H_1}w_j\Big\rbrace \\
\leq& \|w_j\|_{L^2}^2\Big\lbrace\Big{\|}\partial_t\Bigl(\frac{h_6+H_6}{h_1+H_1}\Bigr)\Big{\|}_{L^\infty}+\Big{\|}\frac{h_6+H_6}{h_1+H_1}\Big{\|}_{L^\infty}\|\div v\|_{L^\infty}+\Big{\|}\nabla \Bigl(\frac{h_6+H_6}{h_1+H_1}\Bigr)\Big{\|}_{L^\infty}\|v\|_{L^\infty}\Big\rbrace\\
& +\|w_j\|_{L^2}\Big{\|}\frac{h_6+H_6}{h_1+H_1}\Big{\|}_{L^\infty}(\|\dot{\Delta}_{j} S_2\|_{L^2}+\|T^1_j\|_{L^2}) \\
\leq& \|w_j\|_{L^2}^2\Bigl( C_3 H(t)+ 3C_2 \|\div v\|_{L^\infty}+ C_3 c_0\|v\|_{L^\infty}\Bigr)+ 3C_2 \|w_j\|_{L^2}(\|\dot{\Delta}_{j} S_2\|_{L^2}+\|T^1_j\|_{L^2})
\end{align*}
and
\begin{align*}
\frac{1}{2}\frac{d}{dt}&\int_{\mathbb{R}^d} \frac{h_5+H_5}{h_3+H_3} r_j^2+ \int_{\mathbb{R}^d} (h_5+H_5)r_j \,\div u_j\\
=& \frac{1}{2}\int_{\mathbb{R}^d}\Big\lbrace r_j^2\,\partial_t\Bigl(\frac{h_5+H_5}{h_3+H_3}\Bigr)+\frac{h_5+H_5}{h_3+H_3} r_j^2\, \div v+ r_j^2 \,v\cdot\nabla\Bigl(\frac{h_5+H_5}{h_3+H_3}\Bigr)+ 2(\dot{\Delta}_{j} S_3+T^2_j)\frac{h_5+H_5}{h_3+H_3}r_j\Big\rbrace\\
\leq & \|r_j\|_{L^2}^2 \Big\lbrace\Big{\|}\partial_t\Bigl(\frac{h_5+H_5}{h_3+H_3}\Bigr)\Big{\|}_{L^\infty}+\Big{\|}\frac{h_3+H_5}{h_3+H_3}\Big{\|}_{L^\infty}\|\div v\|_{L^\infty}+\Big{\|}\nabla \Bigl(\frac{h_5+H_5}{h_3+H_3}\Bigr)\Big{\|}_{L^\infty}\|v\|_{L^\infty}\Big\rbrace\\
&+ \|r_j\|_{L^2}\Big{\|}\frac{h_5+H_5}{h_3+H_3}\Big{\|}_{L^\infty}(\|\dot{\Delta}_{j} S_3\|_{L^2}+\|T^2_j\|_{L^2}) \\
\leq& \|r_j\|_{L^2}^2\Bigl( C_3 H(t)+ 3C_2\|\div v\|_{L^\infty}+C_3 c_0\|v\|_{L^\infty}\Bigr)+3C_2 \|r_j\|_{L^2}(\|\dot{\Delta}_{j} S_3\|_{L^2}+\|T^2_j\|_{L^2}).
\end{align*}

Multiplying the velocity equation with $u_{j}$ and observing that
\begin{align*}
-\int_{\mathbb{R}^{d}}(h_4+H_4)\mathcal{A}_{\mu,\lambda}u_{j}\cdot u_{j}& 
=\mu\int_{\mathbb{R}^{d}}\nabla u_j :\nabla\bigl((h_4+H_4)
u_{j}\bigr)+\left(  \mu+\lambda\right)  \int_{\mathbb{R}^{d}} \div u_j ~\div \bigl((h_4+H_4)u_j\bigr)\notag\\
\geq &\frac{h_4}{2}\mu \|\nabla u_j\|_{L^2}^2+\frac{h_4}{2}(\mu+\lambda)\|\div u_j\|_{L^2}^2\\&-(\mu \|\nabla u_j\|_{L^2}+(\mu+\lambda)\|\div u_j\|_{L^2})\| u_j\|_{L^2}\|\nabla H_4\|_{L^\infty},\notag
\end{align*}
we get that
\begin{align*}
\frac{1}{2}\frac{d}{dt}& \|u_j\|_{L^2}^2+\frac{h_4}{2}\mu \|\nabla u_j\|_{L^2}^2+\frac{h_4}{2}(\mu+\lambda)\|\div u_j\|_{L^2}^2+ \eta\,\|u_j\|_{L^2}^2\\&\hspace*{7cm}- \int_{\mathbb{R}^{d}}(h_6+H_6)w_j\,\div u_j+(h_5+H_5)r_j\,\div u_j\\
\leq & \int_{\mathbb{R}^{d}}\Big\lbrace |u_j|^2\,\div v+ (\dot{\Delta}_{j} S_4+T^3_j)u_j+w_j\,u_j\cdot\nabla H_6+r_j\,u_j\cdot\nabla H_5\Big\rbrace\\&+(\mu \|\nabla u_j\|_{L^2}+(\mu+\lambda)\|\div u_j\|_{L^2})\| u_j\|_{L^2}\|\nabla H_4\|_{L^\infty} \\
\leq & \|u_j\|_{L^2}^2 \|\div v\|_{L^\infty}+ \|u_j\|_{L^2}(\|\dot{\Delta}_{j} S^4\|_{L^2}+\|T^3_j\|_{L^2}+ \|w_j\|_{L^2}\|\nabla H_6\|_{L^\infty}+\|r_j\|_{L^2}\|\nabla H_5\|_{L^\infty})\\
&+(\mu \|\nabla u_j\|_{L^2}+(\mu+\lambda)\|\div u_j\|_{L^2})\| u_j\|_{L^2}\|\nabla H_4\|_{L^\infty}\\
\leq& \|u_j\|_{L^2}^2 \|\div v\|_{L^\infty}+ c_0 \|u_j\|_{L^2}\Bigl(\|w_j\|_{L^2}+ \|r_j\|_{L^2}+ \mu \|\nabla u_j\|_{L^2}+(\mu+\lambda)\|\div u_j\|_{L^2} \Bigr)\\&+\|u_j\|_{L^2}(\|\dot{\Delta}_{j} S^4\|_{L^2}+\|T^3_j\|_{L^2}).
\end{align*}
Summing up the resulting inequalities together, we obtain
\begin{align}
\frac{1}{2}\frac{d}{dt}&\int_{\mathbb{R}^d} \Bigl(\frac{h_6+H_6}{h_1+H_1} w_j^2+ \frac{h_5+H_5}{h_3+H_3} r_j^2+ |u_j|^2\Bigr)+ \frac{h_2h_6}{6h_1}\frac{\|w_j\|_{L^2}^2}{\nu}+ \eta\,\|u_j\|_{L^2}^2\notag\\&+\frac{h_4}{2}\mu \|\nabla u_j\|_{L^2}^2+\frac{h_4}{2}(\mu+\lambda)\|\div u_j\|_{L^2}^2\notag\\
\leq &  \|(w_j, r_j, u_j)\|_{L^2}^2\Bigl( C_3 H(t)+ 3C_2 \|\div v\|_{L^\infty}+C_3 h_0\|v\|_{L^\infty}\Bigr)\notag+ 3C_2 \|w_j \|_{L^2}(\|\ddj S_2\|_{L^2}+\|T^1_j\|_{L^2})\\&+  3C_2 \|r_j \|_{L^2}(\|\ddj S_3\|_{L^2}+\|T^2_j\|_{L^2})\notag+\|u_j \|_{L^2}(\|\ddj S_4\|_{L^2}+\|T^3_j\|_{L^2})\\&+  c_0 \|u_j\|_{L^2}\Bigl(\|w_j\|_{L^2}+ \|r_j\|_{L^2}+ \mu \|\nabla u_j\|_{L^2}+(\mu+\lambda)\|\div u_j\|_{L^2} \Bigr)
.\label{high-basic-energy-es1}
\end{align}

Observe that the left-hand side of \eqref{high-basic-energy-es1} does not encode any decay properties for $r.$  For this, take into account that $\nabla r_j$ verifies
\begin{equation}
\partial_{t}\nabla r_{j}+\nabla\left(  v\cdot \nabla r_{j}\right)  +\nabla\left(
(h_{3} +H_{3} )\operatorname{div}u_{j}\right)  =\nabla (\ddj S_3+  T_{j}^{2}).\label{ESyst_2}%
\end{equation}
Hence multiplying this equation with $u_{j}$, testing the equation of $u_{j}$
from $\left(  \text{\ref{ESyst:2}}\right)  $ with $\nabla r_{j}$ and summing
up the results we end up with%
\begin{align}
\frac{d}{dt}&  \int_{\mathbb{R}^{d}}(u_{j}\cdot\nabla r_{j})%
+ h_5 \|\nabla r_j\|_{L^2}^2\notag\\
 =&  \int_{\mathbb{R}^{d}}\Big\lbrace v\cdot\nabla r_{j}\operatorname{div}%
u_{j} +\left(  h_{3} +H_{3} \right)
|\operatorname{div}u_{j}|^{2}-(v\cdot\nabla u_j)\cdot \nabla r_j+(h_4+H_4)\mathcal{A}_{\mu, \lambda} u_j\cdot\nabla r_j\notag\\
&- \eta\,u_j\cdot\nabla r_j-(h_6+H_6)\nabla w_j\cdot\nabla r_j-(\ddj S_3 + T^2_j)\,\div u_j+ (\ddj S_4 + T^3_j)\cdot\nabla r_j-H_5 |\nabla r_j|^2\Bigr\rbrace\notag\\
\leq& 2\|v\|_{L^\infty}\|\nabla r_j\|_{L^2}\|\div u_j\|_{L^2}+ (h_3+\|H_3\|_{L^\infty})\|\div u_j\|_{L^2}^2+ (h_4+\|H_4\|_{L^\infty})\|\mathcal{A}_{\mu, \lambda} u_j\|_{L^2}\|\nabla r_j\|_{L^2}\notag\\
&+ \eta\,\|u_j\|_{L^2}\|\nabla r_j\|_{L^2}+ (h_6+\|H_6\|_{L^\infty})\|\nabla w_j\|_{L^2}\|\nabla r_j\|_{L^2}+(\|\ddj S_3\|_{L^2}+\|T^2_j\|_{L^2})\|\div u_j\|_{L^2} \notag\\
&+(\|\ddj S_4\|_{L^2}+\|T^3_j\|_{L^2})\|\nabla r_j\|_{L^2}+\frac{h_5}{2}\|\nabla r_j\|_{L^2}^2
\label{terme_croise-high}.
\end{align}

Using the spectral localization of $r_j, u_j$ and $\varepsilon_h\leq \frac{5}{48}$,  we have 
\begin{align*}
\varepsilon_h 2^{-2j}2\|v\|_{L^\infty}\|\nabla r_j\|_{L^2}\|\div u_j\|_{L^2}&\leq  
2\varepsilon_h(\frac{12}{5})^2\|v\|_{L^\infty}\| r_j\|_{L^2}\| u_j\|_{L^2}\\
&\leq   \|v\|_{L^\infty}(\| r_j\|_{L^2}^2+\| u_j\|_{L^2}^2).
\end{align*}

Using Young's inequality and \eqref{in-cH1}, we have
\begin{align*}
\varepsilon_h 2^{-2j}(h_3+\|H_3\|_{L^\infty})\|\div u_j\|_{L^2}^2&\leq \varepsilon_h (\frac{12}{5})^2\frac{3h_3}{2}\|u_j\|_{L^2}^2\leq 9 h_3\varepsilon_h\|u_j\|_{L^2}^2,\\
\varepsilon_h 2^{-2j} (h_6+\|H_6\|_{L^\infty})\|\nabla w_j\|_{L^2}\|\nabla r_j\|_{L^2}
&\leq \varepsilon_h 2^{-2j}\frac{3h_6}{2}\|\nabla w_j\|_{L^2}\|\nabla r_j\|_{L^2}\\
&\leq\varepsilon_h 2^{-2j}(\frac{h_5}{16}\|\nabla r_j\|_{L^2}+\frac{16}{h_5}(\frac{3h_6}{2} )^2\|\nabla w_j\|_{L^2}^2)\\
&\leq \varepsilon_h 2^{-2j}\frac{h_5}{16}\|\nabla r_j\|_{L^2}+256\frac{h_6^2}{h_5}\varepsilon_h\|w_j\|_{L^2}^2.
\end{align*}
Moreover, taking into account that $\max\{\mu, \mu+\lambda\}\leq \nu$ we obtain that
\begin{align*}
\varepsilon_h 2^{-2j}&(h_4+\|H_4\|_{L^\infty})\|\mathcal{A}_{\mu, \lambda} u_j\|_{L^2}\|\nabla r_j\|_{L^2}\\
 \leq& \varepsilon_h 2^{-j}\frac{3h_4}{2}\frac{12}{5}(\mu \|\nabla u_j\|_{L^2}+(\mu+\lambda)\|\div u_j\|_{L^2}) \|\nabla r_j\|_{L^2}\\
\leq&   \varepsilon_h 2^{-2j}\frac{h_5}{16}\|\nabla r_j\|_{L^2}^2+\varepsilon_h\frac{16}{h_5}(\frac{3h_4}{2}\frac{12}{5})^2(\mu \|\nabla u_j\|_{L^2}+(\mu+\lambda)\|\div u_j\|_{L^2})^2\\
\leq&    2^{-2j}\frac{h_5}{16}\varepsilon_h\|\nabla r_j\|_{L^2}^2+512\frac{h_4^2}{c_5}\nu\varepsilon_h(\mu \|\nabla u_j\|_{L^2}^2+(\mu+\lambda)\|\div u_j\|_{L^2}^2).
\end{align*}

As $j\geq -1$, we use the following rough inequalities:
\begin{align*}
\varepsilon_h 2^{-2j} \eta\,\|u_j\|_{L^2}\|\nabla r_j\|_{L^2}&\leq \varepsilon_h 2^{-2j}(\frac{16\eta^2}{h_5} \|u_j\|_{L^2}^2+\frac{h_5}{16}\|\nabla r_j\|_{L^2}^2)\\
&\leq  2^{-2j}\frac{h_5}{16}\varepsilon_h\|\nabla r_j\|_{L^2}^2+ \frac{64\eta^2}{h_5} \varepsilon_h \|u_j\|_{L^2}^2,\\
\varepsilon_h 2^{-2j}(\|\ddj S_3\|_{L^2}+\|T^2_j)\|_{L^2})\|\div u_j\|_{L^2} &\leq \varepsilon_h 2^{-2j}\frac{12}{5}2^{j}(\|\ddj S_3\|_{L^2}+\|T^2_j)\|_{L^2})\| u_j\|_{L^2} \\
&\leq \frac{24}{5}\varepsilon_h\| u_j\|_{L^2} (\|\ddj S_3\|_{L^2}+\|T^2_j)\|_{L^2}),\\
\varepsilon_h 2^{-2j}(\|\ddj S_4\|_{L^2}+\|T^3_j)\|_{L^2})\|\nabla r_j\|_{L^2} &\leq \varepsilon_h 2^{-2j}\frac{12}{5}2^{j}(\|\ddj S_4\|_{L^2}+\|T^3_j)\|_{L^2})\| r_j\|_{L^2} \\
&\leq \frac{24}{5}\varepsilon_h\| r_j\|_{L^2}(\|\ddj S_4\|_{L^2}+\|T^3_j)\|_{L^2}).
\end{align*}

Inserting the above estimates into \eqref{terme_croise-high}, one has
\begin{align}
 \varepsilon_h& 2^{-2j} \Bigl(\frac{d}{dt}\int_{\mathbb{R}^{d}}(u_{j}\cdot\nabla r_{j})%
+ \frac{h_5}{4} \|\nabla r_j\|_{L^2}^2\Bigr)\notag \\
\leq &  \|v\|_{L^\infty}(\| r_j\|_{L^2}^2+\| u_j\|_{L^2}^2)+256\frac{h_6^2}{h_5}\varepsilon_h\|w_j\|_{L^2}^2+(9 h_3\varepsilon_h+\frac{64\eta^2}{h_5} \varepsilon_h)\|u_j\|_{L^2}^2\notag\\
& + \frac{24}{5}\varepsilon_h\| u_j\|_{L^2}(\|\ddj S_3\|_{L^2}+\|T^2_j)\|_{L^2})+\frac{24}{5}\varepsilon_h\| r_j\|_{L^2}(\|\ddj S_4\|_{L^2}+\|T^3_j)\|_{L^2})\notag\\
& +512\frac{h_4^2}{h_5}\nu\varepsilon_h(\mu \|\nabla u_j\|_{L^2}^2+(\mu+\lambda)\|\div u_j\|_{L^2}^2)
\label{terme_croise-high1}.
\end{align}

Remembering the definition of $\mathcal{L}_j$, summing up \eqref{high-basic-energy-es1} and \eqref{terme_croise-high1} and keeping in mind that $\varepsilon_h\leq \frac{5}{48},$ we obtain
\begin{align}
&\frac{1}{2}\frac{d}{dt}\mathcal{L}_j^2+ (\frac{h_2h_6}{6h_1}-256\frac{h_6^2}{h_5}\nu\varepsilon_h)\frac{\|w_j\|_{L^2}^2}{\nu}+ 2^{-2j} \frac{h_5}{4}\varepsilon_h\|\nabla r_j\|_{L^2}^2+(\eta-9 h_3\varepsilon_h-\frac{64\eta^2}{h_5}\varepsilon_h)\|u_j\|_{L^2}^2\notag\\
&\qquad+(\frac{h_4}{2}-512\frac{h_4^2}{h_5}\nu\varepsilon_h)\bigl(\mu \|\nabla u_j\|_{L^2}^2+ (\mu+\lambda)\|\div u_j\|_{L^2}^2\bigr)\notag\\
&\quad\leq  \|(w_j, r_j, u_j)\|_{L^2}^2\Bigl( C_3 H(t)+3C_2 \|\div v\|_{L^\infty}+(C_3 c_0+1)\|v\|_{L^\infty}\Bigr)\notag\\
&\qquad\:\:+  c_0 \|u_j\|_{L^2}\Bigl(\|w_j\|_{L^2}+ \|r_j\|_{L^2}+ \mu \|\nabla u_j\|_{L^2}+(\mu+\lambda)\|\div u_j\|_{L^2} \Bigr)\notag\\
&\qquad\:\:+ 3C_2\|(w_j, r_j, u_j) \|_{L^2}\|\ddj S_2, \ddj S_3, \ddj S_4, ~T^1_j, T^2_j,~T^3_j)\|_{L^2}.\label{high-total-energy-es}
\end{align}

Let us choose $\varepsilon_h>0$ such that 
\begin{align*}
\varepsilon_h &\leq \min\{\frac{5h_5}{144h_3}, \frac{5}{48}\},\\
 \frac{h_2h_6}{12h_1}&\leq \frac{h_2h_6}{6c_1}-256\frac{h_6^2}{h_5}\nu\varepsilon_h,\\
 \frac{\eta}{2}&\leq  \eta- 9 h_3\varepsilon_h-\frac{64\eta^2}{h_5} \varepsilon_h,\\
 \frac{h_4}{4}&\leq \frac{h_4}{2}-512\frac{h_4^2}{c_5}\nu\varepsilon_h,
\end{align*}
for example one may take 
\begin{align*}
\varepsilon_h=\frac{1}{3072} \min\Big{\{} \frac{h_5}{h_3},~  1, ~ \frac{h_2 h_5}{h_1h_6\,\nu}, ~ \frac{h_5\eta  }{ h_3h_5+\eta^2 }, ~ \frac{h_5}{h_4\,\nu}\Big{\}}.
\end{align*}
Comparing to $\varepsilon_{\ell}$ which has been defined in the previous  part, we see that $\frac{1}{16}\varepsilon_{\ell}\leq\varepsilon_{h}\leq \varepsilon_{\ell}.$ Thus, one has $2^{-2j}\frac{h_5}{4}\varepsilon_h\|\nabla r_j\|_{L^2}^2\geq \frac{h_5}{256}\varepsilon_{\ell}\|r_j\|_{L^2}^2.$  
We now update the definition of $\kappa$ in \eqref{def-low-damping-constant} by
\begin{align}
\kappa:= \frac{1}{256}\min\{ \frac{h_2h_6}{h_1\,\nu}, ~ {h_5}\,{\varepsilon_{\ell}} ,~\eta \}
~(\geq\min\{ \frac{h_2h_6}{12h_1\,\nu}, ~ \frac{h_5}{256} {\varepsilon_{\ell}},~\frac{\eta}{2} \}).\label{def-kappa}
\end{align}
Then using \eqref{equiv-norm-h} we are able to rewrite \eqref{high-total-energy-es} as
 \begin{align}
\frac{1}{2}\frac{d}{dt}\mathcal{L}_j^2+   \frac{\kappa}{4C_2^2}\mathcal{L}_j^2
\leq &  \frac{9}{C_1^2}\mathcal{L}_j^2\Bigl( C_3 H(t)+ 3C_2 \|\div v\|_{L^\infty}+(C_3 c_0+1)\|v\|_{L^\infty}\Bigr)\notag\\
&+    \frac{3h_0}{C_1}\mathcal{L}_j\, \Bigl(\|w_j\|_{L^2}+ \|u_j\|_{L^2}+ \mu \|\nabla u_j\|_{L^2}+(\mu+\lambda)\|\div u_j\|_{L^2} \Bigr)\notag\\
&+ \frac{3C_2}{C_1}\mathcal{L}_j\,\|\ddj S_2, \ddj S_3, \ddj S_4, ~T^1_j, T^2_j,~T^3_j)\|_{L^2}.\label{high-total-energy-es11}
\end{align}

Now, we start  to estimate $T^1_j, T^2_j, T^3_j$, which require some particular attention. For example, the term $[H_5, \ddj]\nabla r$ in $T^3_j.$  We  write
\begin{align*}
[H_5, \ddj]\nabla r =[H_5, \ddj]\nabla r_j^{\ell}+ [H_5, \ddj]\nabla r^{h}
\end{align*}
 and  by Proposition \ref{Commutator}, we have
\begin{align*}
2^{s_2j}\,\|[H_5, \ddj]\,\nabla r^{h}\|_{L^2}\leq Cq_j \|\nabla H_5\|_{B^{\frac{d}{2}}} \|\nabla r^{h}\|_{B^{s_2-1}}\leq Cq_j \|\nabla H_5\|_{B^{\frac{d}{2}}} \|r\|_{B^{s_2}}^{h},
\end{align*}
\begin{align*}
2^{s_2j}\,\|[H_5, \ddj]\,\nabla r^{\ell}\|_{L^2}&\leq 2^{(\frac{d}{2}+1)j}\,\|[H_5, \ddj]\nabla r^{\ell}\|_{L^2}\\
&\leq Cq_j \|\nabla H_5\|_{B^{\frac{d}{2}}} \|\nabla r^{\ell}\|_{B^{\frac{d}{2}}}
\leq Cq_j \|\nabla H_5\|_{B^{\frac{d}{2}}} \|r\|_{B^{s_1+2}}^{\ell}.
\end{align*}
Above, the conditions that 
$$s_1+2,~s_2\leq \frac{d}{2}+1$$ 
are crucial. The  other terms can be estimated using similar argument, we thus obtain
 \begin{align*}
 \|T^1_j\|_{L^2}&\leq C  2^{-s_2j}q_j \Bigl(\|\nabla v\|_{B^{\frac{d}{2}}}\|\nabla w\|_{B^{s_2-1}}+\|\nabla H_1\|_{B^{\frac{d}{2} }}\|\div u\|_{B^{s_2-1}}+\|\nabla H_2\|_{B^{\frac{d}{2}}}  (\|\frac{w}{\nu}\|_{B^{s_1+1}}^{\ell}+\|\frac{w}{\nu}\|_{B^{s_2-1}}^h)\Bigr),\\
\|T^2_j\|_{L^2}&\leq  C2^{-s_2j}q_j \Bigl(\|\nabla v\|_{B^{\frac{d}{2}}}\|\nabla r\|_{B^{s_2-1}}+\|\nabla H_3\|_{B^{\frac{d}{2}}}\|\div u\|_{B^{s_2-1}}\Bigr)
 \end{align*}
and
\begin{align*}
 \|T^3_j\|_{L^2}
\leq&C 2^{-s_2j}q_j  \Bigl(\|\nabla v\|_{B^{\frac{d}{2}}}\|\nabla u\|_{B^{s_2-1}}+\|\nabla H_4\|_{B^{\frac{d}{2}}}\|\mathcal{A}_{\mu, 
\lambda} u\|_{B^{s_2-1}}+ \|\nabla H_5\|_{B^{\frac{d}{2}}}(\|r\|_{B^{s_1+2}}^\ell+\|r\|_{B^{s_2}}^h)\\
&\hspace*{2cm}+ \|\nabla H_6\|_{B^{\frac{d}{2}}}\|\nabla w\|_{B^{s_2-1}}\Bigr).
 \end{align*}
Thanks to smallness assumption \eqref{A2}, we gather that
 \begin{align*}
 \quad 2^{s_2j} &\|(T^1_j, T^2_j, T^3_j)\|_{L^2}\\
  &\leq C q_j  \|\nabla v\|_{B^{\frac{d}{2}}}\|( w, r, u)\|_{B^{s_2}}+
C c_0 \,q_j  \Bigl(\|(w,  u)\|_{B^{s_2}}+ {\| (\frac{1}{\nu}w, ~\mathcal{A}_{\mu, 
\lambda}\, u)\|_{B^{s_2-1}}} +(\|r\|_{B^{s_1+2}}^\ell+\|r\|_{B^{s_2}}^h)\Bigr)\\
&\leq C K2^{s_2 j}\mathcal{L}_j\|\nabla v\|_{B^{\frac{d}{2}}}   + C\|\nabla v\|_{B^{\frac{d}{2}}}\Bigl(\,q_j  \|( w, r, u)\|_{B^{s_2}}-   K2^{s_2 j}\mathcal{L}_j\Bigr)\\
&\quad+ C c_0 \,q_j  \Bigl(\|(w,  u)\|_{B^{s_2}}+ {\|\mathcal{A}_{\mu, 
\lambda}\, u\|_{B^{s_2-1}}}+(\|\frac{w}{\nu}\|_{B^{s_1+1}}^{\ell}+\|\frac{w}{\nu}\|_{B^{s_2-1}}^h) +(\|r\|_{B^{s_1+2}}^\ell+\|r\|_{B^{s_2}}^h)\Bigr).
 \end{align*}
 Define $\mathcal{L}_j^h=2^{s_2j}\mathcal{L}_j$, and use Gronwall's Lemma which implies that
\begin{align*}
 \mathcal{L}^{h}_j(t)&  + \frac{\kappa}{4C_2^2} \,\int_0^t \mathcal{L}^{h}_j(\tau)\,\notag\\
\leq& \exp\Bigl(\frac{CKC_2^2}{C_1^2}\bigl(C_3 (H(t)+ V(t))  +  V(t)\bigr)\Bigr)\,   \\
&\, \Big{\{}\mathcal{L}^{h}_j(0)+\int_0^t\,q_j(\tau)\Bigl(\|(S_2, S_3, S_4)(\tau)\|_{B^{s_2}}+  \| \nabla v(\tau)\|_{B^\frac{d}{2}}\Bigl(q_j\,\|(w, r, u)(\tau)\|_{B^{s_2}}-K\mathcal{L}^{h}_j(\tau)\Bigr)\,
\Bigr)\,\\
&+\int_0^t
 c_0  \bigl(\|(w,  u)\|_{B^{s_2}}+ {\|  \mathcal{A}_{\mu, 
\lambda}\, u\|_{B^{s_2-1}}}+(\|\frac{w}{\nu}\|_{B^{s_1+1}}^{\ell}+\|\frac{w}{\nu}\|_{B^{s_2-1}}^h)+(\|r\|_{B^{s_1+2}}^\ell+\|r\|_{B^{s_2}}^h)\bigr)
\Big{\}}.
\end{align*}

Note that $\mathcal{L}^{h}_j(t)$ may be replaced by $\sup_{[0, t]} \mathcal{L}^{h}_j(\tau)$ on the left-hand side of above inequality.
Thanks to \eqref{equiv-norm-h}, we conclude  after summation over $j\geq-1$,  that 
\begin{align}
\|(w, r,& u)\|_{\widetilde{L}_t^\infty (B^{s_2})}^{h}+  {\kappa } \,\|(w, r, u) \|_{ {L}^1_t(B^{s_2})}^{h} \notag\\
\leq& \exp\Bigl(\frac{CKC_2^2}{C_1^2}\bigl(C_3 (H(t)+ V(t))  +  V(t)\bigr)\Bigr)\, \notag  \\
&\, \Big{\{}\|(w_0, r_0, u_0)\|_{  B^{s_2}}^{h}+\int_0^t\, \Bigl(\|(S_2, S_3, S_4)(\tau)\|_{B^{s_2}}+  \| \nabla v(\tau)\|_{B^\frac{d}{2}}\Bigl( \,\|(w, r, u)(\tau)\|_{B^{s_2}}-K\sum_{j\geq -1}\mathcal{L}^{h}_j(\tau)\Bigr)\,
\Bigr)\,\notag\\
&+ c_0 \int_0^t
 \Bigl(\|(w,  u)\|_{B^{s_2}}+ {\|  \mathcal{A}_{\mu, 
\lambda}\, u\|_{B^{s_2-1}}}+(\|\frac{w}{\nu}\|_{B^{s_1+1}}^{\ell}+\|\frac{w}{\nu}\|_{B^{s_2-1}}^h) +(\|r\|_{B^{s_1+2}}^\ell+\|r\|_{B^{s_2}}^h)\, \Bigr)
\Big{\}}\label{L-H-es1}.
\end{align}
We are now going to show that inequalities \eqref{low-energy} and \eqref{L-H-es1} entails a decay for $w$ and $u.$ In  fact, one finds  from \eqref{equiv-norm-l} and \eqref{equiv-norm-h} that 
\begin{align*}
&\int_0^t \| \nabla v(\tau)\|_{B^\frac{d}{2}}\Bigl( \,\|(w, r, u)(\tau)\|_{B^{s_1}}-K\sum_{j\leq 0}\mathcal{L}^{\ell}_j(\tau)\Bigr)\, + \int_0 ^t\| \nabla v(\tau)\|_{B^\frac{d}{2}}\Bigl( \,\|(w, r, u)(\tau)\|_{B^{s_2}}-K\sum_{j\geq -1}\mathcal{L}^{h}_j(\tau)\Bigr)\,\\
&\leq \int_0^t \| \nabla v(\tau)\|_{B^\frac{d}{2}}\Bigl(  \|(w, r, u)(\tau)\|_{B^{s_1}\cap B^{s_2}}- \frac{C_1K}{3}\Bigl(\sum_{j\leq 0} 2^{s_1j} \|(w_j, r_j, u_j)\|_{L^2}+\sum_{j\geq -1} 2^{s_2j} \|(w_j, r_j, u_j)\|_{L^2}\Bigr)\\
&\leq \int_0^t \| \nabla v(\tau)\|_{B^\frac{d}{2}}\Bigl(  \|(w, r, u)(\tau)\|_{B^{s_1}\cap B^{s_2}}- \frac{C_1K}{6}\Bigl(\|(w, r, u)(\tau)\|_{B^{s_1}}+ \|(w, r, u)(\tau)\|_{ B^{s_2}}\Bigr)\leq 0
\end{align*}
when we choose  $K=\dfrac{12}{C_1}.$
Thus, we conclude from summing up \eqref{low-energy} and \eqref{L-H-es1}  that 
\begin{align}
\|(w, r&, u)\|_{\widetilde{L}_t^\infty (B^{s_1})}^{\ell}+\|(w, r, u)\|_{\widetilde{L}_t^\infty (B^{s_2})}^{h}+  {\kappa } \,\Bigl(\|(w, r, u) \|_{ {L}^1_t(B^{s_1+2})}^{\ell}+\|(w, r, u) \|_{ {L}^1_t(B^{s_2})}^{h} \Bigr)\notag\\
\leq& \exp\Bigl(\frac{CC_2^2}{C_1^2}\bigl(C_3 H(t)+ \max\{C_3, 1\} V(t)\bigr)\Bigr)\, \notag  \\
&\, \Big{\{}\|(w_0, r_0, u_0)\|_{  B^{s_1}}^{\ell}+\|(w_0, r_0, u_0)\|_{  B^{s_2}}^{h}+\int_0^t\,  \|(S_2, S_3, S_4)(\tau)\|_{B^{s_1}\cap B^{s_2}} \,\notag\\
&\quad + c_0 \int_0^t
 \Bigl((\|\frac{w}{\nu}\|_{B^{s_1}\cap B^{s_2-1}} +  \|    \mathcal{A}_{\mu, 
\lambda}\, u\|_{B^{s_1}\cap B^{s_2-1}} 
  + \|(w, u)\|_{B^{s_1+1}\cap B^{s_2}}  
 +(\|r\|_{B^{s_1+2}}^\ell+\|r\|_{B^{s_2}}^h)\, \Bigr)
\Big{\}}. \label{L-H-es2}
\end{align}

\subsection{The damping effects and estimation for time derivatives}
We will now recover uniform estimates concerning the decay  and   time derivatives of our solutions. We have the following lemma.
\begin{lemma}
Under the hypotheses  in Proposition \ref{Prop-L} then
\begin{multline} \label{Decay-L}
 \int_0^t (\|(\partial_t w, \frac{w}{\nu})\|_{B^{s_1}}^\ell+\|(\partial_t w, \frac{w}{\nu})\|_{B^{s_2-1}}^h)+ \int_0^t(\|(\partial_t u,  u,  \eta u, \partial_t r) \|_{B^{{s_1}+1}}^\ell+\|(\partial_t u,    \eta u, \partial_t r) \|_{B^{s_2-1}}^h)\\
 +\int_0^t(\|\mu\Delta u, (\mu+\lambda)\nabla\div u \|_{B^{{s_1}+1}}^\ell+\|\mu\Delta u, (\mu+\lambda)\nabla\div u \|_{B^{s_2-1}}^h)\\
 \leq C\Bigl(\|(w_0, r_0, u_0)\|_{B^{s_1}\cap B^{s_2}}+\int_0^t\|(v, \nabla v)\|_{B^{\frac{d}{2}}}\|(w, r, u)\|_{B^{s_1}\cap B^{s_2}}+\int_0^t\|(S_2,  S_3, S_4)\|_{B^{s_1}\cap B^{s_2}}\\
 +  \int_0^t  (\| r \|_{B^{s_1+2}}^\ell+ \| ( w, r, u) \|_{B^{s_2}}^h)\Bigr).
\end{multline}
\end{lemma}
\begin{proof}
Let us first look at the equation of $w$ in System \eqref{L}, which reads
\begin{align*}
\partial_t w+ v\cdot\nabla w+ \frac{h_2}{\nu} \,w= -\frac{1}{\nu}H_2\,w-(h_1+H_1)\,\div u +S_2.
\end{align*}
We  infer from Proposition \ref{Prop_Transport} (take $r_1=s_1, r_2=s_2-1$) that
\begin{align*}
\|w(t)\|_{B^{s_1}}^{\ell}+   \frac{h_2}{\nu} \int_0^t  {\|w(\tau)\|_{B^{s_1}}^{\ell}} \,\notag
\leq &\:\|w_0\|_{B^{s_1}}^{\ell}+C\int^t_0 \| v(\tau)\|_{B^{\frac{d}{2}}}\|w\|_{B^{s_1+1}}\\+& \int_0^t  \Bigl(\| \frac{1}{\nu}H_2\,w\|_{B^{s_1}}^{\ell} +\|(h_1+H_1)\,\div u \|_{B^{s_1}}^{\ell} +\|S_2 \|_{B^{s_1}}^{\ell}\Bigr)
\end{align*}
and
\begin{align*}
\|w(t)\|_{B^{s_2-1}}^h+   \frac{h_2}{\nu} \int_0^t  {\|w(\tau)\|^h_{B^{s_2-1}}}\,\notag
\leq &\: \|w_0\|_{B^{s_2-1}}^h+C\int^t_0 \|\nabla v \|_{B^{\frac{d}{2}}}\|w\|_{B^{s_2-1}}\,\\&+\int_0^t  \Bigl(\| \frac{1}{\nu}H_2\,w\|_{B^{s_2-1}}^{h} +\|(h_1+H_1)\,\div u \|_{B^{s_2-1}}^{h} +\|S_2 \|_{B^{s_2-1}}^{h}\Bigr)
\,. 
\end{align*}

By  assumption \eqref{A2} and $s_1\leq s_2-1$, one has
\begin{align*}
\frac{1}{\nu}\| H_2 \,w\|_{B^{s_1}}^\ell&\leq  \frac{C}{\nu} \|H_2\|_{B^{\frac{d}{2}}}\|w\|_{B^{s_1}}\leq\frac{Cc_0}{\nu}  (\|w\|_{B^{s_1}}^{\ell}+\|w\|_{B^{s_2-1}}^h) ,\\
\|(h_1+H_1)\div u\|_{B^{s_1}}^{\ell}&\leq C(h_1+ \|H_1\|_{B^{\frac{d}{2}}})\| u\|_{B^{s_1+1}},\\
\frac{1}{\nu}\| H_2 \,w\|_{B^{s_2-1}}^h&\leq  \frac{C}{\nu} \|H_2\|_{B^{\frac{d}{2}}}\|w\|_{B^{s_2-1}}\leq\frac{Cc_0}{\nu}  (\|w\|_{B^{s_1}}^{\ell}+\|w\|_{B^{s_2-1}}^h) ,\\
\|(h_1+H_1)\div u\|_{B^{s_2-1}}^{h}&\leq C(h_1+ \|H_1\|_{B^{\frac{d}{2}}})\| u\|_{B^{s_2}}.
\end{align*}
Assuming that $c_0\leq \frac{h_2}{2C}$, we conclude that
\begin{multline}\label{back-w-es}
\|w(t)\|_{B^{s_1}}^\ell+ \|w(t)\|_{B^{s_2-1}}^h+ \frac{h_2}{2\nu} \int_0^t ({\|w(\tau)\|^\ell_{B^{s_1}}}+\|w(\tau)\|^h_{B^{s_2-1}}) \\
\leq   \|w_0\|_{B^{s_1}}^\ell+ \|w_0\|_{B^{s_2-1}}^h +C\int^t_0 (\|  v \|_{B^{\frac{d}{2}}}\|w\|_{B^{s_1+1}}+\|\nabla v\|_{B^{\frac{d}{2}}}\|w\|_{B^{s_2-1}})\\ 
+C \int_0^t  (\|  u \|_{B^{s_1+1}}+ \|  u \|_{B^{s_2}}) 
+C\int_0^t(\|S_2 \|_{B^{s_1}}^{\ell}+ \|S_2 \|_{B^{s_2-1}}^h) \,.
\end{multline}
The estimate of the  time derivative comes readily from the  equation of $w$, we have  
\begin{multline*}
    \int_0^t(\|\partial_t w\|_{B^{s_1}}^\ell+ \|\partial_t w\|_{B^{s_2-1}}^h)\\
    \leq \int_0^t\Bigl(\|v\cdot \nabla w\|_{B^{s_1}}^\ell+\|(h_2+H_2)\frac{w}{\nu}\|_{B^{s_1}}^\ell+\|(h_1+H_1)\div u\|_{B^{s_1}}^\ell+\|S_2\|_{B^{s_1}}^\ell\\
    \quad +\|v\cdot \nabla w\|_{B^{s_2-1}}^h+\|(h_2+H_2)\frac{w}{\nu}\|_{B^{s_2-1}}^h+\|(h_1+H_1)\div u\|_{B^{s_2-1}}^h+\|S_2\|_{B^{s_2-1}}^h\Bigr)\\
    \leq C\int_0^t\Bigl(\|v\|_{B^{\frac{d}{2}}}\|w\|_{B^{s_1+1}\cap B^{s_2} }+(\|\frac{w}{\nu}\|_{B^{s_1}}^\ell+\|\frac{w}{\nu}\|_{B^{s_2-1}}^h)+\|u\|_{B^{s_1+1}\cap B^{s_2}}+ (\|S_2\|_{B^{s_1}}^\ell+\|S_2\|_{B^{s_2-1}}^h)\Bigr).
\end{multline*}
 This combined with \eqref{back-w-es} implies that
 \begin{multline*}
     \|w(t)\|_{B^{s_1}}^\ell+ \|w(t)\|_{B^{s_2-1}}^h+ \frac{h_2}{4\nu} \int_0^t ({\|w(\tau)\|^\ell_{B^{s_1}}}+\|w(\tau)\|^h_{B^{s_2-1}}) + \frac{h_2}{4C}\int_0^t(\|\partial_t w\|_{B^{s_1}}^\ell+ \|\partial_t w\|_{B^{s_2-1}}^h)\\
\leq   \|w_0\|_{B^{s_1}}^\ell+ \|w_0\|_{B^{s_2-1}}^h +C\int^t_0 (\|  v \|_{B^{\frac{d}{2}}}\|w\|_{B^{s_1}\cap B^{s_2}}+\|\nabla v\|_{B^{\frac{d}{2}}}\|w\|_{B^{s_2-1}})\\ 
+C \int_0^t  \Bigl(\|  u \|_{B^{s_1+1}\cap B^{s_2}} +(\|S_2 \|_{B^{s_1}}^{\ell}+ \|S_2 \|_{B^{s_2-1}}^h)\Bigr) \,,
 \end{multline*}
 which further gives
 \begin{multline}\label{Prop4.1-1}
     \frac{h_2}{4\nu} \int_0^t ({\|w(\tau)\|^\ell_{B^{s_1}}}+\|w(\tau)\|^h_{B^{s_2-1}}) + \frac{h_2}{4C}\int_0^t(\|\partial_t w\|_{B^{s_1}}^\ell+ \|\partial_t w\|_{B^{s_2-1}}^h)\\
\leq   \|w_0\|_{B^{s_1}}^\ell+ \|w_0\|_{B^{s_2}}^h +C\int^t_0 \| ( v, \nabla v) \|_{B^{\frac{d}{2}}}\|w\|_{B^{s_1}\cap B^{s_2}}\\ 
+C \int_0^t  \Bigl(\|  u \|_{B^{s_1+1}\cap B^{s_2}} +(\|S_2 \|_{B^{s_1}}^{\ell}+ \|S_2 \|_{B^{s_2-1}}^h)\Bigr) \,.
 \end{multline}

Next, we consider the equation of $u$ in System \eqref{L} and we write that
\begin{align*}
\partial_t u+v\cdot\nabla u-h_4\mu\Delta u-h_4(\lambda+\mu)&\nabla\div u+  \eta u\notag\\
=&H_4\mu\Delta u+H_4(\lambda+\mu)\nabla\div u-(h_5+H_5)\nabla r-(h_6+H_6)\nabla w+S_4.
\end{align*}
Applying Proposition \ref{P-Lame} (take $r_1=s_1+1, r_2=s_2-1$) to above equation gives that
\begin{multline*}
\|u(t) \|_{B^{s_1+1}}^\ell+\|u(t) \|_{B^{s_2-1}}^h +\mu\int_0^t(\|\nabla u\|_{B^{{s_1}+2}}^\ell+\|\nabla u\|_{B^{{s_2}}}^h) \\
+ (\lambda+\mu)\int_0^t(\|\div u\|_{B^{s_1+2}}^\ell+\|\div u\|_{B^{s_2}}^h)
+\eta\int_0^t(\|u \|_{B^{s_1+1}}^\ell+\|u \|_{B^{s_2-1}}^h)\\
\lesssim \|u_0\|_{B^{s_1+1}}^\ell+\|u_0\|_{B^{s_2-1}}^h + \int_0^t(\|  v\|_{B^{\frac{d}{2}+1}}\|u\|_{B^{s_1+1}\cap B^{s_2-1}}+\|  v\|_{B^\frac{d}{2}}\|u\|_{  B^{s_1+2}\cap B^{s_2}})+\int_0^t\|S_4\|_{B^{s_1+1}}^\ell+\|S_4\|_{B^{s_2-1}}^h\\
+\int_0^t\Bigl(\mu\|H_4\Delta u\|_{B^{s_1+1}}^\ell+(\lambda+\mu)\|H_4\nabla\div u\|_{B^{s_1+1}}^\ell+\|(h_5+H_5)\nabla r\|_{B^{s_1+1}}^\ell+\|(h_6+H_6)\nabla w\|_{B^{s_1+1}}^\ell\Bigr)\\
 +\int_0^t\Bigl(\mu\|H_4\Delta u\|_{B^{s_2-1}}^h+(\lambda+\mu)\|H_4\nabla\div u\|_{B^{s_2-1}}^h+\|(h_5+H_5)\nabla r\|_{B^{s_2-1}}^h+\|(h_6+H_6)\nabla w\|_{B^{s_2-1}}^h\Bigr)
\end{multline*}
Then we use the low and high frequencies decomposition,  and  product law \eqref{eq:prod2} to estimate (notice that $s_2-1\leq s_1+1$)
\begin{align*}
\mu\|H_4\Delta u\|_{B^{s_1+1}}^\ell
&\leq \mu(\|H_4 \Delta u^\ell\|_{B^{s_1+1}}^\ell+ \|H_4 \Delta u^h\|_{B^{s_2-1}}^\ell)\\
&\leq C\mu\|H_4\|_{B^\frac{d}{2}}(\|\nabla u\|_{B^{s_1+2}}^\ell+ \|\nabla u\|_{B^{s_2}}^h).
\end{align*}
Similarly, we have
\begin{align*}
(\lambda+\mu)\|H_4\nabla\div u\|_{B^{s_1+1}}^\ell&\leq
C(\lambda+\mu)\|H_4\|_{B^\frac{d}{2}}(\|\div u\|_{B^{s_1+2}}^\ell+ \|\div u\|_{B^{s_2}}^h),\\
\|(h_5+H_5)\nabla r\|_{B^{s_1+1}}^\ell&\leq C(h_5+\|H_5\|_{B^\frac{d}{2}})(\|r\|_{B^{s_1+2}}^\ell+ \|r\|_{B^{s_2}}^h), \\
\|(h_6+H_6)\nabla w\|_{B^{s_1+1}}^\ell&\leq C(h_6+\|H_6\|_{B^\frac{d}{2}})(\|w\|_{B^{s_1+2}}^\ell+ \|w\|_{B^{s_2}}^h).
\end{align*}

For terms of high frequencies part,  
using product law \eqref{eq:prod2} and low and high frequencies decomposition, and the fact that $s_2-1\leq s_1+1$ again, we have
\begin{align*}
\mu\|H_4\Delta u\|_{B^{s_2-1}}^h
&\leq C\mu(\|H_4 \Delta u^\ell\|_{B^{s_1+1}}^h+ \|H_4 \Delta u^h\|_{B^{s_2-1}}^h)\\
&\leq C\mu\|H_4\|_{B^{\frac{d}{2} }}(\|\nabla  u\|_{B^{s_1+2}}^\ell+\| \nabla  u\|_{B^{s_2}}^h)
\end{align*}
and similarly
\begin{align*}
    (\lambda+\mu)\|H_4\nabla\div u\|_{B^{s_1+1}}^h&\leq
     C(\lambda+\mu)\|H_4\|_{B^{\frac{d}{2} }}(\|\div  u\|_{B^{s_1+2}}^\ell+\| \div u\|_{B^{s_2}}^h),\\
\|(h_5+H_5)\nabla r\|_{B^{s_2-1}}^h&\leq C(h_5+\|H_5\|_{B^{\frac{d}{2} }})(\|  r\|_{B^{s_1+2}}^\ell+\|  r\|_{B^{s_2}}^h),\\
\|(h_6+H_6)\nabla w\|_{B^{s_2-1}}^h&\leq C(h_6+\|H_6\|_{B^{\frac{d}{2} }})\left(\|w\|^\ell_{B^{s_1+2}}+\|  w\|^h_{B^{s_2}}\right).
\end{align*}
Choosing $c_0$ small enough, we see that
\begin{multline*} 
\|u(t) \|^\ell_{B^{s_1+1}}+\|u(t)\|^h_{B^{s_2-1}}+\mu\int_0^t(\|\nabla u\|_{B^{{s_1}+2}}^\ell+\|\nabla u\|_{B^{s_2}}^h)\\+ (\lambda+\mu)\int_0^t(\|\div u\|_{B^{s_1+2}}^\ell+\|\div u\|_{B^{s_2}}^h)
+\eta\int_0^t(\|u \|^\ell_{B^{s_1+1}}+\|u\|_{B^{s_2-1}}^h)\\
\leq C(\|u_0\|^\ell_{B^{s_1+1}}+\|u_0\|^h_{B^{s_2-1}})+C \int_0^t(\|\nabla v\|_{B^\frac{d}{2}}\|u\|_{B^{s_1+1}\cap B^{s_2-1}}+\|  v\|_{B^\frac{d}{2}}\|u\|_{B^{s_2}})\\
+C\int_0^t(\|S_4\|_{B^{s_1+1}}^\ell+\|S_4\|_{B^{s_2-1}}^h)
+C\int_0^t (\|(w, r)\|^\ell_{B^{s_1+2}}+\|  (w, r)\|^h_{B^{s_2}}) \,.
\end{multline*}
The estimates of  the  time derivative $\partial_t u$ will be obtained through the equation of $u$. Notice that
\begin{align*}
    \|v\cdot\nabla u\|_{B^{s_1+1}}^\ell&\leq  \|v\cdot\nabla u^\ell\|_{B^{s_1+1}}^\ell+ \|v\cdot\nabla u^h\|_{B^{s_2-1}}^\ell\\
    &\leq C\|v\|_{B^{\frac{d}{2}}}(\|u\|_{B^{s_1+1}}^\ell+ \|u\|_{B^{s_2}}^h)\leq   C\|v\|_{B^{\frac{d}{2}}}\|u\|_{B^{s_1+1}\cap B^{s_2}}
\end{align*}
and
\begin{align*}
 \|v\cdot\nabla u\|_{B^{s_2-1}}^h\leq  C\|v\|_{B^{\frac{d}{2}}}\|u\|_{B^{s_2}}.
\end{align*}
Then one immediately has
\begin{multline}\label{es-4.1-u-1}
\int_0^t(\|(\partial_t u,    \eta u) \|_{B^{{s_1}+1}}^\ell+\|(\partial_t u,    \eta u) \|_{B^{s_2-1}}^h)+\int_0^t(\|\mu\Delta u, (\mu+\lambda)\nabla\div u \|_{B^{{s_1}+1}}^\ell+\|\mu\Delta u, (\mu+\lambda)\nabla\div u \|_{B^{s_2-1}}^h)\\
\leq C(\|u_0\|^\ell_{B^{s_1+1}}+\|u_0\|^h_{B^{s_2-1}})+C \int_0^t(\|\nabla v\|_{B^\frac{d}{2}}\|u\|_{B^{s_1+1}\cap B^{s_2-1}}+\|  v\|_{B^\frac{d}{2}}\|u\|_{B^{s_1+1}B^{s_2}})\\+C\int_0^t(\|S_4\|_{B^{s_1+1}}^\ell+\|S_4\|_{B^{s_2-1}}^h)
+C\int_0^t (\|(w, r)\|^\ell_{B^{s_1+2}}+\|  (w, r)\|^h_{B^{s_2}}) \,.
\end{multline}
At last, we show estimates of the time derivative of $r.$ Recall that $r$ satisfies
$$\partial_t r =-v\cdot\nabla  r+ (h_3+H_3)\div u+S_3,$$
thus we have (note that $s_2-1\leq s_1+1$)
\begin{align*}  
\int_0^t\|\partial_t r\|^h_{B^{s_2-1}}&\leq \int_0^t\Bigl(\|v\cdot\nabla  r\|_{B^{s_2-1}}^h+ \|(h_3+H_3)\div u\|_{B^{s_2-1}}^h+ \|S_3\|^h_{B^{s_2-1}}\Bigr)\nonumber\\
&\leq C \int_0^t\Bigl(\|v\|_{B^{\frac{d}{2}}}\| r\|_{B^{s_2}}+ (h_3+\|H_3\|_{B^{\frac{d}{2}}})(\| u\|_{B^{s_1+2}}^\ell+\| u\|_{B^{s_2}}^h)\Bigr)+\int_0^t\|S_3\|_{B^{s_2-1}}^h,
\end{align*}
and
\begin{align*}  
\int_0^t\|\partial_t r\|^\ell_{B^{s_1+1}}&\leq \int_0^t\Bigl(\|v\cdot\nabla  r\|_{B^{s_1+1}}^\ell+ \|(h_3+H_3)\div u\|_{B^{s_1+1}}^\ell+ \|S_3\|^\ell_{B^{s_1+1}}\Bigr)\nonumber\\
&\leq C  \int_0^t \Bigl(\|v\|_{B^{\frac{d}{2}}}(\| r\|^\ell_{B^{s_1+2}}+\| r\|^h_{B^{s_2}} )+   (h_3+\|H_3\|_{B^{\frac{d}{2}}})(\| u\|_{B^{s_1+2}}^\ell+\| u\|_{B^{s_2}}^h)\Bigr)\\&\quad+ \int_0^t\|S_3\|_{B^{s_1+1}}^\ell.
\end{align*}
Combine \eqref{Prop4.1-1}, \eqref{es-4.1-u-1} and above two inequalities, keep in mind that $\eta\geq 1,$ we conclude that \eqref{Decay-L} is satisfied.
\end{proof}

{\bf{Completion of the proof of Proposition \ref{Prop-L}}} 
Multiplying \eqref{Decay-L} by a small  constant (far less than $\kappa$)  and adding it to \eqref{L-H-es2}. Then by applying Gronwall's lemma and taking  $c_0$ small enough,  we are able to obtain \eqref{L-H-total}. The proof of Proposition \ref{Prop-L} is completed. \quad$\blacksquare$

\section{Proof of Theorem \ref{Thm-III}}
Here we expose the main arguments one has to use to obtain the existence of a unique global-in-time solution for System \eqref{III}. We follow the scheme explained in details in \cite{CoursDanchinCharve}. As we mentioned in the Remark \ref{Re-Prop-L}, we are not able to consider the case $\eta\to\infty.$ For simplicity, we take $\eta=1$ in the sequel of the paper. In this section we set the value of the regularity indexes $s_1=\frac{d}{2}-1$ and $s_2=\frac{d}{2}+1$, it's the only setting in which we can derive our global existence result. Note that the couple $(s_1,s_2)$ satisfies the conditions in Theorem \ref{Thm-III}.
\subsection{Existence scheme}
Here we expose a classical iterative method to build a solution.
\subsubsection{Iterative existence scheme}

We consider the sequence $(Z^n)_{n\in\mathbb{N}}=(y^n,w^n,r^n,u^n)_{n\in\mathbb{N}}$ with smoothed out initial data $$(y^n_0,w^n_0,r^n_0,u^n_0)=( \dot{S}_ny_0,\dot{S}_nw_0,\dot{S}_nr_0,\dot{S}_nu_0)$$ where $\dot{S}_nf=\sum_{j\leq n}\ddj f_j$ and define the first term of the sequence $Z^0=(0,0,0,0)$. Then, assuming that  $Z^n$ is smooth and globally well defined, we choose $Z^{n+1}$ as the solution
of the following linear system (the existence and uniqueness of  solutions for such system can be found in e.g. \cite{Benzoni-Serre})
\begin{equation}
\left\{\label{Syst:cvg}
\begin{array}
[c]{l}%
\partial_ty^{n+1}+u^n\cdot \nabla y^{n+1}=0,\\
\partial_tw^{n+1}+u^n\cdot \nabla w^{n+1}+\bigl(\bar F_{1}+G_1^n\bigr)\div u^{n+1}+\bigl(\bar F_2+G_2^n\bigr)\dfrac{w^{n+1}}{\nu}=0,\\
\partial_tr^{n+1}+u^n\cdot \nabla r^{n+1}+\bigl(\bar F_{3}+G_3^n\bigr)\div u^{n+1}=F_4^n\dfrac{(w^{n})^2}{\nu},\\
\partial_tu^{n+1}+u^n\cdot \nabla u^{n+1}+ u^{n+1}+(\bar{F}_0+G^n_0)%
\nabla  r^{n+1}+\left(  \gamma_{+}-\gamma_{-}\right) (\bar{F}_0+G_0^n) \nabla w^{n+1}=(\bar{F}_0+G^n_0)\mathcal{A}_{\mu,\lambda}u^{n+1}\\
 \\(y^{n+1},w^{n+1},r^{n+1},u^{n+1})_{t=0}=(y^{n+1}_0,w^{n+1}_0,r^{n+1}_0,u^{n+1}_0)
\end{array}
\right.
\end{equation}
where $G_i^n\triangleq G_i(w^n,r^n,u^n)$ for $i\in\overline{0,3}$.
In the next part, we prove uniform estimates for $Z^n$ in $E^{\frac{d}{2}-1,\frac{d}{2}+1}.$
\subsubsection*{First step: uniform estimates}
We shall use the following classical inductive argument: \\
We claim that there exists constants $c$ and $N$, such that if we assume that  $\|Z_0\|_{B^{\frac{d}{2}-1}\cap B^{\frac{d}{2}+1}}<c$ then for all $n\in \mathbb{N}$, we have \begin{eqnarray} \label{InductiveHyp} \|Z^n\|_{E^{\frac{d}{2}-1,\frac{d}{2}+1}}\leq Nc.\end{eqnarray}
This is obviously true for $n=0$, let's assume that it is true for some fixed $n\in\mathbb{N}$ and prove it for $n+1$.
First, looking at the equation of $y^{n+1}$ and applying proposition \ref{Prop_Transport} with $a=0$ and $f=0$, we get
\[
\left\Vert y^{n+1}\left(  \tau\right)  \right\Vert
_{B^{\frac{d}{2}-1}\cap B^{\frac{d}{2}+1}}+\int_0^t\left\Vert \partial_t y^{n+1}(\tau)  \right\Vert
_{B^{\frac{d}{2}}}d\tau
\leq\left\Vert y_{0}\right\Vert _{B^{\frac{d}{2}-1}\cap B^{\frac{d}{2}+1}}\exp\left(  C\int_{0}^{T}\left\Vert u^{n}\left(
\tau\right)  \right\Vert _{B^{\frac{d}{2}+1}}d\tau\right)
\]
Then, using \eqref{InductiveHyp}, we get 
\begin{eqnarray}\label{eqyn+1}
\left\Vert y^{n+1}\left(  \tau\right)  \right\Vert
_{L^\infty_T(B^{\frac{d}{2}-1}\cap B^{\frac{d}{2}+1})}+\int_0^t\left\Vert \partial_t y^{n+1}(\tau) \right\Vert
_{B^{\frac{d}{2}}}d\tau
\leq ce^{Nc}.
\end{eqnarray}
Then, to recover some estimates for $(w^{n+1},r^{n+1},u^{n+1})$, we need to apply Proposition \ref{Prop-L} to \eqref{Syst:cvg} without the equation of $y^{n+1}$, to do so, we have to show that for all $i\in\overline{0,3}$, we have \begin{eqnarray}\label{ControlHn}\|G^n_i\|_{L^\infty_T(L^\infty)}&\leq&  CNc.
\end{eqnarray} 
To that matter, expressing $\rho$ as a function of $r,w,y$ and using \eqref{InductiveHyp}, we get
\begin{align} \nonumber
\left\Vert \rho^n\right\Vert _{\widetilde{L}_{T}^{\infty}(B^{\frac{d}{2}-1}\cap B^{\frac{d}{2}+1})}  & \leq C(\left\Vert
r^n\right\Vert _{\widetilde{L}_{T}^{\infty}(B^{\frac{d}{2}-1}\cap B^{\frac{d}{2}+1})}+\left\Vert w^n\right\Vert _{\widetilde{L}_{T}^{\infty
}(B^{\frac{d}{2}-1}\cap B^{\frac{d}{2}+1})}+\left\Vert
y^n\right\Vert _{\widetilde{L}_{T}^{\infty}B^{\frac{d}{2}-1}\cap B^{\frac{d}{2}+1})})\label{bounds_rho}\\
& \leq CNc.
\end{align}
Next, as the $G_i$ are smooth functions vanishing at origin, we can apply the composition Lemma \ref{Composition} and, for all $i\in\overline{1,3}$, obtain
\begin{equation}
\left\Vert G_{i}\left(  r^n,w^n,y^n\right)  \right\Vert _{\widetilde{L}_{T}^{\infty
}(B^{\frac{d}{2}-1}\cap B^{\frac{d}{2}+1})}\leq
C\left\Vert \left(  r^n,w^n,y^n\right)  \right\Vert _{\widetilde{L}_{T}^{\infty}(B^{\frac{d}{2}-1}\cap B^{\frac{d}{2}+1})}\leq CNc \label{bounds_Gi_Hs_Y}.
\end{equation}

This combined with the inductive hypothesis \eqref{InductiveHyp} tells us that \eqref{ControlHn} is satisfied.

Therefore, applying Propositions \ref{Prop-L} to \eqref{Syst:cvg} and adding the resulting inequality to \eqref{eqyn+1}, we obtain

\begin{align*}
&\displaystyle\|(y^{n+1},w^{n+1}, r^{n+1}, u^{n+1})(t)\|_{B^{\frac{d}{2}-1}\cap B^{\frac{d}{2}+1}}+\int_0^t \|(w^{n+1}, r^{n+1}, u^{n+1})(\tau)\|_{B^{\frac{d}{2}+1}} \,d\tau\\&\displaystyle+\int_0^t \|\frac{w^{n+1}}{\nu}(\tau)\|_{B^{\frac{d}{2}-1}\cap B^{\frac{d}{2}}}  \,d\tau+\int_0^t \|u^{n+1}(\tau)\|_{B^{\frac{d}{2}}}  \,d\tau
\quad+\int_0^t\|(\partial_{\tau} y^{n+1},\partial_{\tau} w^{n+1} , \partial_t r^{n+1}, \partial_t u^{n+1})(\tau)\|_{B^{\frac{d}{2}}}\,d\tau\\\
& \leq Ce^{CNc}+\displaystyle\exp(CV^n(t))\Bigl(\|(w_0, r_0, u_0)\|_{B^{\frac{d}{2}-1}\cap B^{\frac{d}{2}+1}}+\int^t_0 \|F_4^n\frac{(w^n)^2}{\nu})(\tau)\|_{B^{\frac{d}{2}-1}\cap B^{\frac{d}{2}+1}}\,d\tau\Bigr),
\end{align*}
where $V^n(t)\displaystyle=\int_0^t \bigl(\sum_{i=0}^3\|\partial_t G^n_i(\tau)\|_{B^{\frac{d}{2}}}+ \|v^n(\tau)\|_{B^{\frac{d}{2}}\cap{B^{\frac{d}{2}+1}}}\bigr)\,d\tau.$
Then, in order to obtain the desired estimate, we need to control the right hand side terms using an inductive argument.
Defining $G^n_4:=F_4-\bar{F}_4$, thanks to \eqref{bounds_Gi_Hs_Y} and \eqref{InductiveHyp}, we have
\begin{eqnarray*}\int_0^t\|(F_4^n)\frac{(w^n)^2}{\nu}\|_{B^{\frac{d}{2}-1}\cap B^{\frac{d}{2}+1}}&\leq& C(\bar{F}_4+\|G^n_4\|_{L^\infty_T(B^{\frac{d}{2}})})\|w^n\|_{L^\infty_T(B^{\frac{d}{2}-1}\cap B^{\frac{d}{2}+1})}\|\frac{w^n}{\nu}\|_{L^1_T(B^{\frac{d}{2}})}\\&\leq& C(1+\|G^n_4\|_{L^\infty_T(B^{\frac{d}{2}})})\|Z^n\|^2_{E^{\frac{d}{2}-1,\frac{d}{2}+1}} \\&\leq& CN^2c^2.
\end{eqnarray*}\\
Concerning $V^n$, it is clear that $\displaystyle\int_0^t \|v^n(\tau)\|_{B^{\frac{d}{2}}\cap{B^{\frac{d}{2}+1}}}\,d\tau\leq Cc$. Thus we are left with controlling the terms with time-derivative.
For all $i\in\overline{0,3}$, we have \begin{eqnarray}\partial_tG^n_i=\frac{\partial G^n_{i}}{\partial
y}\partial_{t}y^n+\frac{\partial G^n_i}{\partial r}\partial_{t}%
r^n+\frac{\partial G^n_i}{\partial w}\partial_{t}w^n \label{dtG}
\end{eqnarray}
Using that the $G^n_i$ are smooth functions and \eqref{InductiveHyp}, we obtain
\begin{eqnarray*}\int_0^t\sum_{i=1}^3\|\partial_t G^n_i(t)\|_{B^{\frac{d}{2}}}&\leq&  CNc.
\end{eqnarray*}
Gathering all those estimates, we obtain 
$$\|Z^{n+1}\|_{E^{\frac{d}{2}-1,\frac{d}{2}+1}}+\int_0^t \|\frac{w^{n+1}}{\nu}\|_{B^{\frac{d}{2}-1}\cap B^{\frac{d}{2}}}  + \int_0^t \|\eta u^{n+1}\|_{B^{\frac{d}{2}}}\leq Ce^{CNc}(c+CN^2c^2).$$ Thus, choosing $c$ small enough and a suitable $N$, the inductive hypothesis is fulfilled for $n+1,$ and thus for all $n\in\mathbb{N}$.

\subsubsection*{Second step: Existence of a solution}
\indent Here we show that the sequence $(Z^n)_{n\in\mathbb{N}}$
converges in $\mathcal{D}'(\mathbb{R}_+ \times \mathbb{R}^d )$ to a solution $Z$ of \eqref{III} which has the desired
regularity properties. The following lemma will imply that $(Z^n)_{n\in\mathbb{N}}$ is a Cauchy sequence in a suitable space and also the uniqueness of the solution.

\begin{lemma} \label{LemmaCauchyUniq}Let $U=(y_1,w_1,r_1,u_1)$ and $V=(y_2,w_2,r_2,u_2)$ be two solutions of \eqref{III} having, respectively, $U_0$ and $V_0$ as initial data and such that $U,V\in E_T^{\frac{d}{2}-1,\frac{d}{2}+1}$. We set $\tilde{V}=U-V$, it satisfies
\begin{eqnarray}\label{est:cauchy} \|\tilde{V}\|_{E^{\frac{d}{2}-1,\frac{d}{2}}_T}&\leq C(\|\tilde{V}_0\|_{B^{\frac{d}{2}-1}\cap B^{\frac{d}{2}}}+R(T) \|\tilde{V}\|_{E^{\frac{d}{2}-1,\frac{d}{2}}_T})
\end{eqnarray}
with $\displaystyle R(T)=\|V\|_{L^\infty_T(B^{\frac{d}{2}+1}\cap B ^{\frac{d}{2}-1})}+\frac{1}{\nu}\|w_1\|_{L^1_T(B^{\frac{d}{2}})}+\frac{1}{\nu}\|w_1\|^2_{L^1_T(B^{\frac{d}{2}})}+\|V\|_{L^\infty_T(B^{\frac{d}{2}-1}\cap B^{\frac{d}{2}})}\frac{1}{\nu}\|(w_1,w_2)\|_{L^1_T(B^{\frac{d}{2}})}.$ 
\end{lemma}

\begin{proof}
Let first consider the case  $d\geq 3.$  Observe that $\tilde{V}$ is a solution of
\begin{equation}
\left\{
\begin{array}
[c]{l}%
\partial_t\tilde y+u_1\cdot\nabla \tilde y=\tilde u \nabla y_2,\\
\partial_{t}\tilde w+u_1\cdot\nabla \tilde w+\bigl(\bar F_{1}+G_1(w_1, r_1, y_1)\bigr)\div \tilde u+\bar F_2\tilde w=\mathcal{R}_1+R_1(U)-R_1(V),\\
\partial_{t}\tilde r+u_1\cdot\nabla \tilde r+\bigl(\bar F_{3}+G_3(w_1, r_1, y_1)\bigr)\div \tilde u=\mathcal{R}_2+R_2(U)-R_2(V),\\
\partial_{t}\tilde u+u_1\cdot\nabla  \tilde u-(\bar{F_0} +G_{0,1})\mathcal{A}_{\mu,\lambda} \tilde u+\eta \tilde u+(\bar{F_0} +G_{0,1})%
\nabla  \tilde r+\left(  \gamma_{+}-\gamma_{-}\right) (\bar{F_0} +G_{0,1}) \nabla \tilde w   =\mathcal{R}_3
\end{array}
\right.\label{Tilde}
\end{equation}
where \begin{eqnarray*}&&\mathcal{R}_1=-\tilde u \nabla w_2-(G_1(w_1, r_1, y_1)-G_1(w_2, r_2, y_2)) \div u_2,\\&&\mathcal{R}_2=\tilde u \nabla r_2-(G_3(w_1, r_1, y_1)-G_3(w_2, r_2, y_2)) \div u_2,\\&&\mathcal{R}_3=-\tilde u \nabla u_2-(G_{0,1}-G_{0,2}) \left(-\mathcal{A}_{\mu,\lambda} u_2+\nabla r_2+(\gamma_+-\gamma-)\nabla w_2\right),\\&&R_1(U)=-G_2(w_1, r_1, y_1)\dfrac{w_1}{\nu}\quad \text{and} \quad R_2(U)=-F_4(w_1, r_1, y_1)\dfrac{w^{2}_1}{\nu}.
\end{eqnarray*}

From similar arguments as for system \eqref{Syst:cvg}, we can apply Proposition \ref{Prop-L} to the three last equations of \eqref{Tilde} and Proposition \ref{Prop_Transport} for the equation of $\tilde{y}$ in the case $r_1=\dfrac d2-1$ and $r_2=\dfrac d2$, we obtain
\begin{eqnarray*} \|\tilde{V}\|_{E^{\frac{d}{2}-1,\frac d2}_T}&\leq& \exp\Bigl(C\int_0^t V(\tau)\Bigr)\Bigl(\|\tilde{V}_0\|_{B^{\frac{d}{2}-1}\cap B^{\frac{d}{2}}}\\&&+\int^t_0 \|(\tilde u \nabla y_2,\mathcal{R}_1,R_1(U)-R_1(V),\mathcal{R}_2,R_2(U)-R_2(V), \mathcal{R}_3)\|_{B^{\frac{d}{2}-1}\cap B^{\frac{d}{2}}}\Bigr)
\end{eqnarray*}
where $\displaystyle V(t)=\sum_{i=0}^3\|\partial_t G_i(w_1,r_1,y_1)\|_{B^{\frac{d}{2}}}+ \|u_1\|_{B^{\frac{d}{2}}\cap{B^{\frac{d}{2}+1}}}$.
Again, using from the smoothness of the $G^n_i$ and \eqref{dtG}, it is clear that there exists a $C$ such that
\begin{eqnarray*}\int_0^tV(T)&\leq&  C.
\end{eqnarray*}
Concerning the source terms, since $$R_1(U)-R_1(V)=-G_2(w_1,r_1,y_1)\frac{w_1}{\nu}+G_2(w_2,r_2,y_2)\frac{w_1}{\nu}-G_2(w_2,r_2,y_2)\frac{w_1}{\nu}+G_2(w_2,r_2,y_2)\frac{w_2}{\nu}$$ and similarly for $R_2$, using composition Lemma \ref{Composition}, for $i=1,2$, we obtain \begin{eqnarray*} \|R_i(U)-R_i(V)\|_{B^{\frac{d}{2}}}&\lesssim & \frac{1}{\nu}\|w_1\|_{B^{\frac{d}{2}}}\|\tilde V\|_{B^{\frac{d}{2}}}+\frac{1}{\nu}\|\tilde{w}\|_{B^{\frac{d}{2}}}\|V\|_{B^{\frac{d}{2}}}\\&&+\frac{1}{\nu}\|w_1\|^2_{B^{\frac{d}{2}}}\|\tilde{V}\|_{B^{\frac{d}{2}}}+\frac{1}{\nu}\|\tilde{w}\|_{B^{\frac{d}{2}}}\|(w_1,w_2)\|_{B^{\frac{d}{2}}}\|V\|_{B^{\frac{d}{2}}}    .\end{eqnarray*}
Similarly
 \begin{eqnarray*} \|R_i(U)-R_i(V)\|_{B^{\frac{d}{2}-1}}&\lesssim & \frac{1}{\nu}\|w_1\|_{B^{\frac{d}{2}}}\|\tilde V\|_{B^{\frac{d}{2}-1}}+\frac{1}{\nu}\|\tilde{w}\|_{B^{\frac{d}{2}}}\|V\|_{B^{\frac{d}{2}-1}}\\&&+\frac{1}{\nu}\|w_1\|^2_{B^{\frac{d}{2}}}\|\tilde{V}\|_{B^{\frac{d}{2}-1}}+\frac{1}{\nu}\|\tilde{w}\|_{B^{\frac{d}{2}}}\|(w_1,w_2)\|_{B^{\frac{d}{2}}}\|V\|_{B^{\frac{d}{2}-1}}    .\end{eqnarray*}

Using composition Lemma \ref{Composition}, Corollary 2.66 from \cite{HJR} and product law we obtain
\begin{eqnarray*} \int_0^T\|(\dot{\Delta}_j\mathcal{R}_1,\dot{\Delta}_j\mathcal{R}_2,\dot{\Delta}_j\mathcal{R}_3)\|_{B^{\frac{d}{2}-1}\cap B^{\frac{d}{2}}}\lesssim  \|\tilde{V}\|_{L^\infty_T(B^{\frac{d}{2}-1}\cap B^{\frac{d}{2}})}\|V\|_{E^{\frac{d}{2}-1,\frac{d}{2}+1}}.\end{eqnarray*}

Gathering those estimates, we obtain

\begin{eqnarray*} \|\tilde{V}\|_{E^{\frac{d}{2}-1,\frac{d}{2}}_T}&\leq C(\|\tilde{V}_0\|_{B^{\frac{d}{2}-1}\cap B^{\frac{d}{2}}}+R(T) \|\tilde{V}\|_{E^{\frac{d}{2}-1,\frac{d}{2}}_T}).
\end{eqnarray*}
Which is the desired estimate in the case $d\geq 3$.
The above proof fails for $d=2$ as some right-hand side terms have to be estimated in Besov spaces with a regularity index equal to zero (such as e.g. $G_2(w_2,r_2,y_2)-G_2(w_1,r_1,y_1))$. To overcome this difficulty one must adapt the proof to Chemin-Lerner spaces with third index $r=\infty$ and to estimate the difference of solutions with logarithmic interpolation inequality. For more details you may refer to \cite{HJR} p. 445-447.
\end{proof}

\medbreak
Applying Lemma \ref{LemmaCauchyUniq} with $U=Z^{n}$ and $V=Z^{n+1}$  and using that thanks to the uniform bounds \eqref{InductiveHyp} the right hand side of \eqref{est:cauchy} can be absorbed by its left hand side, we infer that $(Z^n)_{n\in\mathbb{N}}$ is a Cauchy sequence in $E^{\frac{d}{2}-1,\frac{d}{2}+1}$ and therefore there exists a $Z$ such that $(Z^n)$ converges strongly toward $Z$ in $E^{\frac{d}{2}-1,\frac{d}{2}+1}$. 

We are now left with proving that $Z$ is a solution of $\eqref{III}$ and   indeed satisfies the stated regularity properties.
The proof of such results are quite classical, we would like to omit  details here, and advice the reader  to the lecture \cite{CoursDanchinCharve} and the paper \cite{NSCL2} about the study of Navier-Stokes equation.

\subsection*{Third step:  Uniqueness}
As a direct consequence of Lemma \ref{LemmaCauchyUniq}, the following result implies the uniqueness of our solution.

\begin{lemma}
Let $U$ and $V$ be two solutions of \eqref{III} with the same initial data such that $U,V\in E_T^{\frac{d}{2}-1,\frac{d}{2}+1}$. There exists a constant $K>0$ such that if \begin{eqnarray}\label{smallnessuniq}\|V\|_{L^\infty_T(B^{\frac d2-1}\cap B^{\frac d2+1})}\leq K
\end{eqnarray}
then $U\equiv V$.
\end{lemma}
\begin{proof} Let $\tilde{V}=U-V$, since $U_0=V_0$, lemma \ref{LemmaCauchyUniq} implies that

\begin{eqnarray*} \|\tilde{V}\|_{E^{\frac{d}{2}-1,\frac{d}{2}}_T}&\leq CR(T) \|\tilde{V}\|_{E^{\frac{d}{2}-1,\frac{d}{2}}_T}
\end{eqnarray*}
with $\displaystyle R(T)=\|V\|_{L^\infty_T(B^{\frac{d}{2}+1}\cap B ^{\frac{d}{2}-1})}+\frac{1}{\nu}\|w_1\|_{L^1_T(B^{\frac{d}{2}})}+\frac{1}{\nu}\|w_1\|^2_{L^1_T(B^{\frac{d}{2}})}+\|V\|_{L^\infty_T(B^{\frac{d}{2}-1}\cap B^{\frac{d}{2}})}\frac{1}{\nu}\|(w_1,w_2)\|_{L^1_T(B^{\frac{d}{2}})}.$ \\
As we have  $$\textrm{lim sup}_{T\rightarrow0_+} R(T) \leq C\|V\|_{L^\infty_T(B^{\frac{d}{2}+1}\cap B^{\frac{d}{2}-1})},$$
choosing $K$ such that $CK< 1$ we deduce that $\|\tilde{V}\|_{E_T^{\frac{d}{2}-1,\frac{d}{2}}}=0$ for a $T>0$ small enough. Therefore, $U=V$ on $[0,T]$. Then, a classical bootstrap argument allows to show that it is also true for $T=+\infty$.\end{proof}

\section{Relaxation limit}
 In this section we assume $d\geq3$ as they are some limitations for $d=2$ due to negative regularity indexes.
 \subsection{Recovering a solution for the Kapilla system}
 Here we establish the strong convergence locally in space of system \eqref{DBN} to system \eqref{K}, which proves Theorem \ref{Th-1}.
 First of all, with the solution $(\alpha_+, \alpha_-, \rho_+, \rho_-, u)$ described in \eqref{solu-spaces}  
 we will derive the equation of $\alpha_+$ from System \eqref{K} by following an idea of Bresch and Hillairet in \cite{BrHi1}.  In System \eqref{K}, the second equation also reads:
 $$\alpha_+(\partial_t \rho_++ u\cdot\nabla \rho_++ \rho_+\div u)+ \rho_+(\partial_t\alpha_++u\cdot\nabla\alpha_+) =0.$$
 Multiplying this last equation by $P_+'(\rho_+)$ (notice that $\rho_+ P_+'(\rho_+)=\gamma_+P_+(\rho_+)$) we get,
\begin{equation}\label{K-eq-P}
  \alpha_+(\partial_t P_++ u\cdot\nabla P_++ \gamma_+P_+\div u)+ \gamma_+P_+(\partial_t\alpha_++u\cdot\nabla\alpha_+) =0.
  \end{equation}  
Proceeding similarly with $\alpha_-\rho_-$ and subtracting the obtained equation multiplied respectively by $\alpha_-$ and $\alpha_+$, we obtain:
\begin{align*}
\partial_t \alpha_++u\cdot\nabla\alpha_+=-\frac{(\gamma_+-\gamma_-)\alpha_+\alpha_-}{\gamma_+\alpha_-+\gamma_-\alpha_+}\,\div u.
\end{align*}
Taking this equation back to \eqref{K-eq-P},  one finds that  System \eqref{K}-\eqref{initial-data-S1} is equivalent to the following system 
\begin{equation}
\left\{
\begin{array}
[c]{l}%
\alpha_{+}+\alpha_{-}=1,\\
\partial_t \alpha_++u\cdot\nabla\alpha_+=-\dfrac{(\gamma_+-\gamma_-)\alpha_+\alpha_-}{\gamma_+\alpha_-+\gamma_-\alpha_+}\,\div u,\\
\partial_{t} P +u\cdot\nabla P  =-\dfrac{\gamma_+\gamma_-P}{\gamma_+\alpha_-+\gamma_-\alpha_+}\div u,\\
 \rho(\partial_{t} u+ u\cdot\nabla u)+  \nabla P+
\rho u=0,\\
\rho= \alpha_+\rho_{+}+\alpha_-\rho_{-},\\
P=P_{+}\left(  \rho_{+}\right)  =P_{-}\left(  \rho_{-}\right),\\
(\alpha_+ , \alpha_- , P_+ , P_-, u)|_{t=0}=(\alpha_{+0}, \alpha_{-0}, P_{0}, P_{0}, u_{0}).
\end{array}
\right.  \label{BN-0}%
\end{equation}

 For simplicity, we use $P_{\pm}^{\nu} $ to represent $P_{\pm}(\rho_{\pm}^{\nu})$ respectively. 
 Since  the solution $(\alpha_+^\nu, \alpha_-^\nu, \rho_+^\nu, \rho_+^\nu, u^\nu)$ is  regular enough, 
the following  equations for $P_{\pm}^\nu$ can be obtained rigorously like in Section 3, we have
\begin{align*}
\partial_{t}P_{\pm}^{\nu}   +u^\nu\cdot\nabla P_{\pm}^\nu+ \gamma_{\pm}P_{\pm}^{\nu} \operatorname{div}u^{\nu}=\mp\frac
{\gamma_{\pm}\alpha_{\mp}^{\nu}P_{\pm}^{\nu} }{\nu}(P_{+}^{\nu} -P_{-}^{\nu} )
\end{align*}
and thus
 \begin{equation}\label{relation-limit1}
\frac{\alpha_{+}^{\nu} \alpha_{-}^{\nu}}{\nu}\left(
P_{+}^{\nu}-P_{-}^{\nu}\right) =-\frac{\alpha_+^{\nu}\alpha_-^{\nu}}{\gamma_+\alpha_- ^{\nu}P_+^{\nu} + \gamma_-\alpha_+^{\nu} P_-^{\nu}}\Bigl(\partial_t(P_+^{\nu} - P_-^{\nu})+u^{\nu}\cdot\nabla (P_+^{\nu} - P_-^{\nu} )+(\gamma_+P_+^{\nu} - \gamma_-P_-^{\nu})\div u^{\nu}\Bigr).
\end{equation}

  Substituting   equation \eqref{relation-limit1}  to the equation of $\alpha_+^{\nu}, P_{\pm}^\nu,$  we have
 \begin{equation}
 \left\{
\begin{array}
[c]{l}%
\alpha^\nu_++\alpha_-^\nu=1,\\
 \partial_t \alpha_+^{\nu} + u^{\nu}\cdot\nabla \alpha_+^{\nu} =\Gamma_1\Bigl(\partial_t(P_+^{\nu}-P_-^{\nu})+u^{\nu}\cdot\nabla (P_+^{\nu}-P_-^{\nu})\Bigr)-\dfrac{\alpha_+^\nu\alpha_-^\nu(\gamma_+P_+^{\nu}-\gamma_-P_-^{\nu})}{\gamma_+\alpha_-^{\nu} P_+^{\nu} +\gamma_-\alpha_+^{\nu}P_- ^{\nu}}\div u^{\nu},\\
 \partial_{t}P_{+}^{\nu}   +u^\nu\cdot\nabla P_{+}^\nu=\Gamma_2\Bigl(\partial_t(P_+^{\nu}-P_-^{\nu})+u^{\nu}\cdot\nabla (P_+^{\nu}-P_-^{\nu})\Bigr)-\dfrac{\gamma_+\gamma_- P_+^\nu P_-^\nu}{\gamma_+\alpha_-^{\nu} P_+^{\nu} +\gamma_-\alpha_+^{\nu}P_- ^{\nu}}\,\div u^\nu,\\
 \partial_{t} P_{-}^{\nu}   +u^\nu \cdot\nabla P_{-}^\nu=(\Gamma_2-1)\Bigl(\partial_t (P_+^{\nu}-P_-^{\nu})+u^{\nu}\cdot\nabla (P_+^{\nu}-P_-^{\nu})\Bigr)-\dfrac{\gamma_+\gamma_- P_+^\nu P_-^\nu}{\gamma_+\alpha_-^{\nu} P_+^{\nu} +\gamma_-\alpha_+^{\nu}P_- ^{\nu}}\,\div u^\nu,\\
 \rho^\nu (\partial_{t} u^\nu+ u^\nu \cdot\nabla u^\nu)+  \nabla P^\nu+\rho^\nu u^\nu=\mathcal{A}_{\mu, \lambda} u^\nu,\\
\rho^\nu=\alpha_{+}^\nu \rho_{+}^\nu + \alpha_-^\nu \rho_-^\nu,\\
P^\nu=\alpha_+^\nu P_+^\nu+\alpha_-^\nu P_-^\nu,\\
(\alpha_+^\nu , \alpha_-^\nu , P_+^\nu , P_-^\nu, u^\nu)|_{t=0}=(\alpha^\nu_{+0}, \alpha^\nu_{-0}, P^\nu_{+0}, P^\nu_{-0}, u^\nu_{0})
\end{array}\;
\right.\label{BN-nu}
\end{equation}
where
\begin{align*}
\Gamma_1 := -\dfrac{\alpha_+^{\nu}\alpha_-^{\nu}}{\gamma_+\alpha_-^{\nu} P_+^{\nu} +\gamma_-\alpha_+^{\nu}P_- ^{\nu}}, 
\quad\Gamma_2 :=\dfrac{\gamma_+\alpha_-^\nu P_+^\nu}{\gamma_+\alpha_-^{\nu} P_+^{\nu} +\gamma_-\alpha_+^{\nu}P_- ^{\nu}}.
\end{align*}
Here and next,  $\Gamma_i$ are  some regular functions of variables $(\alpha_+^\nu, \alpha_-^\nu, P_+^\nu, P_-^\nu)$  and $\bar\Gamma_i:= \Gamma_i(\bar\alpha_+,~ \bar\alpha_-, ~\bar P,  ~\bar P).$  Thus by Proposition \ref{Composition}, we have
\begin{align}
\|\Gamma_i- \bar\Gamma_i\|_{L^\infty(B^{\frac{d}{2}-1}\cap B^{\frac{d}{2}+1})}\leq C\|(\alpha_+^\nu-\bar\alpha_+, \alpha_-^\nu-\bar \alpha_-, P_+^\nu-\bar P, P_-^\nu-\bar P)\|_{L^\infty(B^{\frac{d}{2}-1}\cap B^{\frac{d}{2}+1})}\leq CM_1. \label{es-Gamma}
\end{align}

From the  bounds \eqref{M-1}, we see that $\dfrac{1} {\nu}(P_{+}^\nu - P_{-}^\nu)$  is uniformly bounded in $L^1(B^{\frac{d}{2}}).$ Therefore, $\partial_t (P_+^\nu-P_-^\nu)$  converges to zero when $\nu$ goes to zero in the sense of distributions, and   the product law $B^{\frac{d}{2}-1}\times B^{\frac{d}{2}}\hookrightarrow B^{\frac{d}{2}-1}$  yields
$$\|u^\nu \cdot\nabla (P_+^\nu - P_-^\nu)\|_{L^1(B^{\frac{d}{2}-1})}\leq C \|u^\nu\|_{L^\infty(B^{\frac{d}{2}})}\|P_+^\nu - P_-^\nu\|_{L^1(B^{\frac{d}{2}})}\to0,~\quad{{\rm{as}}~~\nu\to0.}$$
  In particular, this implies that the first terms in the right-hand sides of equations of $\alpha_+^\nu, P_+^\nu, P_-^\nu$ converge  to zero respectively in the sense of distributions, since it is easy to find that $\Gamma_1, \Gamma_2\in L^\infty(\mathbb{R}_+\times \mathbb{R}^d).$

At this stage,  with the uniform bounds \eqref{M-0} and \eqref{M-1} in hand,  one may perform the classical weak compactness method  to show that  there exists a function $(\alpha_+^0, \alpha_-^0, P_+^0, P_-^0, u^0)$ such that
$$(\alpha_+^0-\bar\alpha_+, \alpha_-^0-\bar\alpha_-, P_+^0-\bar P_+, P_-^0-\bar P_-, u^0)  \in \mathcal{C} (\mathbb{R}^{+};  
{B}^{\frac{d}{2}-1}\cap {B}^{\frac{d}{2}+1}),$$
$$(\alpha_+^{\nu}, \alpha_-^{\nu}, P_+^{\nu}, P_-^{\nu}, u^{\nu})\to (\alpha_+, \alpha_-, P_+^0, P_-^0, u^0)\quad{\rm{in}}~~L^\infty_{{\rm{loc}}}(\mathbb{R}_{+}\times \mathbb{R}^d)~~ {\rm{as}}~~\nu\to0.$$
Moreover, with this strong convergence one can further show that  $(\alpha_+^0, \alpha_-^0, P_+^0, P_-^0, u^0)$ is a solution to the  Cauchy problem \eqref{BN-0}. In virtue of the uniqueness result in Theorem \ref{Th-2}, we conclude that  $(\alpha_+^0, \alpha_-^0, P_+^0, P_-^0, u^0)=(\alpha_+ , \alpha_- , P_+ , P_- , u).$

However, the rate of convergence is not very clear, and we shall work on this direction in the sequel.
\subsection{Convergence rate}
\subsubsection{Presentation of the problem and strategy}
To tackle this problem, let us first define the difference of two solutions by
\begin{align*}
(\delta \alpha_+,~\delta \alpha_-,  \delta \rho_+, ~\delta\rho_-,~\delta u):= (\alpha_+^\nu-\alpha_+, \alpha_-^\nu-\alpha_-,~ \rho_+^\nu-\rho_+,  ~\rho_-^\nu-\rho_-, ~ u^\nu-u).
\end{align*}
We have to admit that it seems hard to obtain decay rate for the terms with $\partial_t(P_+^\nu-P_-^\nu)$ in the system \eqref{BN-nu}, to avoid this problem we will replace the equation of $P_+^\nu$ by the equation of $Q_+^\nu:= P_+^\nu-\Gamma_2 (P_+^\nu-P_-^\nu)$. Moreover we  see that in the equation of $\delta\alpha_+$ there will be a linear higher-order term of $ \delta u$ (i.e. $\div \delta u$),  so instead of the equation of $\delta \alpha_+$, we will consider the equation of   $Y_+^\nu$ and  $Y_+$. In other words, we will consider
the following differences
\begin{align} \label{def:YP}
\delta Y_+ := \dfrac{\alpha_+^\nu\rho_+^\nu}{\alpha_+^\nu\rho_+^\nu+\alpha_-^\nu\rho_+^\nu}-\dfrac{\alpha_+\rho_+}{\alpha_+\rho_++\alpha_-\rho_-},\quad
\delta Q_+:= P_+^\nu-\Gamma_2 (P_+^\nu-P_-^\nu)-P_+,
\end{align}
so that we have the  obvious relationships  
\begin{equation}
\left\{
\begin{array}
[c]{l}%
 \delta\alpha_+ =\delta Y_+ \dfrac{\rho^\nu\rho}{\rho_+^\nu\rho_- } + \dfrac{\alpha_+\alpha_-^\nu\rho_- } {\rho_+^\nu\rho_- } \delta\rho_+ +\dfrac{\alpha_+\alpha_-^\nu\rho_+ } {\rho_+^\nu\rho_- }  \delta\rho_-,\\
\delta P_+:=P_+^\nu-P_+=\delta Q_++\Gamma_2 (P_+^{\nu}-P_-^{\nu}),\\
\delta P_-:=P_-^\nu-P_-=\delta Q_++(\Gamma_2^\nu-1)(P_+^{\nu}-P_-^{\nu}),\\
\delta P:= P^{\nu} - P =\delta Q_+ +(\Gamma_2-\alpha_-^{\nu}) (P_+^{\nu} - P_-^{\nu}).
\end{array}\label{delta-relation1}
\right. 
\end{equation}
Notice that we also have
\begin{equation}
\left\{
\begin{array}
[c]{l}%
\delta\alpha_++\delta\alpha_-=0,\\
\delta\rho_{\pm}= (\frac{1}{A_{\pm}}P_{\pm}^\nu)^{\frac{1}{\gamma_\pm}}- (\frac{1}{A_\pm}P_{\pm})^{\frac{1}{\gamma_\pm}},\\
\delta\rho=(\rho_+^\nu-\rho_-^\nu)\delta\alpha_++\alpha_+\delta\rho_++\alpha_-\delta\rho_-.\label{delta-relation2}
\end{array}
\right.
\end{equation}

Using System \eqref{BN-0} and System \eqref{BN-nu}, we obtain the following system for $(\delta Y_+, \delta Q_+, \delta u)$
 \begin{equation}
 \left\{
\begin{array}
[c]{l}%
 \partial_t \delta Y_+  + u^{\nu}\cdot\nabla \delta Y_+ =\delta S_1,\\
 \\
 \partial_t \delta Q_{+} +u^\nu\cdot\nabla\delta Q_++\bigl( \bar\Gamma_3+(\Gamma_3-\bar\Gamma_3)\bigr)\,\div \delta u=\delta S_2,\\
\\
 \partial_{t} \delta u+  u^\nu\cdot\nabla\delta u+  \delta u + \bigl( \dfrac{1}{\bar\rho} +(\dfrac{1}{\rho^\nu}-\dfrac{1}{\bar\rho })\bigr)\nabla  \delta Q_+=\delta S_3,\\
    \\
(\delta Y_+ ,    \delta Q_+ ,   \delta u)|_{t=0}=(0, 0, 0)
 \end{array}
 \right.\label{BN-diff}
 \end{equation}
where
\begin{align*}
\Gamma_3:=&\dfrac{\gamma_+\gamma_-P_+^\nu P_-^\nu}{\gamma_+\alpha_-^{\nu} P_+^{\nu} +\gamma_-\alpha_+^{\nu}P_- ^{\nu}},\quad 
\Gamma_4:=\dfrac{ \gamma_+P_+^\nu-\gamma_-P_-^\nu}{\gamma_+\alpha_-^\nu P_+^\nu+\gamma_-\alpha_+^\nu P_-^\nu }
\end{align*}
and
\begin{equation*}
\left\{
\begin{array}[c]{l}%
\delta  S_1= -\delta u\cdot\nabla Y_+, \\
 \delta S_2=-\delta u\cdot\nabla P_+-(P_+^\nu-P_-^\nu)(\partial_t \Gamma_2 + u^\nu\cdot\nabla \Gamma_2 )-\dfrac{\gamma_+\gamma_-}{\gamma_+\alpha_-+\gamma_-\alpha_+}(\delta Q_++\Gamma_4\,\delta\alpha_+)\,\div u\\
 \delta S_3=-  \delta u\cdot\nabla u+\dfrac{1}{\rho^\nu}\mathcal{A}_{\mu, \lambda} \,u^\nu+\dfrac{\nabla P_+}{\rho^\nu\rho}\,\delta\rho-\dfrac{1}{\rho^\nu} \nabla \Bigl((\Gamma_2-\alpha_-^\nu)(P_+^{\nu}-P_-^{\nu})\Bigr).
 \end{array}
 \right.
\end{equation*}

At this moment, one may find that the advantages of our choice of $(\delta Y_+, \delta Q_+)$ is that all the differences  can be represented by the quantities $(\delta Y_+, \delta Q_+, \delta u)$ and $P_+^\nu-P_-^\nu$ which will be proven to have  convergence rate of at least $\sqrt{\nu}$.   
Moreover,  the first equation in  System \eqref{BN-diff} is a   transport equation, while the  coupling between the second and the third equation is covered by  Proposition \ref{Prop-L} or more precisely, by eliminating all the $"w"$ factors one can obtain the  following reliable proposition:
\begin{proposition}\label{Prop-ED}
Given functions  $E_1, E_2, E_3, E_4$ and  positive constants $e_1, e_2$    such that
\begin{align*}
E_1, \cdots, E_4, ~v \in \mathcal{C}^1(\mathbb{R}_+ \times\mathcal{S}(\mathbb{R}^d)),\qquad \|E_i\|_{L^\infty(0, T;\,\mathbb{R}^d)}\leq \frac{1}{2}\,e_i, \qquad i=1, 2.
\end{align*}
Let $-\frac{d}{2}< s_1 \leq \frac{d}{2}-1$ and  $s_1\leq s_2-1\leq s_1+1.$  
Suppose  that $(q,  u)$ is a solution of the following linear System \eqref{ED} on time interval $[0, T)$
\begin{equation}
\left\{
\begin{array}
[c]{l}%
\partial_{t} q+v\cdot\nabla q+\bigl(e_1+E_1 \bigr)\div u=E_3,\\
\partial_{t}u+v\cdot\nabla u +  u+\bigl(e_2+E_2 \bigr)%
\nabla  q   = E_4,\\
(q, u)|_{t=0}=(q_0, u_0).
\end{array}
\right.\label{ED}
\end{equation}
There exists a   positive constant $e_0$ depends only on $e_1, e_2 $ and dimension $d$ such that if 
\begin{align*} 
  \| (E_1, E_2)\|_{L^\infty_t(B^{\frac{d}{2}-1}\cap B^{\frac{d}{2}+1})}\leq e_0,
\end{align*}
then the following estimate holds on $[0, T)$
\begin{align*}
&\|(q, u)\|_{ L^\infty_t(B^{s_1})}^{\ell}+\|(q, u)\|_{L^\infty_t (B^{s_2})}^{h} + \|q\|_{ {L}^1_t(B^{s_1+2})}^\ell+\|u\|_{ {L}^1_t(B^{s_1+1})}^\ell + \|(q, u)\|_{ {L}^1_t(B^{s_2})}^h \\
\leq& \exp(CE(t))\Bigl(\|(q_0, u_0)\|_{B^{s_1}}^{\ell}+\|(q_0, u_0)\|_{B^{s_2}}^{h}+  \|(E_3, E_4) \|_{ {L}^1_t(B^{s_1}\cap B^{ s_2})}\Bigr),
\end{align*}
where $\displaystyle E(t)=\int_0^t  \bigl(\| (\partial_t E_1, \partial_t E_2)(\tau)\|_{B^{\frac{d}{2}}}+
 \|v(\tau)\|_{B^{\frac{d}{2}}\cap{B^{\frac{d}{2}+1}}}\bigr)\,.$
\end{proposition}

We are now in the position of stating our convergence result.

 \begin{theorem}\label{Th-limit1}
Let $d\geq3$.  Let $\nu\in(0, 1]$ and assume that \eqref{constants_at_infinity_1} are \eqref{constants_at_infinity_2} are satisfied. Given any $T\in(0, \infty],$ suppose that $(\alpha_+^\nu, \alpha_-^\nu, P_+^\nu, P_-^\nu, u^\nu)$ (resp. $(\alpha_+, \alpha_-, P_+, P_-, u)$) is the solution to the Cauchy problem \eqref{BN-nu} (resp. \eqref{BN-0}) that satisfies
 \begin{eqnarray}\label{est:StabCvg}\|(\delta Y_{+0},\delta Q_{+0},P_{\pm0}^\nu-P_{\pm0},u^\nu_0-u_0)\|_{B^{\frac{d}{2}-\frac{3}{2}} \cap B^{\frac{d}{2}-\frac{1}{2}}}\leq C\sqrt{\nu}. \end{eqnarray}
 
 \noindent where $\delta Y_{+0}$, $\delta Q_{+0}$ are the initial data of the differences $\delta Y_{+}$ and $\delta Q_+$ defined in \eqref{def:YP},
\begin{equation}\label{solu-spaces}
\left\{
\begin{array}
[c]{l}%
(\alpha_{+}^{\nu} -\bar{\alpha}_{+}, \alpha_{-}^{\nu} -\bar{\alpha}_{-}, \rho_{+}^{\nu} -\bar{\rho}_{+}, \rho_{-}^{\nu} -\bar{\rho}_{-}, u^{\nu})\in\mathcal{C}([0, T); 
{B}^{\frac{d}{2}-1}\cap {B}^{\frac{d}{2}+1}),\;\\

(\alpha_{+} -\bar{\alpha}_{+}, \alpha_{-}  -\bar{\alpha}_{-}, \rho_{+} -\bar{\rho}_{+}, \rho_{-} -\bar{\rho}_{-}, u)\in\mathcal{C}([0, T); 
{B}^{\frac{d}{2}-1}\cap {B}^{\frac{d}{2}+1}),
\end{array}
\right.  \;
\end{equation}
and that there exist positive constants    $M_{0}, M_{1}$ independent of $\nu$ such that 
\begin{equation}\label{M-0}
\|(\alpha_{\pm}^\nu-\bar\alpha_{\pm}, \rho_{\pm}^\nu-\bar\rho)\|_{\widetilde{L}^\infty_{T}(B^{\frac{d}{2}-1}\cap  B^{\frac{d}{2}+1} )} +\|(\alpha_{\pm} -\bar\alpha_{\pm}, \rho_{\pm} -\bar\rho)\|_{\widetilde{L}^\infty_{T}(B^{\frac{d}{2}-1}\cap  B^{\frac{d}{2}+1} )}  \leq M_{0},
\end{equation}
\begin{eqnarray}\label{M-1}
&&\|\dfrac{1}{\nu} (P_+^\nu-P_-^\nu)\|_{ \widetilde{L}^1_T(B^{\frac{d}{2}-1}\cap B^{\frac{d}{2}})}+\|~\mathcal{A}_{\mu, \lambda}\, u^{\nu}\|_{ \widetilde{L}^1_T(B^{\frac{d}{2}}) }+ \|P_+^\nu-P_-^\nu\|_{  \widetilde{L}^1_T(B^{\frac{d}{2}+1}) }\\&&\quad+ \|(\partial_t\alpha_{\pm}^\nu, ~\partial_t\rho_{\pm}^\nu,  ~ \nabla u^\nu, ~\nabla u, )\|_{  \widetilde{L}^1_T(B^{\frac{d}{2}-1}\cap B^{\frac{d}{2}})}\nonumber+\|(u^\nu, ~u)\|_{\widetilde{L}^\infty_{T}(B^{\frac{d}{2}-1}\cap  B^{\frac{d}{2}+1} )} \leq M_{1}.
\end{eqnarray}
Then there exists a constant $C$ such that we have the following   estimate   for all  $t\in[0, T),$
\begin{align*}
\|(\delta Y_+,  \delta Q_+, u^\nu-u)\|_{L^\infty_T (B^{\frac{d}{2}-\frac{3}{2}}\cap B^{\frac{d}{2}-\frac{1}{2})}}&+   \|\delta Q_+\|^\ell_{L^1_T(B^{\frac{d}{2}+\frac{1}{2}})}\\&+\|\delta Q_+\|^h_{L^1_T( B^{\frac{d}{2}-\frac{1}{2}})}+\|\delta u\|_{L^1_T( B^{\frac{d}{2}-\frac{1}{2}})}\leq\sqrt{\nu}\,\exp(CM_1^4 ).
\end{align*}
Above, the constants $M_0,~M_1$ and $C$ depend only on $\bar\alpha_{\pm}, ~\bar\rho_{\pm},~ \bar P$ and the dimension $d$.
  \end{theorem}
 \begin{remark}
We shall  suppose that   $M_0$ is small   (say $M_0\ll1$ without loss of generality), however $M_1$ is not necessarily to be small.     
 \end{remark}
 \begin{remark}
 The assumption \eqref{est:StabCvg} can be lowered to $\mathcal{O}(\nu^\gamma)$ with $\gamma>0$ but then we would end up with a convergence rate equal to $\displaystyle\nu^{\beta}$ with $\beta=\min\{\frac12,\gamma\}$.
 \end{remark}
 
 The rest of this section is devoted to the proof of  above Theorem.

Now,  it is clear that the  System \eqref{BN-diff} can be looked as a linear system with given convection velocity $u^\nu$ and coefficients that fall in the range of application of Proposition \ref{Prop-ED}. Indeed,  under  condition \eqref{M-0}  with small enough $M_0,$ the 
assumptions presented in Proposition \ref{Prop-ED} are satisfied.  
Thus one only needs to find the appropriate regularity indexes in low and high frequencies  so that  all source terms are bounded and "have decay" appropriately. This is the purpose of the following lines.

Notice that  by interpolation inequality and Young's inequality we have
\begin{align*}
\|\mathcal{A}_{\mu, \lambda}\, u^{\nu}\|_{B^{\frac{d}{2}-\frac{1}{2}}}&\leq \|\mathcal{A}_{\mu, \lambda}\, u^{\nu}\|_{B^{\frac{d}{2}}}^{\frac{1}{2}}\,\|\mathcal{A}_{\mu, \lambda}\, u^{\nu}\|_{B^{\frac{d}{2}-1}}^{\frac{1}{2}}\\
&\leq \sqrt{\nu}\,\|\mathcal{A}_{\mu, \lambda}\, u^{\nu}\|_{B^{\frac{d}{2}}}^{\frac{1}{2}}  \, \|\Delta  u^\nu\|_{B^{\frac{d}{2}-1}}^{\frac{1}{2}}\\
&\leq \sqrt{\nu}\,\bigl(\|\mathcal{A}_{\mu, \lambda}\, u^{\nu}\|_{B^{\frac{d}{2}}}+  \|  u^\nu\|_{B^{\frac{d}{2}+1}}\bigr).
\end{align*}
Then the uniform bounds \eqref{M-1}  imply that
\begin{align}
 \|\mathcal{A}_{\mu, \lambda}\, u^{\nu}\|_{ {L}^1(B^{\frac{d}{2}-\frac{1}{2}})}\leq \sqrt{\nu}\,\Bigl(\|\mathcal{A}_{\mu, \lambda}\, u^{\nu}\|_{ {L}^1(B^{\frac{d}{2}})}+  \|  u^\nu\|_{ {L}^1(B^{\frac{d}{2}+1})}\Bigr)\leq M_{11}\sqrt{\nu},\label{convergence-Delta-u}
\end{align}
and \begin{align*}
\|\mathcal{A}_{\mu, \lambda}\, u^{\nu}\|_{B^{\frac{d}{2}-\frac{3}{2}}}&\leq \|\mathcal{A}_{\mu, \lambda}\, u^{\nu}\|_{B^{\frac{d}{2}-1}}^{\frac{1}{2}}\,\|\mathcal{A}_{\mu, \lambda}\, u^{\nu}\|_{B^{\frac{d}{2}-2}}^{\frac{1}{2}}\\
&\leq \sqrt{\nu}\,\|\mathcal{A}_{\mu, \lambda}\, u^{\nu}\|_{B^{\frac{d}{2}-1}}^{\frac{1}{2}}  \, \|\Delta  u^\nu\|_{B^{\frac{d}{2}-2}}^{\frac{1}{2}}\\
&\leq \sqrt{\nu}\,\bigl(\|\mathcal{A}_{\mu, \lambda}\, u^{\nu}\|_{B^{\frac{d}{2}-1}}+  \|  u^\nu\|_{B^{\frac{d}{2}}}\bigr).
\end{align*}
Then the uniform bounds \eqref{M-1}  imply that
\begin{align}
 \|\mathcal{A}_{\mu, \lambda}\, u^{\nu}\|_{ {L}^1(B^{\frac{d}{2}-\frac{3}{2}})}\leq \sqrt{\nu}\,\Bigl(\|\mathcal{A}_{\mu, \lambda}\, u^{\nu}\|_{ {L}^1(B^{\frac{d}{2}-1})}+  \|  u^\nu\|_{ {L}^1(B^{\frac{d}{2}})}\Bigr)\leq M_{11}\sqrt{\nu}.\label{convergence-Delta-u-l}
\end{align}
\begin{remark}

In fact, for any $\theta\in(0, 1)$ we have
\begin{align*}
\|\mathcal{A}_{\mu, \lambda}\, u^{\nu}\|_{ {L}^1(B^{\frac{d}{2}-\theta})} &\leq  \|\mathcal{A}_{\mu, \lambda}\, u^{\nu}\|_{ {L}^1(B^{\frac{d}{2}-1})}^{\theta}
\|\mathcal{A}_{\mu, \lambda}\, u^{\nu}\|_{ {L}^1(B^{\frac{d}{2}})}^{1-\theta}\\
&\leq  \Bigl(\nu\,\|\Delta\, u^{\nu}\|_{ {L}^1(B^{\frac{d}{2}-1})}\Bigr)^{\theta}\|\mathcal{A}_{\mu, \lambda}\, u^{\nu}\|_{ {L}^1(B^{\frac{d}{2}})}^{1-\theta}\\
&\leq \nu^{\theta} \|  u^{\nu}\|_{ {L}^1(B^{\frac{d}{2}+1})} ^{\theta}\|\mathcal{A}_{\mu, \lambda}\, u^{\nu}\|_{ {L}^1(B^{\frac{d}{2}})}^{1-\theta}\leq C(\theta)M_{1}\,\nu^{\theta}.
\end{align*}
So this means that we could get a better decay rate for this term, like $\nu^\theta$ with $\theta\in[\frac{1}{2},1)$. However, the terms with $\alpha$ will have a decay rate bounded by $\sqrt{\nu}$ so it is not relevant to use this idea until decay rate for the volume fraction is improved.
\end{remark}

From the bounds \eqref{M-1} we know that $\|P_+^\nu-P_-^\nu\|_{L^1(B^{\frac{d}{2}-1}\cap B^{\frac{d}{2}})}$ converges to zero at rate $\nu,$ but the convergence rates in the $ {L}^\infty$-in-time based spaces are not clear. We have to emphasis that this kind of convergence rates are very important, since when handling the source term  $(P_+^\nu-P_-^\nu)\partial_t\Gamma_2$ lying in $\delta S_2$, one could only get $\partial_t\Gamma_2\in  {L}^1(B^{\frac{d}{2}-1}\cap B^{\frac{d}{2}})$ from  the bounds \eqref{M-1}, therefore obtaining convergence rate in $L^\infty$-in-time for $P_+^\nu-P_-^\nu$ is necessary.
In fact we have such a result in Chemin-Lerner type spaces.
\begin{proposition}
Let $(P_+^\nu,P_-^\nu)$ satisfying the condition of Theorem \ref{Th-limit1}, we have
\begin{align}
\|P_+^\nu-P_-^\nu\|_{\widetilde{L}^\infty(B^{\frac{d}{2}-\frac{3}{2}} \cap B^{\frac{d}{2}-\frac{1}{2}})} 
  \leq C{\sqrt{\nu}}.\label{convergence-P}
\end{align}


\end{proposition}
\begin{proof}
It is easy to check that
\begin{align*}
\|{\gamma_+\alpha_- ^{\nu}P_+^{\nu} + \gamma_-\alpha_+^{\nu} P_-^{\nu}}-({\gamma_+\bar\alpha_-  \bar P + \gamma_-\bar\alpha_+  \bar P })\|_{\widetilde{L}^\infty(B^{\frac{d}{2}})}
&\leq  C \|(\alpha_+ ^{\nu}-\bar\alpha_+, \alpha_- ^{\nu}-\bar\alpha_-,  P_+^{\nu}-\bar P,  P_-^{\nu}-\bar P)\|_{\widetilde{L}^\infty(B^{\frac{d}{2}})}\\
&\leq CM_0
\end{align*}
is small enough. So we can apply Proposition \ref{Prop-D} to \eqref{relation-limit1} with $s\in[\frac{d}{2}-\frac{3}{2}, ~\frac{d}{2}-\frac{1}{2}]$, one has
\begin{align}
 \|P_+^\nu-P_-^\nu\|_{\widetilde{L}^\infty(B^s)} 
  \leq& C \|P_{+0}^\nu-P_{-0}^\nu\|_{B^s}\notag\\&+C{\sqrt{\nu}} \Bigl(\|u^\nu\nabla (P_+^\nu-P_-^\nu)\|_{\widetilde{L}^2(B^s)}    +   \|(\gamma_+P_+^\nu-\gamma_-P_-^\nu)\,\div u^\nu\|_{\widetilde{L}^2(B^s)}\Bigr).\label{convergence-P1}
\end{align}
Notice that thanks to \eqref{est:StabCvg} and the fact that $P_{+0}=P_{-0}$, we have 
$\|P_{+0}^\nu-P_{-0}^\nu\|_{B^s}\leq C\sqrt{\nu}$.
Now, recalling bounds \eqref{M-0} and \eqref{M-1},
using product law
\begin{align}
 B^{\frac{d}{2}}\times B^{\frac{d}{2}-\frac{3}{2}}\hookrightarrow B^{\frac{d}{2}-\frac{3}{2}}\label{p-law-61}
\end{align}
and interpolation inequalities we have  
\begin{align*}
\|u^\nu\cdot\nabla (P_+^\nu-P_-^\nu)\|_{\widetilde{L}^2(B^{\frac{d}{2}-\frac{3}{2}})} &\leq C\,\|u^\nu\|_{\widetilde{L}^\infty (B^{\frac{d}{2}})}\|\nabla (P_+^\nu-P_-^\nu)\|_{\widetilde{L}^2(B^{\frac{d}{2}-\frac{3}{2}})} \\
&\leq C\,\|u^\nu\|_{\widetilde{L}^\infty (B^{\frac{d}{2}})}\| P_+^\nu-P_-^\nu\|_{\widetilde{L}^1(B^{\frac{d}{2}})}^{\frac{1}{2}} \| P_+^\nu-P_-^\nu\|_{\widetilde{L}^\infty(B^{\frac{d}{2}-1})}^{\frac{1}{2}} \leq  C M_{0}^{\frac{1}{2}}M_{1}^{\frac{3}{2}}
\end{align*}
and
\begin{align*}
\|(\gamma_+&P_+^\nu-\gamma_-P_-^\nu)\,\div u^\nu\|_{\widetilde{L}^2(B^{\frac{d}{2}-\frac{3}{2}})}\\
 \leq&  \gamma_+ \|(P_+^\nu-\bar P)\,\div u^\nu\|_{\widetilde{L}^2(B^{\frac{d}{2}-\frac{3}{2}})}+ \gamma_-\|(P_-^\nu-\bar P)\,\div u^\nu\|_{\widetilde{L}^2(B^{\frac{d}{2}-\frac{3}{2}})}+ (\gamma_++\gamma_-)\bar P\,\|\div u^\nu\|_{\widetilde{L}^2(B^{\frac{d}{2}-\frac{3}{2}})}\\
 \leq& C \Bigl( \gamma_+\|(P_-^\nu-\bar P)\|_{\widetilde{L}^\infty (B^{\frac{d}{2}})}  +\gamma_-\|(P_+^\nu-\bar P)\|_{\widetilde{L}^\infty (B^{\frac{d}{2}})} + (\gamma_++\gamma_-)\bar P\Bigr)\,\|\div u^\nu\|_{\widetilde{L}^2(B^{\frac{d}{2}-\frac{3}{2}})}\\
 \leq& C (\gamma_++\gamma_-) ( M_0+\bar P )\,\|  u^\nu\|_{\widetilde{L}^1(B^{\frac{d}{2}})}^{\frac{1}{2}}\|  u^\nu\|_{\widetilde{L}^\infty(B^{\frac{d}{2}-1})}^{\frac{1}{2}}\leq C (\gamma_++\gamma_-) ( M_0+\bar P )\,M_{1}.
\end{align*}
Once again, using the following product law
\begin{align}
 B^{\frac{d}{2}}\times B^{\frac{d}{2}-\frac{1}{2}}\hookrightarrow B^{\frac{d}{2}-\frac{1}{2}}\label{p-law-62}
\end{align}
and interpolation inequalities we have
\begin{align*}
\|u^\nu\cdot\nabla (P_+^\nu-P_-^\nu)\|_{\widetilde{L}^2(B^{\frac{d}{2}-\frac{1}{2}})} &\leq C\,\|u^\nu\|_{\widetilde{L}^\infty (B^{\frac{d}{2}})}\|\nabla (P_+^\nu-P_-^\nu)\|_{\widetilde{L}^2(B^{\frac{d}{2}-\frac{1}{2}})} \\
&\leq C\,\|u^\nu\|_{\widetilde{L}^\infty (B^{\frac{d}{2}})}\| P_+^\nu-P_-^\nu\|_{\widetilde{L}^1(B^{\frac{d}{2}+1})}^{\frac{1}{2}} \| P_+^\nu-P_-^\nu\|_{\widetilde{L}^\infty(B^{\frac{d}{2}})}^{\frac{1}{2}} \leq  C M_{0}^{\frac{1}{2}}M_{1}^{\frac{3}{2}}
\end{align*}
and
\begin{align*}
\|(\gamma_+&P_+^\nu-\gamma_-P_-^\nu)\,\div u^\nu\|_{\widetilde{L}^2(B^{\frac{d}{2}-\frac{1}{2}})}\\
 \leq&  \gamma_+ \|(P_+^\nu-\bar P)\,\div u^\nu\|_{\widetilde{L}^2(B^{\frac{d}{2}-\frac{1}{2}})}+ \gamma_-\|(P_-^\nu-\bar P)\,\div u^\nu\|_{\widetilde{L}^2(B^{\frac{d}{2}-\frac{1}{2}})}+ (\gamma_++\gamma_-)\bar P\,\|\div u^\nu\|_{\widetilde{L}^2(B^{\frac{d}{2}-\frac{1}{2}})}\\
 \leq& C \Bigl( \gamma_+\|(P_-^\nu-\bar P)\|_{\widetilde{L}^\infty (B^{\frac{d}{2}})}  +\gamma_-\|(P_+^\nu-\bar P)\|_{\widetilde{L}^\infty (B^{\frac{d}{2}})} + (\gamma_++\gamma_-)\bar P\Bigr)\,\|\div u^\nu\|_{\widetilde{L}^2(B^{\frac{d}{2}-\frac{1}{2}})}\\
 \leq& C (\gamma_++\gamma_-) ( M_0+\bar P )\,\|  u^\nu\|_{\widetilde{L}^1(B^{\frac{d}{2}+1})}^{\frac{1}{2}}\|  u^\nu\|_{\widetilde{L}^\infty(B^{\frac{d}{2}})}^{\frac{1}{2}}\leq C (\gamma_++\gamma_-) ( M_0+\bar P )\,M_{1}.  
\end{align*}
Plugging the last four inequalities into \eqref{convergence-P1}, one gets the desired inequality.
\end{proof}

In conclusion, from this preliminary analysis it seems that the couple $s_1=\frac{d}{2}-\frac{3}{2}$ and $s_2=\frac{d}{2}-\frac{1}{2}$ is a good choice when applying Proposition \ref{Prop-ED} to the System \eqref{BN-diff}.
\subsubsection{Derivation of the convergence rate}
Applying Proposition \ref{Prop-ED} to the equations of $\delta Q_+$ and $\delta u$ in \eqref{BN-diff} with $s_1=\frac{d}{2}-\frac{3}{2}$ and $s_2=\frac{d}{2}-\frac{1}{2}$, Proposition \ref{Prop_Transport} to the equation of $\delta Y_+$ and summing the resulting estimates together, we obtain

\begin{align}
\|(\delta Y_+,\delta Q_+,\delta u)\|&_{ \widetilde{L}^\infty_t(B^{\frac{d}{2}-\frac{3}{2}} \cap B^{\frac{d}{2}-\frac{1}{2}})} + \|(\delta Q_+,\delta u)\|^h_{ {L}^1_t(B^{\frac{d}{2}-\frac{1}{2}})} +\|\delta Q_+\|_{ {L}^1_t(B^{\frac{d}{2}+\frac{1}{2}})}^\ell+ \|\delta u\|_{ {L}^1_t(B^{\frac{d}{2}-\frac{1}{2}})}^\ell\notag \\
\leq& \exp(CE(t))\left(\|(\delta Y_{+0},\delta Q_{+0},\delta u_0)\|_{B^{\frac{d}{2}-\frac{3}{2}} \cap B^{\frac{d}{2}-\frac{1}{2}}}+  \|(\delta S_1,\delta S_2,\delta S_3) \|_{ {L}^1_t(B^{\frac{d}{2}-\frac{3}{2}}\cap B^{\frac{d}{2}-\frac{1}{2}})}\right), \label{CVGBGron}
\end{align}
where $\displaystyle E(t)=\int_0^t  \bigl(\| (\partial_t (\Gamma_3-\bar\Gamma_3), \partial_t (\frac{1}{\rho^\nu}-\frac{1}{\bar{\rho}}))\|_{B^{\frac{d}{2}}}+
 \|u^\nu\|_{B^{\frac{d}{2}}\cap{B^{\frac{d}{2}+1}}}\bigr)\,.$
 
 From the bounds \eqref{M-0} and \eqref{M-1}, $E(t)$ is clearly uniformly bounded and from the hypothesis \eqref{est:StabCvg}, we have $\|(\delta Y_{+0},\delta Q_{+0},\delta u_0)\|_{ B^{\frac{d}{2}-\frac{3}{2}} \cap B^{\frac{d}{2}-\frac{1}{2}}}\leq C\sqrt{\nu}.$
 Let us now check that the source terms lie in $\delta S_1,~\delta S_2$ and $\delta S_3$  in the spaces $ L^1_t(B^{\frac{d}{2}-\frac{3}{2}}\cap ^{\frac{d}{2}-\frac{1}{2}}).$ To do so, we have to split the analysis for the low and high frequencies.
 
\subsubsection*{Low regularity estimates}
Here we us repeatedly the product law Proposition \ref{Productlaw}. Concerning $\delta S_1$, we have
\begin{align}
\|\delta u\cdot\nabla Y_+\|_{   B^{\frac{d}{2}-\frac{3}{2}}}&\leq C\|\delta u\|_{ B^{\frac{d}{2}-\frac{1}{2}}}\,\|\nabla (Y_+-\bar Y_+)\|_{ B^{\frac{d}{2}-1}}\nonumber\\
&\leq C\|\delta u\|_{ B^{\frac{d}{2}-\frac{1}{2}}}\,\| Y_+-\bar Y_+\|_{  B^{\frac{d}{2}}}\leq C M_{0}\|\delta u\|_{  B^{\frac{d}{2}-\frac{1}{2}}}.\label{es-delta-S1}\\
\end{align}

For $\delta S_2$ and $\delta S_3$, we have
\begin{align}
\|\delta u\cdot\nabla P_+\|_{  B^{\frac{d}{2}-\frac{3}{2}}}&\leq C\|\delta u\|_{  B^{\frac{d}{2}-\frac{1}{2}}}\,\|\nabla (P_+-\bar P_+)\|_{  B^{\frac{d}{2}-1}}\notag\\
&\leq C\|\delta u\|_{ B^{\frac{d}{2}-\frac{1}{2}}}\,\| P_+-\bar P\|_{  B^{\frac{d}{2}}}\leq C M_{0}\|\delta u\|_{  B^{\frac{d}{2}-\frac{1}{2}}},\label{es-delta-S2}
\end{align}
\begin{align}
\|\delta u\cdot\nabla  u\|_{ B^{\frac{d}{2}-\frac{3}{2}}}&\leq C\|\delta u\|_{  B^{\frac{d}{2}-\frac{1}{2}}}\,\|\nabla u\|_{  B^{\frac{d}{2}-1}}\leq C\|\delta u\|_{  B^{\frac{d}{2}-\frac{1}{2}}}\,\| u\|_{  B^{\frac{d}{2}}},    \label{es-delta-S3}
\end{align}
and 
\begin{align}
\|(P_+^\nu-P_-^\nu)(\partial_t  \Gamma_2+ u^\nu\cdot\nabla \Gamma_2)\|_{ B^{\frac{d}{2}-\frac{3}{2}}} \leq&
C\|P_+^\nu-P_-^\nu\|_{  B^{\frac{d}{2}-\frac{1}{2}}}\,\|\partial_t \Gamma_2\|_{  B^{\frac{d}{2}-1}}\\&+C\|P_+^\nu-P_-^\nu\|_{  B^{\frac{d}{2}-\frac{1}{2}}}\| u^\nu\cdot\nabla \Gamma_2\|_{  B^{\frac{d}{2}-1}}.\notag
\end{align}

Thanks to the estimate \eqref{convergence-P} and bounds \eqref{M-1}, by product law \eqref{p-law-62}  the last inequality can be  written
\begin{align}
\quad\|&(P_+^\nu-P_-^\nu)(\partial_t  \Gamma_2+ u^\nu\cdot\nabla \Gamma_2)\|_{ B^{\frac{d}{2}-\frac{3}{2}}}\notag\\
&\leq C{\sqrt{\nu}}\,M_{1}\,\| \partial_t \Gamma_2\|_{  B^{\frac{d}{2}-1}}+ \|P_+^\nu-P_-^\nu\|_{  B^{\frac{d}{2}-\frac{1}{2}}}\|\Gamma_2-\bar\Gamma_2\|_{B^{\frac{d}{2}}}\|u^\nu\|_{B^{\frac{d}{2}}}.\label{es-delta-S4}
\end{align}
Before estimating  the last term in $\delta S_2,$   we need the following lemma. \begin{lemma}
Let $d\geq3$. For $s$ such that $\frac{d}{2}-\frac{3}{2}< s \leq \frac{d}{2}-\frac 12$, we have
\begin{eqnarray}
\|(\delta\alpha_+,~\delta\alpha_-,~\delta\rho_+,~\delta\rho_-,~\delta\rho, \delta P)\|_{B^{s}}
\leq C(M_0+1)\,\Bigl(\|(\delta Y_+, \delta Q_+)\|_{B^{s}}+ \|P_+^\nu-P_-^\nu\|_{B^{s}}\Bigr).\label{es-delta}
\end{eqnarray}
\end{lemma}
\begin{proof}
Thanks to the relations in \eqref{delta-relation1},   by product law \eqref{p-law-62} we have
\begin{eqnarray}
\|(\delta P_+,~\delta P_-)\|_{ B^{s}}&\leq& C\|\delta Q_+\|_{ B^{s}} +C \bigl(\|\Gamma_2-\bar\Gamma_2\|_{L^\infty(B^{\frac{d}{2}})}\|P_+^{\nu}-P_-^{\nu}\|_{B^{s}}+\bar\Gamma_2 \,\|P_+^{\nu}-P_-^{\nu}\|_{B^{s}}\bigr)\notag\\
&\leq& C\|\delta Q_+\|_{ B^{s}}+ C(M_0+1)\|P_+^{\nu}-P_-^{\nu}\|_{B^{s}}.\label{es-delta-P}
\end{eqnarray}
Using  relations in \eqref{delta-relation2} and a decomposition argument similar to \eqref{es-delta-P} yields
\begin{align}
\|\delta \alpha_+\|_{ B^{s}}\leq C(M_0 +1)\,\|(\delta Y_+, \delta \rho_+, \delta\rho_-)\|_{B^{s}}.\label{es-delta-alpha}
\end{align}
Define $f(x)=(\frac{1}{A_+}x)^{\frac{1}{\gamma_+}}.$ We are now going to control $\delta\rho_+$ by $\delta P_+$ but the composition Proposition \ref{Composition} cannot be applied readily since $f'(0)\neq0.$ Still, we can rewrite $\delta\rho_+$ as
\begin{align*}
\delta\rho_+=\Bigl(f(x+\bar P)-f'(\bar P)\,x\Bigr)\Big{|}^{P_+^\nu-\bar P}_{P_+-\bar P}+\Bigl(f'(\bar P)\,x\Bigr)\Big{|}^{P_+^\nu-\bar P}_{P_+-\bar P}=\Bigl(f(x+\bar P)-f'(\bar P)\,x\Bigr)\Big{|}^{P_+^\nu-\bar P}_{P_+-\bar P}+ f'(\bar P)\,\delta P_+,
\end{align*}
then, since $s>\frac{d}{2}-\frac{3}{2}\geq 0$ as $d\geq3$, we can apply the composition Proposition \ref{Composition} to the first quantity of right-hand side and obtain
\begin{align*}
\|\delta\rho_+\|_{B^{s}}&\leq f'(\bar P)\|\delta P_+\|_{B^s}+C \|(P_+^\nu-\bar P, ~P_+-\bar P)\|_{B^{s}}\|\delta P_+\|_{B^{s}}\\
&\leq C(M_0 +1)\|\delta P_+\|_{B^{s}}.
\end{align*}
Proceeding similarly with $\delta\rho_-,$ combining with \eqref{es-delta-alpha}  we conclude that \eqref{es-delta} is satisfied.
\end{proof}

Now, we are ready to use product laws \eqref{p-law-61}, \eqref{p-law-62} and composition lemma  to  write
\begin{align*}
\|&\frac{\gamma_+ \gamma_-}{\gamma_+\alpha_- + \gamma_- \alpha_+}(\delta Q_++\Gamma_4 \delta \alpha_+)\,\div u\|_{B^{\frac{d}{2}-\frac{3}{2}}}\notag\\
\leq& \Bigl(\|\frac{\gamma_+ \gamma_-}{\gamma_+\alpha_- + \gamma_- \alpha_+}-\frac{\gamma_+\gamma_-}{\gamma_+\bar{\alpha}_-+\gamma_-\bar{\alpha}_+}
\|_{B^{\frac{d}{2}}}+1\Bigr)\Bigl(\|\delta Q_+ \,\div u\|_{B^{\frac{d}{2}-\frac{3}{2}}}+
\| \Gamma_4 \, \delta \alpha_{+} \div u\|_{B^{\frac{d}{2}-\frac{3}{2}}}\Bigr)\\
\leq &C (M_0+1)\Bigl(\|\delta Q_+\|_{B^{\frac{d}{2}-\frac{1}{2}}} \,\|\div u\|_{B^{\frac{d}{2}-1}}+\|\Gamma_4-\bar{\Gamma}_4\|_{B^{\frac{d}{2}-1}} \,  \|\delta \alpha_{+} \div u\|_{B^{\frac{d}{2}-\frac{1}{2}}}+\bar{\Gamma}_4\|\delta \alpha_{+} \div u\|_{B^{\frac{d}{2}-\frac{3}{2}}}\Bigr)\notag
\\ \leq & C(M_0+1)\Bigl(\|\delta Q_+\|_{B^{\frac{d}{2}-\frac{1}{2}}} \,\|\div u\|_{B^{\frac{d}{2}-1}}+M_0 \,  \|\delta \alpha_{+} \|_{B^{\frac{d}{2}-\frac{1}{2}}} \|\div u\|_{B^{\frac{d}{2}}}+\bar\Gamma_4\|\delta \alpha_{+} \|_{B^{\frac{d}{2}-\frac{1}{2}}} \|\div u\|_{B^{\frac{d}{2}-1}}\Bigr)\notag
 \\\leq  & C(M_0+1)^2\Bigl(\|\delta Q_+\|_{B^{\frac{d}{2}-\frac{1}{2}}}   +\|\delta \alpha_{+} \|_{ B^{\frac{d}{2}-\frac{1}{2}}} \Bigr)\| u^\nu\|_{B^{\frac{d}{2}}\cap B^{\frac{d}{2}+1}}.
 \end{align*}
 Combining this with \eqref{es-delta}, we obtain
 \begin{align}
\|&\frac{\gamma_+ \gamma_-}{\gamma_+\alpha_- + \gamma_- \alpha_+}(\delta Q_++\Gamma_4 \delta \alpha_+)\,\div u\|_{B^{\frac{d}{2}-\frac{3}{2}}}\notag\\
\leq& C (M_{0}+1)^3\Bigl(\|(\delta Y_+, \delta Q_+)\|_{ B^{\frac{d}{2}-\frac{1}{2}}} +\|P_+^\nu-P_-^\nu\|_{ B^{\frac{d}{2}-\frac{1}{2}}} \Bigr)\| u^\nu\|_{B^{\frac{d}{2}}\cap B^{\frac{d}{2}+1}}.
\label{es-delta-S5}
 \end{align}
Using  a product law and \eqref{convergence-Delta-u-l}, we have
 \begin{align}
 \|\frac{1}{\rho^\nu}\mathcal{A}_{\mu, \lambda}\, u^\nu\|_{B^{\frac{d}{2}-\frac{3}{2}}}&\leq  C\|\frac{1}{\rho^\nu}-\frac{1}{\bar\rho}\|_{B^{\frac{d}{2}-1}}\,\|\mathcal{A}_{\mu, \lambda}\, u^\nu\|_{B^{\frac{d}{2}-\frac{3}{2}}}+C\,\|\mathcal{A}_{\mu, \lambda}\, u^\nu\|_{B^{\frac{d}{2}-\frac{3}{2}}}\notag\\
 &\leq  C\,\nu\,(M_0+1)\,\| u^\nu\|_{ B^{\frac{d}{2}+\frac{1}{2}}}\leq  C\,\nu\,(M_0+1)\,(\| u^\nu\|_{B^{\frac{d}{2}}}^{\frac{1}{2}}+\| u^\nu\|_{B^{\frac{d}{2}+1}}^{\frac{1}{2}}),\label{es-delta-S6}
 \end{align}
 and by \eqref{es-delta},
 \begin{align}
 \|\frac{\nabla P_+}{\rho^\nu\rho}\delta\rho\|_{B^{\frac{d}{2}-\frac{3}{2}}} \leq& C\|\nabla (P_+-\bar P)\|_{B^{\frac{d}{2}-1}} \|\frac{1}{\rho^\nu\rho} \delta \rho\|_{B^{\frac{d}{2}-\frac{1}{2}}}\notag\\
 \leq& C \| P_+-\bar P\|_{B^{\frac{d}{2}}}(\|\frac{1}{\rho^\nu\rho}-\frac{1}{\bar\rho^2}\|_{B^{\frac{d}{2}}}+\frac{1}{\bar\rho^2}) \|\delta \rho\|_{B^{\frac{d}{2}-\frac{1}{2}}}\notag\\
  \leq& CM_0(M_0+1)\|\delta \rho\|_{B^{\frac{d}{2}-\frac{1}{2}}}\notag\\
 \leq& CM_0(M_0+1)\Bigl(\|(\delta Y_+, \delta Q_+)\|_{B^{\frac{d}{2}-\frac{1}{2}}} +\|P_+^\nu-P_-^\nu\|_{B^{\frac{d}{2}-\frac{1}{2}}}\Bigr).\label{es-delta-S7}
\end{align}
We estimate the last term of $\delta S_3$ in the following way
\begin{align}
\quad\|&\frac{1}{\rho^\nu}\nabla\Bigl((\Gamma_2-\alpha_-^\nu)(P_+^\nu-P_-^\nu)\Bigr)\|_{B^{\frac{d}{2}-\frac{3}{2}}}\notag\\
&\leq C\Bigl(\|\frac{1}{\rho^\nu}-\frac{1}{\bar\rho}\|_{B^{\frac{d}{2}}}+\frac{1}{\bar\rho}\Bigr) \,\| (\Gamma_2-\alpha_-^\nu)\,(P_+^\nu-P_-^\nu) \|_{\frac{d}{2}-\frac{1}{2}}\notag\\
&\leq  C(M_0+1)\Bigl(\| \Gamma_2-\bar{\Gamma}_2\|_{B^{\frac{d}{2}}}+ \|\alpha_-^\nu-\bar\alpha_-\|_{B^{\frac{d}{2}}}+\bar{\Gamma}_2+ \bar\alpha_-^\nu\Bigr)\|P_+^\nu-P_-^\nu \|_{\frac{d}{2}-\frac{1}{2}}\notag\\
&\leq  C(M_0+1)^2\|P_+^\nu-P_-^\nu \|_{B^{\frac{d}{2}-1}\cap B^{\frac{d}{2}}}.\label{es-delta-S8}
\end{align} 

Summing up \eqref{es-delta-S1}-\eqref{es-delta-S4} and \eqref{es-delta-S5}-\eqref{es-delta-S8} together,  integrating over $[0, t]$ on the both sides of the resulting inequality, we conclude that (without loss of generality assume that $M_0\ll1$)
\begin{align}
 \|(\delta S_1&, \delta S_2, \delta S_3)\|_{L^1_t(B^{\frac{d}{2}-\frac{3}{2}})}\notag\\
\leq& C  M_0 \int_0^t \|\delta u(\tau)\|_{B^{\frac{d}{2}-\frac{1}{2}}}\,+ C \int_0^t\bigl(\|\delta u\|_{B^{\frac{d}{2}-\frac{3}{2}}}+ \|(\delta Y_+, \delta Q_+)(\tau)\|_{B^{\frac{d}{2}-\frac{1}{2}}}\bigr)\|u^\nu(\tau)\|_{B^{\frac{d}{2}}\cap B^{\frac{d}{2}+1}}\,\notag\\
&+ CM_{1}^2  \int_0^t \sqrt{\nu}\|\partial_t \Gamma_2 (\tau)\|_{B^{\frac{d}{2}-1}} + \nu   \|u^\nu(\tau)\|_{B^{\frac{d}{2}}\cap B^{\frac{d}{2}+1}} + \|P_+^\nu-P_-^\nu\|_{B^{\frac{d}{2}-1}\cap B^{\frac{d}{2}}}\Bigr)\,\notag \\
\leq& C  M_0 \int_0^t \|\delta u(\tau)\|_{B^{\frac{d}{2}-\frac{1}{2}}}\,+ C \int_0^t\bigl(\|\delta u\|_{B^{\frac{d}{2}-\frac{3}{2}}}+ \|(\delta Y_+, \delta Q_+)(\tau)\|_{B^{\frac{d}{2}-\frac{1}{2}}}\bigr)\|u^\nu(\tau)\|_{B^{\frac{d}{2}}\cap B^{\frac{d}{2}+1}}\,\notag\\
&+ CM_{1}^3  \sqrt{\nu},\label{es-delta-S-low}
\end{align}
where we used that $\|\partial_t \Gamma_2\|_{B^{\frac{d}{2}-1}}\leq M_1$ which can be directly obtained from writing the equation verified by $\partial_t \Gamma_2$.
\subsubsection*{High regularity estimates}
We will now show how to control the same terms in $L^1_T(B^{\frac{d}{2}-\frac{1}{2}})$.
Concerning $\delta S_1$, we have
\begin{align}
\|\delta u\cdot\nabla Y_+\|_{   B^{\frac{d}{2}-\frac{1}{2}}}&\leq C\|\delta u\|_{ B^{\frac{d}{2}-\frac{1}{2}}}\,\|\nabla (Y_+-\bar Y_+)\|_{ B^{\frac{d}{2}}}\notag\\
&\leq C\|\delta u\|_{ B^{\frac{d}{2}-\frac{1}{2}}}\,\| Y_+-\bar Y_+\|_{  B^{\frac{d}{2}+1}}\leq C M_{0}\|\delta u\|_{  B^{\frac{d}{2}-\frac{1}{2}}}.\label{es-delta-S1-h}
\end{align}
Similarly as before, we have
\begin{align}
\|\delta u\cdot\nabla P_+\|_{  B^{\frac{d}{2}-\frac{1}{2}}}&\leq C M_{0}\|\delta u\|_{  B^{\frac{d}{2}-\frac{1}{2}}},\label{es-delta-S2-h}
\end{align} \begin{align}
\|\delta u\cdot u\|_{  B^{\frac{d}{2}-\frac{1}{2}}}&\leq C\|\delta u\|_{  B^{\frac{d}{2}-\frac{1}{2}}}\|u\|_{  B^{\frac{d}{2}}},
\end{align}
and
\begin{align}
\|(P_+^\nu-P_-^\nu)(\partial_t \Gamma_2+ u^\nu\cdot\nabla \Gamma_2)\|_{ B^{\frac{d}{2}-\frac{1}{2}}} \leq&
C\|P_+^\nu-P_-^\nu\|_{  B^{\frac{d}{2}-\frac{1}{2}}}\,\|\partial_t \Gamma_2\|_{  B^{\frac{d}{2}}}+C\|P_+^\nu-P_-^\nu\|_{  B^{\frac{d}{2}-\frac{1}{2}}}\| u^\nu\cdot\nabla \Gamma_2\|_{  B^{\frac{d}{2}}}.\notag
\end{align}

Thanks to the estimate \eqref{convergence-P} and bounds \eqref{M-1}, by a product law the last inequality can be written
\begin{align}
\quad\|&(P_+^\nu-P_-^\nu)(\partial_t  \Gamma_2+ u^\nu\cdot\nabla \Gamma_2)\|_{ B^{\frac{d}{2}-\frac{1}{2}}}\notag\\
&\leq C{\sqrt{\nu}}\,M_{1}\,\| \partial_t \Gamma_2\|_{  B^{\frac{d}{2}}}+ \|P_+^\nu-P_-^\nu\|_{  B^{\frac{d}{2}-\frac{1}{2}}}\|\Gamma_2-\bar\Gamma_2\|_{B^{\frac{d}{2}}}\|u^\nu\|_{B^{\frac{d}{2}}}.\label{es-delta-S4-h}
\end{align}
We have,
\begin{align*}
\|&\frac{\gamma_+ \gamma_-}{\gamma_+\alpha_- + \gamma_- \alpha_+}(\delta Q_++\Gamma_4 \delta \alpha_+)\,\div u\|_{B^{\frac{d}{2}-\frac{1}{2}}}\notag\\
\leq& \Bigl(\|\frac{\gamma_+ \gamma_-}{\gamma_+\alpha_- + \gamma_- \alpha_+}-\frac{\gamma_+\gamma_-}{\gamma_+\bar{\alpha}_-+\gamma_-\bar{\alpha}_+}
\|_{B^{\frac{d}{2}}}+1\Bigr)\Bigl(\|\delta Q_+ \,\div u\|_{B^{\frac{d}{2}-\frac{1}{2}}}+
\| \Gamma_4 \, \delta \alpha_{+} \div u\|_{B^{\frac{d}{2}-\frac{1}{2}}}\Bigr)\\
\leq &C (M_0+1)\Bigl(\|\delta Q_+\|_{B^{\frac{d}{2}-\frac{1}{2}}} \,\|\div u\|_{B^{\frac{d}{2}}}+\|\Gamma_4-\bar{\Gamma}_4\|_{B^{\frac{d}{2}}} \,  \|\delta \alpha_{+} \div u\|_{B^{\frac{d}{2}-\frac{1}{2}}}+\bar{\Gamma}_4\|\delta \alpha_{+} \div u\|_{B^{\frac{d}{2}-\frac{1}{2}}}\Bigr)\notag
\\ \leq & C(M_0+1)\Bigl(\|\delta Q_+\|_{B^{\frac{d}{2}-\frac{1}{2}}} \,\|\div u\|_{B^{\frac{d}{2}}}+M_0 \,  \|\delta \alpha_{+} \|_{B^{\frac{d}{2}-\frac{1}{2}}} \|\div u\|_{B^{\frac{d}{2}}}+\bar\Gamma_4\|\delta \alpha_{+} \|_{B^{\frac{d}{2}-\frac{1}{2}}} \|\div u\|_{B^{\frac{d}{2}}}\Bigr)\notag
 \\\leq  & C(M_0+1)^2\Bigl(\|\delta Q_+\|_{B^{\frac{d}{2}-\frac{1}{2}}}   +\|\delta \alpha_{+} \|_{ B^{\frac{d}{2}-\frac{1}{2}}} \Bigr)\| u^\nu\|_{B^{\frac{d}{2}+1}}.
 \end{align*}
 Combining with \eqref{es-delta} gives that
 \begin{align}
\|&\frac{\gamma_+ \gamma_-}{\gamma_+\alpha_- + \gamma_- \alpha_+}(\delta Q_++\Gamma_4 \delta \alpha_+)\,\div u\|_{B^{\frac{d}{2}-\frac{1}{2}}}\notag\\
\leq& C (M_{0}+1)^3\Bigl(\|(\delta Y_+, \delta Q_+)\|_{ B^{\frac{d}{2}-\frac{1}{2}}} +\|P_+^\nu-P_-^\nu\|_{ B^{\frac{d}{2}-\frac{1}{2}}} \Bigr)\| u^\nu\|_{B^{\frac{d}{2}+1}}.
\label{es-delta-S5-h}
 \end{align}
Using a product law and \eqref{convergence-Delta-u} yields
 \begin{align}
 \|\frac{1}{\rho^\nu}\mathcal{A}_{\mu, \lambda}\, u^\nu\|_{B^{\frac{d}{2}-\frac{1}{2}}}&\leq  C\|\frac{1}{\rho^\nu}-\frac{1}{\bar\rho}\|_{B^{\frac{d}{2}-1}}\,\|\mathcal{A}_{\mu, \lambda}\, u^\nu\|_{B^{\frac{d}{2}-\frac{1}{2}}}+C\,\|\mathcal{A}_{\mu, \lambda}\, u^\nu\|_{B^{\frac{d}{2}-\frac{1}{2}}}\notag\\
 &\leq  C\,\nu\,(M_0+1)\| u^\nu\|_{B^{\frac{d}{2}}\cap B^{\frac{d}{2}+1}}+,\label{es-delta-S6-h}
 \end{align}
 and by \eqref{es-delta},
 \begin{align}
 \|\frac{\nabla P_+}{\rho^\nu\rho}\delta\rho\|_{B^{\frac{d}{2}-\frac{1}{2}}} \leq& C\|\nabla (P_+-\bar P)\|_{B^{\frac{d}{2}}} \|\frac{1}{\rho^\nu\rho} \delta \rho\|_{B^{\frac{d}{2}-\frac{1}{2}}}\notag\\
 \leq& C \| P_+-\bar P\|_{B^{\frac{d}{2}+1}}(\|\frac{1}{\rho^\nu\rho}-\frac{1}{\bar\rho^2}\|_{B^\frac{d}{2}}+\frac{1}{\bar\rho^2}) \|\delta \rho\|_{B^{\frac{d}{2}}}\notag\\
  \leq& CM_0(M_0+1)\|\delta \rho\|_{B^{\frac{d}{2}-\frac{1}{2}}}\notag\\
 \leq& CM_0(M_0+1)\Bigl(\|(\delta Y_+, \delta Q_+)\|_{B^{\frac{d}{2}-\frac{1}{2}}} +\|P_+^\nu-P_-^\nu\|_{B^{\frac{d}{2}-\frac{1}{2}}}\Bigr).\label{es-delta-S7-h}
\end{align}
For the last term in $\delta S_3$ we  estimate in the following way
\begin{align}
\quad\|&\frac{1}{\rho^\nu}\nabla\Bigl((\Gamma_2-\alpha_-^\nu)(P_+^\nu-P_-^\nu)\Bigr)\|_{B^{\frac{d}{2}-\frac{1}{2}}}\notag\\
&\leq C\Bigl(\|\frac{1}{\rho^\nu}-\frac{1}{\bar\rho}\|_{B^{\frac{d}{2}}}+\frac{1}{\bar\rho}\Bigr) \,\| (\Gamma_2-\alpha_-^\nu)\,(P_+^\nu-P_-^\nu) \|_{B^{\frac{d}{2}+\frac{1}{2}}}\notag\\
&\leq  C(M_0+1)^2\|P_+^\nu-P_-^\nu \|_{B^\frac{d}{2} \cap B^{\frac{d}{2}+\frac{1}{2}}}
.\label{es-delta-S8-h}
\end{align} 
Using interpolation inequality we have 
\begin{eqnarray*}
\|P_+^\nu-P_-^\nu \|_{B^{\frac{d}{2}+\frac{1}{2}}}&\leq& \|P_+^\nu-P_-^\nu \|_{B^{\frac{d}{2}}}^{\frac 12}\|P_+^\nu-P_-^\nu \|_{B^{\frac{d}{2}+1}}^{\frac 12} \\ &\leq& M_1\|P_+^\nu-P_-^\nu \|_{B^{\frac{d}{2}}}^{\frac 12}.
\end{eqnarray*}

Summing up \eqref{es-delta-S1-h}-\eqref{es-delta-S4-h} and \eqref{es-delta-S5-h}-\eqref{es-delta-S8-h} together,  integrating over $[0, t]$ on the both sides of the resulting inequality, we conclude that (without loss of generality assume that $M_0\ll1$)
\begin{align}
\|(\delta S_1 &, \delta S_2, \delta S_3)\|_{L^1_t(B^{\frac{d}{2}-\frac{1}{2}})}\notag\\
\leq& C  M_0 \int_0^t \|\delta u(\tau)\|_{B^{\frac{d}{2}-\frac{1}{2}}}\,+ C \int_0^t\bigl(\|\delta u\|_{B^{\frac{d}{2}-\frac{1}{2}}}+ \|(\delta Y_+, \delta Q_+)(\tau)\|_{B^{\frac{d}{2}-\frac{1}{2}}}\bigr)\|u^\nu(\tau)\|_{B^{\frac{d}{2}}\cap B^{\frac{d}{2}+1}}\,\notag\\
&+ CM_{1}^2  \int_0^t \sqrt{\nu}\|\partial_t \Gamma_2 (\tau)\|_{B^{\frac{d}{2}}} + \nu   \|u^\nu(\tau)\|_{B^{\frac{d}{2}+1}} + \|P_+^\nu-P_-^\nu\|_{B^{\frac{d}{2}-1}\cap B^{\frac{d}{2}}}^{\frac{1}{2}}\Bigr)\,\notag \\
\leq& C  M_0 \int_0^t \|\delta u(\tau)\|_{B^{\frac{d}{2}-\frac{1}{2}}}\,+ C \int_0^t\bigl(\|\delta u\|_{B^{\frac{d}{2}-\frac{1}{2}}}+ \|(\delta Y_+, \delta Q_+)(\tau)\|_{B^{\frac{d}{2}-\frac{1}{2}}}\bigr)\|u^\nu(\tau)\|_{B^{\frac{d}{2}}\cap B^{\frac{d}{2}+1}}\,\notag\\
&+ CM_{1}^3M_0  \sqrt{\nu}.\label{es-delta-S-high}
\end{align}
where we used again that $\|\partial_t \Gamma_2\|_{B^{\frac{d}{2}}\cap B^{\frac{d}{2}+1}}\leq M_1$.

Finally, gathering \eqref{es-delta-S-low} and \eqref{es-delta-S-high}, we obtain
\begin{eqnarray}
 \|(\delta S_1, \delta S_2, \delta S_3)\|_{L^1_t(B^{\frac{d}{2}-\frac{3}{2}}\cap B^{\frac{d}{2}-\frac{1}{2}})}
&\leq& C  M_0 \int_0^t \|\delta u\|_{B^{\frac{d}{2}-\frac{1}{2}}}\,+ C \int_0^t\|\delta u\|_{B^{\frac{d}{2}-\frac{3}{2}}\cap B^{\frac{d}{2}-\frac{1}{2}}} \|u^\nu\|_{B^{\frac{d}{2}}\cap B^{\frac{d}{2}+1}}\notag\\
&&+ C \int_0^t\|(\delta Y_+, \delta Q_+)\|_{B^{\frac{d}{2}-\frac{1}{2}}}\bigr\|u^\nu\|_{B^{\frac{d}{2}}\cap B^{\frac{d}{2}+1}}\notag\\
&&+ CM_{1}^3  \sqrt{\nu}.\label{es-delta-S-highlow}
\end{eqnarray}

Therefore, using \eqref{es-delta-S-highlow} and \eqref{est:StabCvg} in \eqref{CVGBGron} and Gronwall's lemma yields
\begin{align*}
 &\|( \delta Y_+,  \delta Q_+,  \delta u)(t)\|_{ B^{\frac{d}{2}-\frac{3}{2}}} + \|( \delta Y_+,  \delta Q_+,  \delta u)(t)\|_{ B^{\frac{d}{2}-\frac{1}{2}}} + \int_0^t \|\delta Q_+\|^\ell_{ B^{\frac{d}{2}+\frac{1}{2}}}+ \int_0^t \|\delta Q_+\|^h_{ B^{\frac{d}{2}-\frac{1}{2}}}\\&+ \int_0^t \|\delta u\|_{ B^{\frac{d}{2}-\frac{1}{2}}}
\leq \,\sqrt{\nu} \, \exp(CM_1^4 ).
\end{align*}
Thus the proof of Theorem \ref{Th-limit1} is completed. \quad$\blacksquare$

\appendix 
\section{Some basic linear problems}
\begin{proposition}\label{Prop-D}
Let $s\in(-\frac{d}{2}, \frac{d}{2}]$ and $H_2\in\mathcal{C}([0, T); \mathcal{S}).$ Let $w$ be a solution of the following  damped equation with variable coefficient  on $[0, T),$   
\begin{equation}
\left\{
\begin{array}
[c]{l}%
\partial_t w+  (h_2+ H_2)\, \dfrac{w}{\nu}=f,\\
w(t, x)|_{t=0}=w_0.
\end{array}
\right.\tag{D}\label{D}
\end{equation}
 There exists a positive constant $\bar h_2$  such that if 
\begin{align*}
\|H_2\|_{\widetilde{L}^\infty_T(B^{\frac{d}{2}})}\leq \bar h_2<\dfrac{h_2}{2},
\end{align*}
then  the following estimate holds on $[0, T),$
\begin{align*}
  \frac{1}{\sqrt{\nu}}  \|w\|_{\widetilde{L}^\infty_t(B^s)}+\frac{1}{2\sqrt{h_2}}\|\partial_t w\|_{\widetilde{L}^2_t(B^s)}\leq  \frac{1}{\sqrt{\nu}}\|w_0\|_{B^s}+\max(\frac{1}{2\sqrt{h_2}},4h_2) \,\|f\|_{\widetilde{L}^2_t(B^s)}.
\end{align*}
\end{proposition}
\begin{proof}
We first rewrite the equation into the form
$$\bigl(\dfrac{1}{h_2+H_2}-\dfrac{1}{h_2}+\dfrac{1}{h_2}\bigr)\,\partial_t w + \dfrac{w}{\nu}=\frac{f}{h_2+H_2}.$$ 
Applying the operator $\ddj$ to it, we get
\begin{align*}
    \frac{1}{h_2} \partial_t w_j + \frac{w_j}{\nu}= -\ddj\bigl((\dfrac{1}{h_2+H_2}-\dfrac{1}{h_2})\partial_t w)\bigr)+\ddj(\frac{f}{h_2+H_2}).
\end{align*}
Taking $L^2$ inner product  with $\partial_t w_j$ leads to
\begin{align*}
 \frac{1}{2\nu}\frac{d}{dt} \|w_j\|_{L^2}^2+\frac{1}{h_2}\|\partial_t w_j\|^2 \leq  \|\ddj((\frac{1}{h_2+H_2}-\frac{1}{h_2}) \partial_t w)\|_{L^2} \,\|\partial_t w_j\|_{L^2}+ \|\ddj (\frac{f}{h_2+H_2})\|_{L^2}\,\|\partial_t w_j\|_{L^2}.
\end{align*}
On the right-hand side we use Young's inequality,   it becomes
\begin{align*}
   \frac{1}{2\nu}\frac{d}{dt} \|w_j\|_{L^2}^2+\frac{1}{h_2}\|\partial_t w_j\|_{L^2}^2 
    \leq  4h_2 \|\ddj((\frac{1}{h_2+H_2}-\frac{1}{h_2}) \partial_t w)\|_{L^2}^2+4h_2\|\ddj (\frac{f}{h_2+H_2})\|_{L^2}^2+ \frac{1}{2h_2}\|\partial_t w_j\|_{L^2}^2.
\end{align*}
Integrating on time interval $[0, t]$ for any $t\in[0, T)$, and using Young's inequality again, we readily obtain 
\begin{align*}
     \frac{1}{\sqrt{\nu}}  \|w_j(t)\|_{L^2}+\frac{1}{\sqrt{h_2}}\|\partial_t w_j\|_{L^2_t(L^2)}  
    &\leq \frac{1}{\sqrt{\nu}}  \|w_0\|_{L^2}+4\sqrt{h_2}\,\|\ddj\bigl((\frac{1}{h_2+H_2}-\frac{1}{h_2}) \partial_t w\bigr)\|_{L^2_t(L^2)}\\&+ 4\sqrt{h_2}\|\ddj (\frac{f}{h_2+H_2})\|_{L^2_t(L^2)}.
\end{align*}
Multiplying both sides by $2^{js}$ and summing up in $j\in\mathbb{Z}$, we get
\begin{align*}
  \frac{1}{\sqrt{\nu}}  \|w\|_{\widetilde{L}^\infty_t(B^s)}+\frac{1}{\sqrt{h_2}}\|\partial_t w\|_{\widetilde{L}^2_t(B^s)}&\leq \frac{1}{\sqrt{\nu}} \|w_0\|_{B^s}+4\sqrt{h_2}\, \|(\frac{1}{h_2+H_2}-\frac{1}{h_2}) \partial_t w\|_{\widetilde{L}^2_t(B^s)}\\&+ 4\sqrt{h_2}\, \| \frac{f}{h_2+H_2}\|_{\widetilde{L}^2_t(B^s)}.
\end{align*}
The product law $B^s\times B^{\frac{d}{2}}\hookrightarrow B^s$  and composition lemma entail that we have
\begin{align*}
4\sqrt{h_2}\, \|(\frac{1}{h_2+H_2}-\frac{1}{h_2}) \partial_t w\|_{\widetilde{L}^2_t(B^s)}\leq 4\sqrt{h_2}\,  C\|H_2\|_{\widetilde{L}^\infty_t(B^{\frac{d}{2}})}\|\partial_t w\|_{\widetilde{L}^2_t(B^s)}\leq    \frac{1}{2 \sqrt{h_2} }   \, \|\partial_t w\|_{\widetilde{L}^2_t(B^s)}
\end{align*}
whenever $\bar h_2\leq \dfrac{1}{8Ch_2}.$
Handling the other right hand side term in a similar manner completes the proof of the proposition.
\end{proof}

Recall that the constant coefficient transport-damping equation reads:
\begin{equation}
\left\{
\begin{array}
[c]{l}%
\partial_t w+ v\cdot\nabla w+ a\, w=f,\\
w(t, x)|_{t=0}=w_0.
\end{array}
\right.\label{TD}
\end{equation}
We are going to prove the following proposition.
\begin{proposition} \label{Prop_Transport}
Let $r_1\in (-\frac{d}{2}, \frac{d}{2}], r_2\in (-\frac{d}{2}, \frac{d}{2}+1]$ Assume that $a\geq0$.   Then there exists a universal constant $C$ such that
\begin{equation}\label{TD-1}
\|w(t)\|_{B^{r_2}}^h+ a\,\int_0^t {\|w\|^h_{B^{r_2}}} \,\leq\|w_0\|^h_{B^{r_2}}+C\int^t_0 \|\nabla v\|_{B^{\frac{d}{2}}}\|w\|_{B^{r_2}}\,+\int_0^t \|f\|_{B^{r_2}}^h\,
\end{equation}
 and
\begin{align}
\|w(t)\|_{B^{r_1}}^\ell+ a\,\int_0^t \|w\|_{B^{r_1}}^\ell \,\leq\|w_0\|_{B^{r_1}}^\ell+C\int^t_0 \|v\|_{B^{\frac{d}{2}}}\|w\|_{B^{r_1+1}}\,+\int_0^t \|f\|_{B^{r_1}}^\ell\,. \label{TD-2}
\end{align}
\end{proposition}

\begin{proof}
Applying $\ddj$ to \eqref{TD} yields
\begin{align*}
\partial_t w_j + v\cdot\nabla w_j  + a\,w=- [\ddj, v]\nabla w+\ddj f.
\end{align*}
Taking the $ L^2$ inner product with $w_j$ then  by  integration by parts,  one has
\begin{align*}
\frac{1}{2}\frac{d}{dt}\|w_j\|_{L^2}^2+a\,\|w_j\|_{L^2}^2 &=\frac{1}{2}\int_{\mathbb{R}^d} \div v |w_j|^2 - \int_{\mathbb{R}^d}w_j\,[\ddj, v]\nabla w+\int_{\mathbb{R}^d} w_j,\ddj f\\
&\leq \|\div v\|_{L^\infty}\|w_j\|_{L^2}^2+ \|w_j\|_{L^2}(\|[\ddj, v]\nabla w\|_{L^2}+\|\ddj f\|_{L^2}).
\end{align*}
As $-\frac{d}{2}<r_2\leq \frac{d}{2}+1$, by  Proposition \ref{Commutator}, one has
\begin{align*}
\|[\ddj, v]\nabla w\|_{L^2}\leq C2^{-jr_2}q_j\|\nabla v\|_{B^{\frac{d}{2}}}\|w\|_{B^{r_2}}.
\end{align*}
Thus by embedding $B^{\frac{d}{2}}\hookrightarrow L^\infty,$  we  obtain
\begin{align*}
\|w_j(t)\|_{L^2}+a\, \int_0^t \|w_j(\tau)\|_{L^2} \,\leq\|w_j(0)\|_{L^2} +C2^{-jr_2}q_j\int^t_0  \|\nabla v(\tau)\|_{B^{\frac{d}{2}}} \|w\|_{B^{r_2}}\,+ \int_0^t  \|\ddj f(\tau)\|_{L^2}\,.
\end{align*}
Multiplying the factor $2^{jr_2}$ on both sides, summing up over $j\geq-1$, we get \eqref{TD-1}

For the low frequencies estimate \eqref{TD-2}, one  needs to take  $v\cdot\nabla w$ as a source term, and use the product law $\|v\cdot \nabla w\|_{B^{r_1}}\leq C\|v\|_{B^\frac{d}{2}}\|\nabla w\|_{B^{r_1}},$ for $r_1\in(-\frac{d}{2}, \frac{d}{2}].$ 
\end{proof}

\medskip
Recall that the constant coefficient Lamé system reads:
\begin{equation}
\left\{
\begin{array}
[c]{l}%
\partial_t u+ v\cdot\nabla u-\mu\Delta u-(\lambda+\mu)\nabla\div u+\eta u=g,\\
u(t, x)|_{t=0}=u_0.
\end{array}
\right.\label{Lamé}
\end{equation}
We are going to prove the following proposition.
\begin{proposition}\label{P-Lame}
Let $r_1, r_2\in (-\frac{d}{2}, \frac{d}{2} ].$ Assume that $\mu>0$ and  $\mu+\lambda,\eta\geq 0.$ Then there exists a universal constant $C$ such that
\begin{multline*}
\|u(t) \|^\ell_{B^{r_1}}+\|u(t)\|^h_{B^{r_2}}+\mu\int_0^t(\|\nabla u\|_{B^{{r_1}+1}}^\ell+\|\nabla u\|_{B^{{r_2}+1}}^h)\\+ (\lambda+\mu)\int_0^t(\|\div u\|_{B^{{r_1}+1}}^\ell+\|\div u\|_{B^{{r_2}+1}}^h)
 +\eta\int_0^t(\|u \|^\ell_{B^{r_1}}+\|u\|_{B^{r_2}}^h)\\
\leq C(\|u_0\|^\ell_{B^{r_1}}+\|u_0\|^h_{B^{r_2}})+ C\int_0^t(\|\nabla v\|_{B^{\frac{d}{2}}}\| u\|_{B^{r_1}\cap B^{r_2}}+\|v\|_{B^\frac{d}{2}}\|u\|_{B^{r_1+1}\cap B^{r_2+1}})
+ C\int_0^t(\| g\|^\ell_{B^{r_1}}+\| g\|^h_{B^{r_2}})\,.
\end{multline*}
\end{proposition}

\begin{proof}
Applying $\ddj$ to equation \eqref{Lamé} yields
\begin{align}
\partial_t u_j  -\mu\Delta u_j-(\lambda+\mu)\nabla\div u_j+\eta u_j=-[\ddj, v] \nabla u-v\cdot\nabla u_j+\ddj g. \label{Lame-local-u-eq}
\end{align}
Taking the $ L^2$ inner product with $u_j$ then integrate by parts. One gets (note that $\lambda+\mu\geq 0$)
\begin{align*}
\frac{1}{2}\frac{d}{dt}\|u_j\|_{L^2}^2+\mu\|\nabla u_j\|_{L^2}^2  +\eta\|u_j\|_{L^2}^2&\leq\frac{1}{2}\int_{\mathbb{R}^d} \div v \,|u_j|^2 - \int_{\mathbb{R}^d}u_j\,[\ddj,  v] \nabla u+\int_{\mathbb{R}^d} u_j\,\ddj g\\
&\leq \|\div v\|_{L^\infty}\|u_j\|_{L^2}^2+ \|u_j\|_{L^2}(\|[\ddj, v] \nabla u)\|_{L^2}+\|\ddj g\|_{L^2}).
\end{align*}
By Proposition \ref{Commutator},
\begin{align*}
\|[\ddj, v]\nabla u)\|_{L^2}\leq C 2^{-r_2j}q_j\|\nabla v\|_{B^{\frac{d}{2}}}\|  u\|_{B^{r_2}}.
\end{align*}

Using that we have $\|\nabla u_j\|_{L^2}^2\geq (\frac{5}{6}2^j\|u_j\|_{L^2})^2\geq \frac{1}{2}2^{2j}\|u_j\|_{L^2}2$, we thus get
\begin{align*}
&\quad \|u_j(t)\|_{L^2}+\frac{\mu}{2}2^{2j}\int_0^t \|  u_j(\tau)\|_{L^2}\, +\eta\int_0^t\|u_j(\tau)\|_{L^2}\,\\
&\leq\|u_j(0)\|_{L^2}+  C2^{-r_2j}q_j\int_0^t\|\nabla  v\|_{B^{\frac{d}{2}}}\|u\|_{B^{r_2}}\,+ \int_0^t \|\ddj g\|_{L^2}\,.
\end{align*}
Multiplying the factor $2^{r_2j}$ on both sides, summing up over $j\geq -1$, we further get
\begin{align}
&\quad \|u(t)\|^h_{B^{r_2}}+\frac{\mu}{2}\int_0^t \|  u(\tau)\|^h_{B^{r_2+2}}\, +\eta\int_0^t\|u(\tau)\|^h_{B^{r_2}}\,\notag\\
&\leq\|u_0\|^h_{B^{r_2}} +C \int_0^t\| \nabla v\|_{B^{\frac{d}{2}}}\|u\|_{B^{r_2}}\, + \int_0^t \| g\|^h_{B^{r_2}}\,.\label{Lame-es1}
\end{align}
Similarly, we have
\begin{align}
&\quad \|u(t)\|^{\ell}_{B^{r_1}}+\frac{\mu}{2}\int_0^t \|  u(\tau)\|^{\ell}_{B^{r_1+2}}\, +\eta\int_0^t\|u(\tau)\|^{\ell}_{B^{r_1}}\,\notag\\
&\leq\|u_0\|^{\ell}_{B^{r_1}} +C \int_0^t\| \nabla v\|_{B^{\frac{d}{2}}}\|u\|_{B^{r_1}}\, + \int_0^t \| g\|^{\ell}_{B^{r_1}}\,.\label{Lame-es1-ell}
\end{align}
Now, we focus only on the "compressible part" of $u$ that is $\Lambda^{-1} \div u$. Applying the localization of 0-th order pseudo-differential operator $\Lambda^{-1}\div $ to \eqref{Lamé} and find that 
\begin{align*}
\partial_t  \ddj(\Lambda^{-1}\div u)   -(2\mu+ \lambda) \Delta  (\Lambda^{-1}\div u_j)+\eta  (\Lambda^{-1}\div u_j)=-\ddj\Lambda^{-1}\div(v\cdot\nabla  u)+\ddj\Lambda^{-1}\div g.
\end{align*}
We notice that above equation is similar to \eqref{Lame-local-u-eq}, and following the derivation of  \eqref{Lame-es1} we have
\begin{align}
&\quad \|\Lambda^{-1}\div u(t)\|^h_{B^{{r_2}}}+\frac{(2\mu+\lambda)}{2}\int_0^t \| \Lambda^{-1}\div u(\tau)\|^h_{B^{{r_2}+2}}\, +\eta\int_0^t\|\Lambda^{-1}\div u(\tau)\|^h_{B^{r_2}}\,\notag\\
&\leq \|u_0\|^h_{B^{r_2}}+ \int_0^t(\|v\cdot \nabla u\|_{B^{{r_2}}}+ \| g\|^h_{B^s})\,\notag.
\end{align}
 Using product law  $\|v\cdot \nabla u\|_{B^{{r_2}}}\leq C\|v\|_{B^\frac{d}{2}}\|\nabla u\|_{B^{r_2}},$ for $r_2\in(-\frac{d}{2}, \frac{d}{2}]$ and notice that $\| \Lambda^{-1}\div u \|^h_{B^{r_2+2}}\geq \frac{5}{12}\|\div u \|^h_{B^{r_2+1}}.$ We further obtain that 
\begin{equation}\label{Lame-h-1}
    \frac{(2\mu+\lambda)}{5}\int_0^t \|  \div u(\tau)\|^h_{B^{r_2+1}}\,\leq \|u_0\|_{B^{r_2}}^h+C\int_0^t(\|v\|_{B^\frac{d}{2}}\|\nabla u\|_{B^{r_2}}+ \| g\|^h_{B^{r_2}})\,.
\end{equation}
Similarly, we have
$$\frac{(2\mu+\lambda)}{5}\int_0^t \|  \div u(\tau)\|^{\ell}_{B^{r_1+1}}\,\leq \|u_0\|_{B^{r_1}}^\ell+C\int_0^t(\|v\|_{B^\frac{d}{2}}\|\nabla u\|_{B^{r_1}}+ \| g\|_{B^{r_1}}^\ell)\,.$$
 This combined with \eqref{Lame-es1}, \eqref{Lame-es1-ell} and \eqref{Lame-h-1} completes the proof.
\end{proof}

\begin{lemma}\label{SimpliCarre}
Let  $X : [0,T]\to \mathbb{R}_+$ be a continuous function such that $X^2$ is differentiable. We assume that there exists 
 a constant $B\geq 0$ and  a measurable function $A : [0,T]\to \mathbb{R}_+$ 
such that 
 $$\frac{1}{p}\frac{d}{dt}X^2+BX^2\leq AX\quad\hbox{a.e.  on }\ [0,T].$$ 
 Then, for all $t\in[0,T],$ we have
$$X(t)+B\int_0^tX\leq X_0+\int_0^tA.$$
\end{lemma}
\begin{proof}
We set $X_\varepsilon\triangleq(X^2+\varepsilon^2)^{1/2}$ for $\varepsilon>0,$ and observe that
$$\frac{1}{2}\frac{d}{dt}X_\varepsilon^2+BX_\varepsilon^2\leq AX_\varepsilon+B\varepsilon^2. $$
Dividing both sides by  the positive function $X_\varepsilon$ yields
$$\frac{d}{dt}X_\varepsilon+BX_\varepsilon\leq A+B\frac{\varepsilon^2}{X_\varepsilon}\cdotp$$
Then,  integrating in time, using the fact that  $\varepsilon^2/X_\varepsilon\leq \varepsilon$, and taking the limit as $\varepsilon$ tends to $0$ 
completes the proof.
\end{proof}

\bibliographystyle{plain}

{\footnotesize
\bibliography{Biblio}}

\begin{thebibliography}{10}

\bibitem{Baer_Nunziato_1986}
M.R. Baer and J.W. Nunziato.
\newblock A two-phase mixture theory for the deflagration-to-detonation
  transition ({DDT}) in reactive granular materials.
\newblock {\em International journal of multiphase flow}, 12(6):861--889, 1986.

\bibitem{HJR}
H.~Bahouri, J.-Y. Chemin, and R.~Danchin.
\newblock {\em Fourier Analysis and Nonlinear Partial Differential Equations}.
\newblock Springer-Verlag, Berlin/Heidelberg, 2011.

\bibitem{BZ}
K.~Beauchard and E.~Zuazua.
\newblock Large time asymptotics for partially dissipative hyperbolic systems.
\newblock {\em Arch. Rational Mech. Anal}, 199, 177–227, 2011.

\bibitem{Benzoni-Serre}
S.~Benzoni-Gavage and D.~Serre.
\newblock {\em Multi-dimensional Hyperbolic Partial Differential Equations :
  First-order Sytems and Applications}.
\newblock Oxford Science Publications, New-York, 2007.

\bibitem{BrBurLa}
D.~Bresch, C.~Burtea, and F.~Lagouti{\`e}re.
\newblock Physical relaxation terms for compressible two-phase systems.
\newblock {\em arXiv preprint arXiv:2012.06497}, 2020.

\bibitem{BrDeGhGrHi}
D.~Bresch, B.~Desjardins, J-M. Ghidaglia, E.~Grenier, and M.~Hilliairet.
\newblock Multifluid models including compressible fluids.
\newblock {\em Handbook of Mathematical Analysis in Mechanics of Viscous
  Fluids}, page~52, 2018.

\bibitem{BrHi2}
D.~Bresch and M.~Hillairet.
\newblock Note on the derivation of multi-component flow systems.
\newblock {\em Proceedings of the American Mathematical Society},
  143(8):3429--3443, 2015.

\bibitem{BrHi1}
D.~Bresch and M.~Hillairet.
\newblock A compressible multifluid system with new physical relaxation terms.
\newblock {\em Annales ENS}, 52(1):255--295, 2019.

\bibitem{BrHu}
D.~Bresch and X.~Huang.
\newblock A multi-fluid compressible system as the limit of weak solutions of
  the isentropic compressible {N}avier- {S}tokes equations.
\newblock {\em Archive for rational mechanics and analysis}, 201(2):647--680,
  2011.

\bibitem{BreMucZat2019}
D.~Bresch, P.B. Mucha, and E.~Zatorska.
\newblock Finite-energy solutions for compressible two-fluid \textsc{S}tokes
  system.
\newblock {\em Archive for Rational Mechanics and Analysis}, 232(2):987--1029,
  2019.

\bibitem{BurGavPer}
C.~Burtea, S.~Gavrilyuk, and C.~Perrin.
\newblock Hamilton's principle of stationary action in multiphase flow
  modeling.
\newblock {\em hal-03146159}, 2021.

\bibitem{chen1994}
G-Q. Chen, C.D. Levermore, and T-P. Liu.
\newblock Hyperbolic conservation laws with stiff relaxation terms and entropy.
\newblock {\em Communications on Pure and Applied Mathematics}, 47(6):787--830,
  1994.

\bibitem{CBD2}
T.~Crin-Barat and R.~Danchin.
\newblock Partially dissipative hyperbolic systems in the critical regularity
  setting : the multi-dimensional case.
\newblock {\em arXiv:2105.08333}, 2021.

\bibitem{CBD1}
T.~Crin-Barat and R.~Danchin.
\newblock Partially dissipative one-dimensional hyperbolic systems in the
  critical regularity setting, and applications.
\newblock {\em arXiv:2101.05491}, 2021.

\bibitem{NSCL2}
R.~Danchin.
\newblock Global existence in critical spaces for compressible
  {N}avier-{S}tokes equations.
\newblock {\em Inventiones Mathematicae}, 141, 579-614, 2000.

\bibitem{Handbook}
R.~Danchin.
\newblock Fourier analysis methods for the compressible {N}avier-{S}tokes
  equations.
\newblock {\em in : Giga Y., Novotný A. (eds) Handbook of Mathematical
  Analysis in Mechanics of Viscous Fluids. Springer, Cham}, 2018.

\bibitem{CoursDanchinCharve}
R.~Danchin.
\newblock Master 2 lesson: Fourier analysis methods for models of
  nonhomogeneous fluids.
\newblock {\em
  https://perso.math.u-pem.fr/danchin.raphael/cours/coursM2-20.pdf}, 2020.

\bibitem{DkePerSchVau2021}
T.~Debiec, B.~Perthame, M.~Schmidtchen, and N.Vauchelet.
\newblock Incompressible limit for a two-species model with coupling through
  brinkman's law in any dimension.
\newblock {\em Journal de Math{\'e}matiques Pures et Appliqu{\'e}es},
  145:204--239, 2021.

\bibitem{Evans}
L.~C. Evans.
\newblock {\em Partial differential equations}.
\newblock Graduate studies in Mathematics, Vol. 19, AMS, 1963.

\bibitem{XuFang}
D.~Fang and J.~Xu.
\newblock Existence and asymptotic behavior of $\mathcal{C}^1$ solutions to the
  multidimensional compressible euler equations with damping.
\newblock {\em Nonlinear Analysis: Theory, Methods and Applications}, Volume
  70, Issue 1:244--261, 2009.

\bibitem{Gavrilyuk_2011}
S.~Gavrilyuk.
\newblock Multiphase flow modeling via {H}amilton’s principle.
\newblock In {\em Variational models and methods in solid and fluid mechanics},
  pages 163--210. Springer, 2011.

\bibitem{Gio4}
V.~Giovangigli and L.~Matuszewski.
\newblock Mathematical modeling of supercritical multicomponent reactive
  fluids.
\newblock {\em Mathematical Models and Methods in Applied Sciences},
  23(12):2193--2251, 2013.

\bibitem{Gio3}
V.~Giovangigli and L.~Matuszewski.
\newblock Structure of entropies in dissipative multicomponent fluids.
\newblock {\em Kinetic \& Related Models}, 6(2):373, 2013.

\bibitem{Gio1}
V.~Giovangigli and W-A. Yong.
\newblock Volume viscosity and internal energy relaxation : Symmetrization and
  chapman-enskog expansion.
\newblock {\em Kin. Rel. Models,}, 8, 79–116, 2014.

\bibitem{Gio2}
V.~Giovangigli and W-A. Yong.
\newblock Volume viscosity and internal energy relaxation: Error estimates.
\newblock {\em Nonlinear Analysis-real World Applications}, 8, 79–116, 2018.

\bibitem{Herard_2005}
J.-M. H{\'e}rard and O.~Hurisse.
\newblock A simple method to compute standard two-fluid models.
\newblock {\em International Journal of Computational Fluid Dynamics},
  19(7):475--482, 2005.

\bibitem{Ishii}
M.~Ishii and T.~Hibiki.
\newblock {\em Thermo-fluid dynamics of two-phase flow}.
\newblock Springer Science \& Business Media, 2010.

\bibitem{kapila}
A.K. Kapila, R.~Menikoff, J.B. Bdzil, S.F. Son, and D.S. Stewart.
\newblock Two-phase modeling of deflagration-to-detonation transition in
  granular materials: Reduced equations.
\newblock {\em Physics of fluids}, 13(10):3002--3024, 2001.

\bibitem{LC}
C.~Lin and J.-F. Coulombel.
\newblock The strong relaxation limit of the multidimensional euler equations.
\newblock {\em Nonlinear Diff. Eq. and Applications}, 20:447--461, 2013.

\bibitem{Nov2020}
A.~Novotn{\`y}.
\newblock Weak solutions for a bi-fluid model for a mixture of two compressible
  non interacting fluids.
\newblock {\em Science China Mathematics}, 63(12):2399--2414, 2020.

\bibitem{NovPok2020}
A.~Novotn{\`y} and M.~Pokorn{\`y}.
\newblock Weak solutions for some compressible multicomponent fluid models.
\newblock {\em Archive for Rational Mechanics and Analysis}, 235(1):355--403,
  2020.

\bibitem{GwiPerSwi2019}
P.Gwiazda, B.Perthame, and A.{\'S}wierczewska-Gwiazda.
\newblock A two-species hyperbolic--parabolic model of tissue growth.
\newblock {\em Communications in Partial Differential Equations},
  44(12):1605--1618, 2019.

\bibitem{QuWangNoSk}
P.~Qu and Y.~Wang.
\newblock Global classical solutions to partially dissipative hyperbolic
  systems violating the kawashima condition.
\newblock {\em Journal de Mathématiques Pures et Appliquées}, 109:93--146,
  2018.

\bibitem{RunstSickel}
T.~Runst and W.~Sickel.
\newblock {\em Sobolev spaces of fractional order, Nemytskij operators, and
  nonlinear partial differential equations, Nonlinear Analysis and
  Applications}.
\newblock Walter de Gruyter \& Co., Berlin, 1996.

\bibitem{gav2003JFM}
R.~Saurel, S.L. Gavrilyuk, and F.~Renaud.
\newblock A multiphase model with internal degrees of freedom: application to
  shock-bubble interaction.
\newblock {\em Journal of Fluid Mechanics}, 495:283, 2003.

\bibitem{SK}
S.~Shizuta and S.~Kawashima.
\newblock Systems of equations of hyperbolic-parabolic type with applications
  to the discrete {B}oltzmann equation.
\newblock {\em Hokkaido Math. J.}, 14, 249-275, 1985.

\bibitem{VasWenYu2019}
A.~Vasseur, H.~Wen, and C.~Yu.
\newblock Global weak solution to the viscous two-fluid model with finite
  energy.
\newblock {\em Journal de Math{\'e}matiques Pures et Appliqu{\'e}es},
  125:247--282, 2019.

\bibitem{XK1E}
J.~Xu and S.~Kawashima.
\newblock Diffusive relaxation limit of classical solutions to the damped
  compressible euler equations.
\newblock {\em Journal of Differential Equations}, 256, 771-796, 2014.

\bibitem{Yong}
W.-A. Yong.
\newblock Entropy and global existence for hyperbolic balance laws.
\newblock {\em Arch. Rational Mech. Anal}, 172, 47–266, 2004.

\bibitem{SlideZuazua}
E.~Zuazua.
\newblock Decay of partially dissipative hyperbolic systems.
\newblock {\em Slides available at
  https://caa-avh.nat.fau.eu/enrique-zuazua-presentations/}, 2020.

\end{thebibliography}

\end{document}